\documentclass[11pt, a4paper, reqno]{amsart}

\usepackage[utf8]{inputenc}
\usepackage[T1]{fontenc}
\usepackage[square,sort,comma,numbers]{natbib}
\usepackage{amsmath, mathtools, amsthm}
\usepackage{amssymb,url,xspace}
\usepackage{mathrsfs}
\usepackage{graphicx}
\usepackage{bm}
\usepackage[dvipsnames]{xcolor}
\usepackage{cases}
\usepackage[hidelinks]{hyperref}
\usepackage{cleveref}
\usepackage{tikz}
\usepackage{verbatim}
\usetikzlibrary{matrix}
\usepackage{xfrac}
\usepackage{csquotes}
\usepackage[normalem]{ulem}

\usepackage[outercaption]{sidecap} 

\usepackage{subfigure}
\usepackage{enumitem}
\usepackage{multirow}
\usepackage[ruled,vlined]{algorithm2e}
\usetikzlibrary{positioning,decorations.pathreplacing}
\usepackage[square]{natbib} 

\definecolor{dukeblue}{rgb}{0.0, 0.0, 0.61}

\hypersetup{
	colorlinks,
	citecolor=dukeblue,
	filecolor=black,
	linkcolor=dukeblue,
	urlcolor=black
}
\usepackage{empheq}

\usepackage[svgnames,pdf]{pstricks}
\usepackage{anysize}
\marginsize{3.25cm}{3.25cm}{4cm}{3.5cm}

\newcommand{\diff}{\, \mathrm{d}}
\newcommand{\del}{\partial}

\newcommand{\Z}{\mathbb{Z}}
\newcommand{\norm}[1]{\left\Vert#1\right\Vert}

\newcommand{\R}{\mathbb{R}}
\newcommand{\T}{\mathbb{T}}
\newcommand{\rd}{{\rm d}}

\def\*#1{\mathbf{#1}}

\theoremstyle{plain}
\numberwithin{equation}{section}
\newtheorem{remark}{Remark}

\newtheorem{proposition}{Proposition}[section]
\newtheorem{theorem}{Theorem}[section]
\newtheorem{lemma}{Lemma}[section]
\newtheorem{corollary}{Corollary}[section]

\title 
	[Control of the Stefan problem in a periodic box]{Control of the Stefan problem in a periodic box}
	
	\author{Borjan Geshkovski}	
	\thanks{A major part of this work was done while B.G. was affiliated with Departamento de Matemáticas, Universidad Autónoma de Madrid, and the Chair of Computational Mathematics, Fundación Deusto.}
	\author{Debayan Maity}
	\thanks{}
		
\date{\today}

\begin{document}

		\begin{abstract}
		{\color{black}
In this paper we consider the one-phase Stefan problem with surface tension, set in a two-dimensional strip-like geometry, with periodic boundary conditions respect to the horizontal direction $x_1\in\T$.  
We prove that the system is locally null-controllable in any positive time, by means of a control supported within an arbitrary open and non-empty subset. 
We proceed by a linear test and duality, but quickly find that the linearized system is not symmetric and the adjoint has a dynamic coupling between the two states through the (fixed) boundary. Hence, motivated by a Fourier decomposition with respect to $x_1$, we consider a family of one-dimensional systems and prove observability results which are uniform with respect to the Fourier frequency parameter. 
		The latter results are also novel, as we compute the full spectrum of the underlying operator for the non-zero Fourier modes. The zeroth mode system, on the other hand, is seen as a controllability problem for the linear heat equation with a finite-dimensional constraint. 
		The complete observability of the adjoint is derived by using a Lebeau-Robbiano strategy, and the local controllability of the nonlinear system is then shown by combining an adaptation of the source term method introduced in \cite{tucsnak_burgers} and a Banach fixed point argument.
		Numerical experiments motivate several challenging open problems, foraying even beyond the specific setting we deal with herein.}
\end{abstract}
			
	\maketitle	
	
	\setcounter{tocdepth}{1}
	
	\tableofcontents
	

	{\small {\href{https://mathscinet.ams.org/msc/msc2010.html}{{\bf 	\color{black}{AMS Subject Classification}}}}. 93B05, 35R35, 35Q35, 93C20.}

\section{Introduction and main results}
	
		The Stefan problem is the quintessential macroscopic model of phase transitions in liquid-solid systems. The physical setup thereof typically consists in considering a domain $\Omega\subset\mathbb{R}^d$, which is occupied by water (the liquid phase), a part of whose boundary is some interface $\Gamma$, describing contact with a deformable solid such as ice (the solid phase). 
		Due to melting or freezing, the regions occupied by water and ice will change over time and, consequently, the interface $\Gamma$ will also change its position and shape. 
		This leads to a free boundary problem.
		Albeit classical (see \cite{pruss2016moving, figalli2019regularity} for an overview of the mathematical literature), the Stefan problem continues to be of use in many contemporary applications, such as additive manufacturing of alloys (\cite{koga2021state}), ice modeling for video rendering in computer graphics (\cite{kim2006modeling}), and, reaching even beyond its original fluid-mechanical nature, in the context of mathematical biology, for modeling the spread of various infectious diseases (\cite{lin2017spatial, du2015spreading}).
		
		\subsection{Setup} We shall focus on the strong formulation of the \emph{one-phase} Stefan problem (i.e., where the temperature of the ice is a known constant), with surface tension effects, following \cite{escher2003analytic, hadvzic2012orthogonality, pruss2013qualitative, HS17}. 
		 We shall focus on the problem in two spatial dimensions ($d=2$).	 
		To describe the geometrical setup, let 
		$
		\T:=\R/(2\pi\Z)
		$
		denote the one-dimensional flat torus, which we identify with $[0,2\pi]$. 
	Set
	\begin{equation} \label{eq: def.omega}
	\Omega := \T \times (-1,1).
	\end{equation}
	The domain $\Omega$ will serve as the reference configuration. In the one-phase Stefan problem, a heat-conducting liquid fills a time-varying domain $\Omega(t) \subset \mathbb{R}^{2}$ for $t\geqslant0$. We will suppose that the boundary $\del\Omega(t)$ of the liquid consists of two components: an unknown, time-dependent component (the free boundary $\Gamma(t)$), and a fixed and static component. More specifically, for any $t\geqslant0$, $\Omega(t)$ is assumed to have a flat, rigid bottom, while the free boundary will be parametrized by an unknown function $h(t,z_1)$, representing the former's displacement away from the reference boundary $\T\times\{z_2=1\}$ (see \Cref{fig1}), and thus described by the equation $z_2 =1+h(t,z_1)$. 
	
	In other words,
	\begin{equation*}
	\Omega(t) :=  \left\{ z=(z_{1}, z_{2}) \in \T \times \R\, \Bigm|\, -1 < z_{2} < 1+h(t,z_{1}) \right\},
	\end{equation*}
	where $h=h(t,z_{1})$ is the unknown height function, while the free boundary is given by
	\begin{equation*}
	\Gamma(t) := \left\{ z=(z_{1}, z_{2}) \in \T \times \R\, \Bigm|\,  z_{2} = 1+h(t,z_{1}) \right\}.
	\end{equation*}
	Given a time horizon $T>0$ the strong formulation of the one-phase Stefan problem takes the form\footnote{We make use of the standard notation
	$(0, T)\times\Omega(t):=\bigcup_{0\leqslant t\leqslant T} \{t \} \times \Omega(t)$ and the analog for $(0, T) \times \Gamma(t)$.}
	\begin{equation} \label{eq: stefan.surface.tension}
	\boxed{
	\begin{cases}
	\del_t \varrho-\Delta \varrho = u1_{\omega} &\text{ in }(0, T)\times \Omega(t),\\
	\del_t h +\sqrt{1+|\del_{z_1}h|^2} \, \nabla \varrho_{|_{\Gamma(t)}}\cdot{\bf{n}}=0 &\text{ on }(0, T)\times \T,\\
	\varrho(t,z_1,-1)=0&\text{ on }  (0, T)\times\T\\ 
	\varrho(t,z_1,z_2)= -\sigma \kappa(h(t,z_1)) &\text{ on }(0, T)\times\Gamma(t),\\
	(\varrho, h)_{\mid_{t=0}} = \left(\varrho^0, h^0\right) &\text{ in } \Omega(0) \times \T.
	\end{cases}
	}
	\end{equation}
	This is a coupled system, where the unknown state is the pair $(\varrho,h)$, and $u$ is the control, actuating within $\omega\subset\Omega$. 
	The initial domain $\Omega(0)$ is given by
	\begin{equation*}
	\Omega(0) := \left\{ z = (z_1, z_2) \in \T \times \R\, \Bigm|\, -1 < z_2 < 1+ h_0(z_1)\right\}.
	\end{equation*}
	The constant $\sigma>0$ represents the surface tension coefficient\footnote{The condition involving the surface tension and the mean curvature is referred to as the \emph{Gibbs-Thomson correction}. The physical reason for introducing the Gibbs-Thomson correction stems from the need to account for possible coarsening or nucleation effects (\cite{perez2005gibbs}). When $\sigma=0$, we are dealing with the \emph{classical Stefan problem}. The mesoscopic limit $\sigma\searrow0$ has been addressed in \cite{HS17} (without control).}, whereas $\kappa(h(t,z_1))$ denotes the mean curvature of the free boundary $\Gamma(t)$, defined as
	\begin{equation*}
	\kappa(h):=\frac{\del_{z_1}^2 h}{\left(1+\left|\del_{z_1}h\right|^2\right)^{\sfrac{3}{2}}}.
	\end{equation*}	
	As seen later on, the assumption $\sigma>0$ will be a critical part of our study.
	Finally, ${\bf{n}}={\bf{n}}(t, z_1)$ denotes the unit normal to $\Gamma(t)$ outward $\Omega(t)$ and is given by
	\begin{equation*} \label{eq: normal.fb}
	{\bf{n}} := \displaystyle \frac{1}{\sqrt{1 +|\del_{z_1} h|^{2}}}
	\begin{bmatrix}
	-\del_{z_1} h \\
	1
	\end{bmatrix}.
	\end{equation*}
	
		\begin{figure}[h]
	\centering
		\includegraphics[scale=0.85]{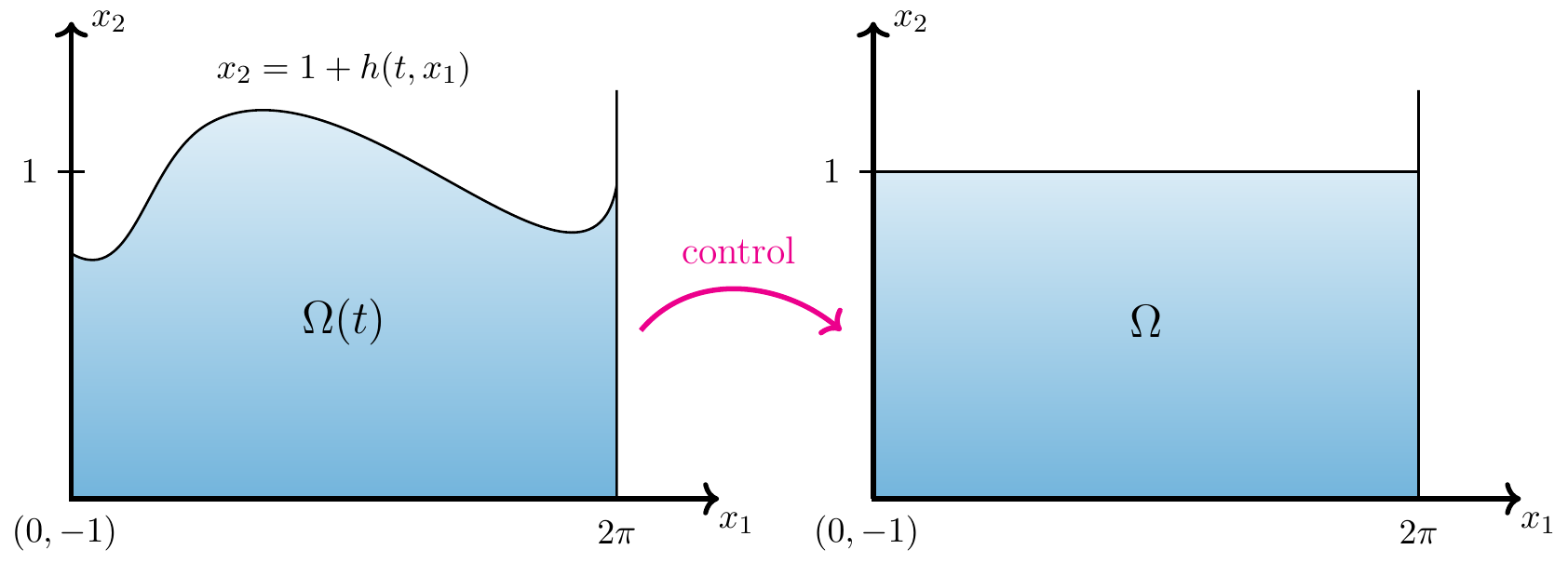}
		\caption{(\emph{Left}) The moving domain $\Omega(t)$, and the free boundary $\Gamma(t)$, parametrized by the height function $h(t,z_1)$. Our goal is to steer both the temperature $\varrho(t)$, and the free boundary $\Gamma(t)$, to rest in time $T>0$ (\emph{right}).
		\label{fig1}}
	\end{figure}

	\subsection{Main results} 
	In view of the applications presented just before, analyzing the controlled evolution of trajectories to \eqref{eq: stefan.surface.tension} is rather natural. 
	Our interest is the problem of \emph{null-controllability} for \eqref{eq: stefan.surface.tension}: given a time horizon $T>0$, we seek to steer the temperature $\varrho(t)$, as well as the interface height $h(t)$, to the equilibrium position $(0,0)$, by means of some control $u(t,z_1,z_2)1_\omega$ actuating   within\footnote{We shall focus on distributed controls since, as it is well known for parabolic equations, a simple extension-restriction argument allows one to obtain results for boundary controls (actuating along the bottom boundary $\T\times\{z_2=-1\}$, replacing the Dirichlet boundary condition in \eqref{eq: stefan.surface.tension}) as well.} the fluid domain $\Omega(t)$. 
	Note that, due to the presence of the free boundary parametrized by $h$, such a controllability problem would amount to also controlling the domain $\Omega(t)$ to the reference configuration $\Omega$. This is seen in \Cref{fig1}. 
	
	{\color{black}
	Our main result yields a positive answer to the null-controllability problem for \eqref{eq: stefan.surface.tension}, under smallness assumptions for the initial data.
	
	\begin{theorem} \label{thm:main-nonlinear}
Suppose $T>0$ and $\sigma>0$ are fixed, and suppose that $\omega\subset\Omega$ is open and non-empty. 
There exists some small enough constant $r>0$ such that for any initial data $\left(\varrho^{0}, h^0\right) \in H^{1}(\Omega(0))\times H^{\sfrac52}(\mathbb{T})$ satisfying 
\begin{align*}
h_{0}(x_1) > -1  \quad &\text{ for }  x_1\in\mathbb{T}, \notag\\
\varrho^{0}(x_1, -1) = 0  \quad &\mbox{ for } x_1\in\mathbb{T}, 
\\ 
\varrho^{0}\left(x_1, 1 + h^{0}(x_1)\right) + \sigma \kappa\left(h^{0}(x_1)\right) = 0 \quad &\mbox{ for } x_1\in\mathbb{T}, \notag
\end{align*}
and 
\begin{equation*} 
\left\|\varrho^{0}\right\|_{H^{1}(\Omega(0))} + \norm{h^{0}}_{H^{\sfrac52}(\mathbb{T})}\leqslant r, 
\end{equation*}
there exists some control $u \in L^{2}((0, T)\times\omega)$ such that the unique solution $(y,h)$ to \eqref{eq: stefan.surface.tension}, where\footnote{Solutions are actually more regular -- see \Cref{thm:main-nonlinear-2}.} 
\begin{equation*}
h \in  C^0\left([0,T];H^{\sfrac{5}{2}}(\T)\right)\cap L^{2}\left((0,T);H^{\sfrac{7}{2}}(\T)\right)
\end{equation*}
and
\begin{equation*}
\varrho \in C^0\Big([0, T]; H^{1}(\Omega(\cdot))\Big)\cap L^{2}\Big((0, T); H^{2}(\Omega(\cdot))\Big),
\end{equation*}
satisfies $\varrho(T, \cdot) \equiv 0 \mbox{ in } \Omega(T)=\Omega  \mbox{ and  } h(T, \cdot)\equiv0  \mbox{ on } \mathbb{T}.$
\end{theorem}
}
	
	As a first but necessary step in solving this problem, we shall focus on the system obtained by linearizing around the equilibrium $(0,0)$. As discussed in \Cref{sec: fixing.domain}, this linearized system will take the form
\begin{equation} \label{sys-lin-con.intro}
\boxed{
	\begin{cases}
	\del_t y - \Delta y = u1_\omega &\text{ in }(0, T)\times \Omega,\\
	\del_t h(t, x_1) - \del_{x_2} y(t, x_1, 1)=0 &\text{ on }(0, T)\times\T,\\
	y(t, x_1, -1) = 0 &\text{ on }(0, T)\times \T,\\
	y(t, x_1, 1) = \sigma \del_{x_1}^2 h(t, x_1) &\text{ on }(0, T)\times \T,\\
	(y, h)_{\mid_{t=0}} = \left(y^0, h^0\right) &\text{ in }\Omega \times \T.
	\end{cases}
	}
	\end{equation}
	The controllability properties of \eqref{sys-lin-con.intro} have not been addressed in the literature, to the best of our knowledge. 
	In fact, even the well-posedness may appear tricky at a first glance, due to the peculiar coupling of the states $y$ and $h$. 
	Nonetheless, we may note the energy dissipation law (when $u\equiv0$)
\begin{equation} \label{eq: energy.dissipation}
\frac{\diff}{\diff t}\left\{\int_\Omega|y(t)|^2\diff x + \sigma\int_{\T}\big|\del_{x_1}h(t)\big|^2\diff x_1\right\} = -2\int_\Omega\big|\nabla y(t)\big|^2\diff x,
\end{equation} 
	which will aid in ensuring that the system is well-posed when considered on the energy space $L^2(\Omega)\times H^1(\T)$. As a matter of fact, we show that the governing operator of \eqref{sys-lin-con.intro} generates an analytic semigroup on this energy space (\Cref{prop: analytic.semigroup} \& \Cref{cor: linear.well.posed}).
	Identity \eqref{eq: energy.dissipation} also clearly illustrates the strength of the coupling between $y$ and $h$, and the impact of $\sigma>0$.  

	We show the following result for \eqref{sys-lin-con.intro}.
	
	{\color{black}
	\begin{theorem} \label{thm: main.result.linear}
		Suppose $T>0$ and $\sigma>0$ are fixed, and suppose that $\omega\subset\Omega$ is open and non-empty. 
		Then for any $\left(y^{0}, h^{0}\right)\in L^2(\Omega)\times H^1(\T)$, there exists some control $u\in L^{2}((0,T) \times \omega)$ such that the unique solution $(y,h)\in C^0([0,T];L^{2}(\Omega)\times H^1(\T))$ to \eqref{sys-lin-con.intro} satisfies $y(T,\cdot)\equiv0$ in $\Omega$ and $h(T,\cdot)\equiv0$ on $\T$, {\color{black} and there exists some constant ${\mathcal{K}}(T,\sigma)>0$ such that 
		\begin{equation*}
		\norm{u}_{L^{2}((0,T)\times \omega)} \leqslant {\mathcal{K}}(T,\sigma) \norm{\left(y^{0}, h^{0}\right)}_{L^2(\Omega)\times H^1(\T)}.
		\end{equation*}
		Furthermore, $T\longmapsto{\mathcal{K}}(T,\sigma)$ is a decreasing function for fixed $\sigma>0$, and if $T\leqslant1$, there exist a constant $M>0$ independent of $T$ (but depending on $\sigma$) such that 
		\begin{equation*}
		{\mathcal{K}}(T,\sigma)=Me^{\frac{M}{T}}.
		\end{equation*}}
\end{theorem}
}
		
	\subsection{Strategy of proof} \label{sec: strategy.proof} 
	
	We shall now present the broad ideas of our proof.
	We postpone further comments regarding the results to \Cref{sec: discussion}.
	\smallskip
	
	{\color{black}
	\noindent
	\textit{\textbf{Part 1.} Duality, and Fourier decomposition of the adjoint.}
	We shall take a bottom-up approach to the proof of \Cref{thm:main-nonlinear}, and thus start with \Cref{thm: main.result.linear}.
	And when one looks to prove the controllability of \eqref{sys-lin-con.intro}, the primal instinct would be to first  write the adjoint system, which reads as 
	\begin{equation} \label{eq: adjoint.stefan.intro}
	\begin{cases}
	-\del_t\zeta-\Delta\zeta=0 &\text{ in }(0,T)\times\Omega,\\
	-\del_t\ell(t,x_1)-\sigma\del_{x_1}^2\del_{x_2}\zeta(t,x_1,1)=0 &\text{ on }(0,T)\times\T,\\
	\zeta(t,x_1,-1)=0&\text{ on }(0,T)\times\T,\\
	\zeta(t,x_1,1)=\ell(t,x_1)&\text{ on }(0,T)\times\T,\\
	(\zeta,\ell)_{|_{t=T}}=(\zeta_T,\ell_T) &\text{ in} \Omega.
	\end{cases}
	\end{equation}
	Note that, since the natural energy space for \eqref{sys-lin-con.intro} is $L^2(\Omega)\times H^1(\T)$, the adjoint problem \eqref{eq: adjoint.stefan.intro} ought to be analyzed in the dual space. Proceeding by the Hilbert Uniqueness Method (HUM, \cite{lions1988exact}), one would look to show an observability inequality of the form
	\begin{equation} \label{eq: obs.ineq.intro}
	\|\zeta(0)\|_{L^2(\Omega)}^2+\|\ell(0)\|_{H^{-1}(\T)}^2\leqslant \mathcal{K}(T,\sigma) \int_0^T\int_\omega |\zeta(t,x)|^2\diff t\diff x,
	\end{equation}
	for all $(\zeta_T,\ell_T)\in L^2(\Omega)\times H^{-1}(\T)$. The standard way to proceed in such an endeavor would be through the use of Carleman inequalities. At this stage, we are not aware of existing Carleman inequalities which come even close to being adapted to the specific nature of the adjoint problem \eqref{eq: adjoint.stefan.intro}, due to the asymmetric and dynamic nature of the coupling between both states.
	
	There is however a simpler way in which the problem of proving \eqref{eq: obs.ineq.intro} can be tackled.\footnote{Such procedures have been used in the literature, typically in the context of hypoelliptic operators (\cite{beauchard2009some, beauchard2014null, Beau14, beauchard20152d}).}  
	We may exploit the periodicity and decompose all functions appearing in \eqref{eq: adjoint.stefan.intro} into Fourier series with respect to $x_1\in\T$. Namely, we write
	\begin{equation} \label{eq: fourier.y}
	\zeta(t,x_1,x_2)=\frac{1}{\sqrt{2\pi}}\sum_{n\in\Z}\zeta_n(t,x_2)e^{inx_1}
	\end{equation}
	with 
	\begin{equation} \label{eq: yn.fourier}
	\zeta_n(t,x_2)=\frac{1}{\sqrt{2\pi}}\int_{\T}\zeta(t,\xi,x_2)e^{in \xi}\diff \xi,
	\end{equation}
	with analogous decompositions for $\ell$ and the data $(\zeta_T,\ell_T)$.  
	We then find ourselves with a family of systems for the Fourier coefficients, parametrized by $n\in\Z$:
	\begin{equation} \label{eq: fourier.projected.system.new}
		\begin{cases}
		-\del_t \zeta_n - \del_{x_2}^2 \zeta_n + n^2 \zeta_n = 0 &\text{ in }(0, T) \times (-1, 1), \\
		-\ell_n'(t)+\sigma n^2\del_{x_2} \zeta_n(t, 1)=0 &\text{ in }(0, T), \\
		\zeta_n(t, -1) = 0 &\text{ in } (0, T), \\
		\zeta_n(t, 1) = \ell_n(t) &\text{ in } (0, T), \\
		\left(\zeta_n, \ell_n\right)_{|_{t=T}} = \left(\zeta_{T,n}, \ell_{T,n}\right) &\text{ in } (-1, 1).
		\end{cases}
		\end{equation}
		We shall prove the following result for \eqref{eq: fourier.projected.system.new}.
		
		\begin{theorem} \label{thm: frequency.obs} 
		Suppose $T>0$ and $\sigma>0$ are fixed. Suppose $(c,d)\subset(-1,1)$. 
		There exists a constant $\mathcal{K}(T,\sigma)>0$ such that 
		\begin{equation} \label{eq: frequency.by.frequency.obs}
		\int_{-1}^1|\zeta_n(0)|^2\diff x_2 + \frac{|\ell_n(0)|^2}{1+\sigma n^2} \leqslant\mathcal{K}(T,\sigma)\int_0^T \int_c^d |\zeta_n(x_{2})|^2\diff x_2\diff t,
		\end{equation}
		holds for any $n\in\mathbb{Z}$, and for any pair $(\zeta_{T,n},\ell_{T,n})\in L^2(-1,1)\times\mathbb{R}$, where the pair $(\zeta_n,\ell_n)\in C^0([0,T];L^2(-1,1)\times\mathbb{R})$ denotes the unique solution to \eqref{eq: fourier.projected.system.new}.
		Furthermore, $T\longmapsto\mathcal{K}(T,\sigma)$ is a decreasing function for fixed $\sigma>0$, and if $c\neq\pm d$ and $T\leqslant1$, there exists a constant $M>0$, independent of $T$ and $n$, such that 
		\begin{equation} \label{eq: zero.mode.estimate.cost}
		\mathcal{K}(T,\sigma)=Me^{\frac{M}{T}}.
		\end{equation} 
		\end{theorem}
		
		We shall distinguish two parts in the proof of \Cref{thm: frequency.obs}.
		\smallskip
		
	\noindent
	\textit{\textbf{Part 2.} Observing frequency by frequency: the non-zero modes}. Suppose $n\neq0$. Then if we set 
	\begin{equation*}
	s_n(t):=-\frac{\ell_n(t)}{\sigma n^2}, 
	\end{equation*}
	we see that $(\zeta_n, s_n)$ satisfies
	\begin{equation} \label{eq: fourier.projected.system.intro}
		\begin{cases}
		-\del_t \zeta_n - \del_{x_2}^2 \zeta_n + n^2 \zeta_n = 0 &\text{ in }(0, T) \times (-1, 1), \\
		-s_n'(t)-\del_{x_2} \zeta_n(t, 1)=0 &\text{ in }(0, T), \\
		\zeta_n(t, -1) = 0 &\text{ in } (0, T), \\
		\zeta_n(t, 1) = -\sigma n^2s_n(t) &\text{ in } (0, T), \\
		\left(\zeta_n, s_n\right)_{|_{t=T}} = \left(\zeta_{T,n}, s_{T,n}\right) &\text{ in } (-1, 1).
		\end{cases}
		\end{equation}
		A curiosity with regard to \eqref{eq: fourier.projected.system.intro} is that the governing operator is self-adjoint when the metric for the second component is weighted by $\sigma n^2$. Then, by computing the spectrum of this operator (this  may be found in \Cref{lem: spectrum} -- note that this is a nontrivial computation, as the operator is not a linear shift of the Laplacian), for \eqref{eq: fourier.projected.system.intro} we can show that
		\begin{equation} \label{eq: frequency.by.frequency.obs.2}
		\int_{-1}^1|\zeta_n(0)|^2\diff x_2 + \sigma n^2|s_n(0)|^2\leqslant\mathcal{K}(T,\sigma)\int_0^T \int_c^d |\zeta_n(t)|^2\diff t\diff x_2
		\end{equation}
		holds for all $n\neq0$, leading to \eqref{eq: frequency.by.frequency.obs} by reverting from $s_n$ to $\ell_n$. 
		\smallskip
			
	\noindent
	\textit{\textbf{Part 3.} Observing the zeroth mode}. On another hand, when $n=0$, we see that \eqref{eq: fourier.projected.system.intro} becomes uncoupled, with a non-homogeneous boundary condition for $\zeta_0$. 
	In terms of the dual, control problem, the system is also uncoupled: we may control the heat component $y_0$ independently of $h_0$. The constraint $h_0(T)=0$ can be seen as a one-dimensional constraint on the control $u_0$ -- this can then be covered via a compactness-uniqueness argument, and is a rather classical procedure for one-dimensional free boundary problems (\cite{fdez-cara_doubova_burgers, geshkovski2021controllability, tucsnak_burgers}). Hence, these results reinterpreted for the adjoint system will yield \eqref{eq: frequency.by.frequency.obs} for $n=0$.
	The exponential form of $\mathcal{K}(T,\sigma)$ for small times is then derived by means of a moment method argument.
	\smallskip
	
	\noindent
	\textit{\textbf{Part 4.} A Lebeau-Robbiano argument and \Cref{thm: frequency.obs} yield \Cref{thm: main.result.linear}.} We begin by noting that it suffices to show \Cref{thm: main.result.linear} when $\omega=(a,b)\times(c,d)$ (a rectangle); indeed, if $\omega$ were not in such a form, we could always find $\omega_\circ\subset\omega$ which is such a rectangle, and apply \Cref{thm: main.result.linear} to $\omega_\circ$. 
	With this remark in hand, we may combine the observability inequality \eqref{eq: frequency.by.frequency.obs} and a classical inequality for the eigenfunctions of the Laplacian\footnote{Here, we consider the Laplacian on the torus $\T$, whose eigenfunctions are the complex exponentials which span $L^2(\T)$.} (due to Lebeau and Robbiano, \cite{lebeau1995controle}) to obtain \eqref{eq: obs.ineq.intro} (with $\omega$ a rectangle) for low-frequency solutions (namely, corresponding to $|n|\leqslant\mu$, for any $\mu>0$) to \eqref{eq: adjoint.stefan.intro}. To cover the high-frequency components and complete the proof, following the Lebeau-Robbiano argument, we exploit the exponentially stable character of the high-frequencies of \eqref{eq: adjoint.stefan.intro} -- when the zero-mode is removed from \eqref{eq: adjoint.stefan.intro}, the governing operator is actually self-adjoint and generates an exponentially decaying semigroup.
	\smallskip
	
	\noindent
	\textit{\textbf{Part 5.} Source-term method and Banach fixed-point yield \Cref{thm:main-nonlinear}.}
	We finally prove \Cref{thm:main-nonlinear} in a sequence of steps. 
	We first add source terms mimicking the nonlinearities (including along the boundary) in the setting of \Cref{thm: main.result.linear}, over which a fixed point argument will be developed. The controllability of the resulting non-homogeneous system is preserved by an adaptation of the source term method of \cite{tucsnak_burgers}, for which the special, exponential form of the control cost $\mathcal{K}(T,\sigma)$ for small times is crucial.
	We then present an appropriate change of variables which fixes the domain in \eqref{eq: stefan.surface.tension}, and show that the nonlinearities are of a quadratic nature, all of which leads us naturally to the framework of a Banach fixed point argument, which will lead to the desired conclusion.
	}
	
	{\color{black}
	\subsection{Outline} 
	
	In {\bf \Cref{sec: discussion}}, we provide an in-depth comparison of our results  with existing works, as well as a commentary on the potential limitations and extensions of our results to more general settings, {\color{black}such as arbitrary geometries, global nonlinear results, and the case $\sigma=0$ -- the latter is also corroborated by numerical experiments.} {\bf \Cref{sec-linear-sg}} presents the basic functional setting and well-posedness results used throughout. 
	 In {\bf \Cref{sec: projected.systems}} we provide the proof to \Cref{thm: frequency.obs}, while {\bf \Cref{sec: proof.main.result.linear.new}} contains the proof to \Cref{thm: main.result.linear}.
	{\color{black} In {\bf\Cref{sec: nonlinear.section}} we provide the proof to our main nonlinear result, namely \Cref{thm:main-nonlinear}.}
	 Finally, in {\bf \Cref{sec: conclusion}}, we conclude with a selection of related open problems. 
	}
	
	\subsection{Notation} 
	Whenever the dependence on parameters of a constant is not specified, we will make use of Vinogradov notation and write $f \lesssim_S g$ whenever a constant $C>0$, depending only on the set of parameters $S$, exists such that $f \leqslant Cg$.

	\section{Discussion} \label{sec: discussion} 
	
	\subsection{Previous work}
	The null-controllability results we prove in this work are among the first of their kind for multi-dimensional free-boundary problems in which the free boundary depends on the spatial variable -- even in the linearized regime.
		In this sense, our setup differs from existing works on the controllability of multi-dimensional fluid-structure interaction models with rigid bodies (\cite{Ta07, boulakia_jems}), and the controllability of one-dimensional free boundary problems (\cite{fdez-cara_doubova_burgers, tucsnak_burgers, tucsnak_burgers_jmpa, geshkovski2021controllability, ramaswamy2020remark}, \cite{barcena2022exact}, \cite{colle2022controllability} see also \cite{dunbar2003motion}), as therein, the free boundary is parametrized by the graph of a \emph{time-only} dependent function, modeling a rigid body.
		In particular, the spatial regularity of the height function $h$ plays a crucial role in the analysis (or even well-posedness) results. 
	
		A partial controllability result for the two-dimensional classical Stefan problem ($\sigma=0$) is shown in \cite{demarque2018local} -- only the temperature $\varrho$ is controlled to $0$ without any consideration of the height function $h(t,z_1)$ defining the free boundary $\Gamma(t)$. In fact, the geometrical setting is also different, as the free boundary $\Gamma(t)$ manifests as the entire boundary of the fluid domain $\Omega(t)$.
		Moreover, the Stefan law governing the velocity of the height function is regularized by adding a Laplacian term, which significantly simplifies the analysis. 
		
		Albeit for a system of different nature to ours, we also refer to \cite{Al18} (and \cite{alazard2017stabilization, Al18b, su2020stabilizability} for related results) for an exact-controllability result of the velocity and the free surface elevation of the water waves equations in two dimensions, by means of a single control actuating along an open subset of the free surface. 
		In the aforementioned works, the two-dimensional geometrical strip-like setting of the free boundary problem is the same as ours.  These results are extended to the three dimensional context in \cite{zhu2020control}. 
	
	\subsection{The (curious) case of $\sigma=0$}
	
	We were unsuccessful in applying our techniques to cover the case $\sigma=0$. In this case, the linearized system \eqref{sys-lin-con.intro} is uncoupled, and proceeding by writing the adjoint system directly might appear as an arid endeavor. 
	We provide more insight into some of the possible obstacles.
	\begin{itemize}
	\item
	Since \eqref{sys-lin-con.intro} is uncoupled, one can first control the heat equation for $y$ to $0$ through HUM, and then see the null-controllability for $h$ as a linear constraint of the form
	\begin{equation} \label{eq: h.constraint}
	0=h(T,x_1) = h^0(x_1) + \int_0^T \partial_{x_1} y(t,x_1,1)\diff t
	\end{equation}
	 for all $x_1\in\mathbb{T}$. 
	 The difference with respect to the one-dimensional case (see, e.g., \cite{Er19, geshkovski2021controllability}) is that \eqref{eq: h.constraint} is not a finite-dimensional constraint anymore, due to the fact that $h$ depends on the spatial variable. Hence, the compactness-uniqueness arguments of these works are not directly applicable.
	 \smallskip
	 \item
	In this spirit, one can rather proceed by Fourier series decompositions to derive \eqref{eq: fourier.projected.system.intro} with $\sigma=0$. Now, for any fixed frequency $n$, the compactness-uniqueness arguments of the above-cited works can be used to derive the null-controllability of the full system, since \eqref{eq: h.constraint} will transform into a one-dimensional constraint for the heat control. The caveat is that, due to the compactness-uniqueness argument used for addressing this finite-dimensional constraint,  the controllability cost will depend on $n$, with an explicit dependence on $n$ being uncertain. Consequently, we cannot paste the controls for all $n$ to derive the controllability of \eqref{sys-lin-con.intro} with $\sigma=0$. 
	\end{itemize}
	
	\begin{figure}[h!]
	\includegraphics[scale=0.3]{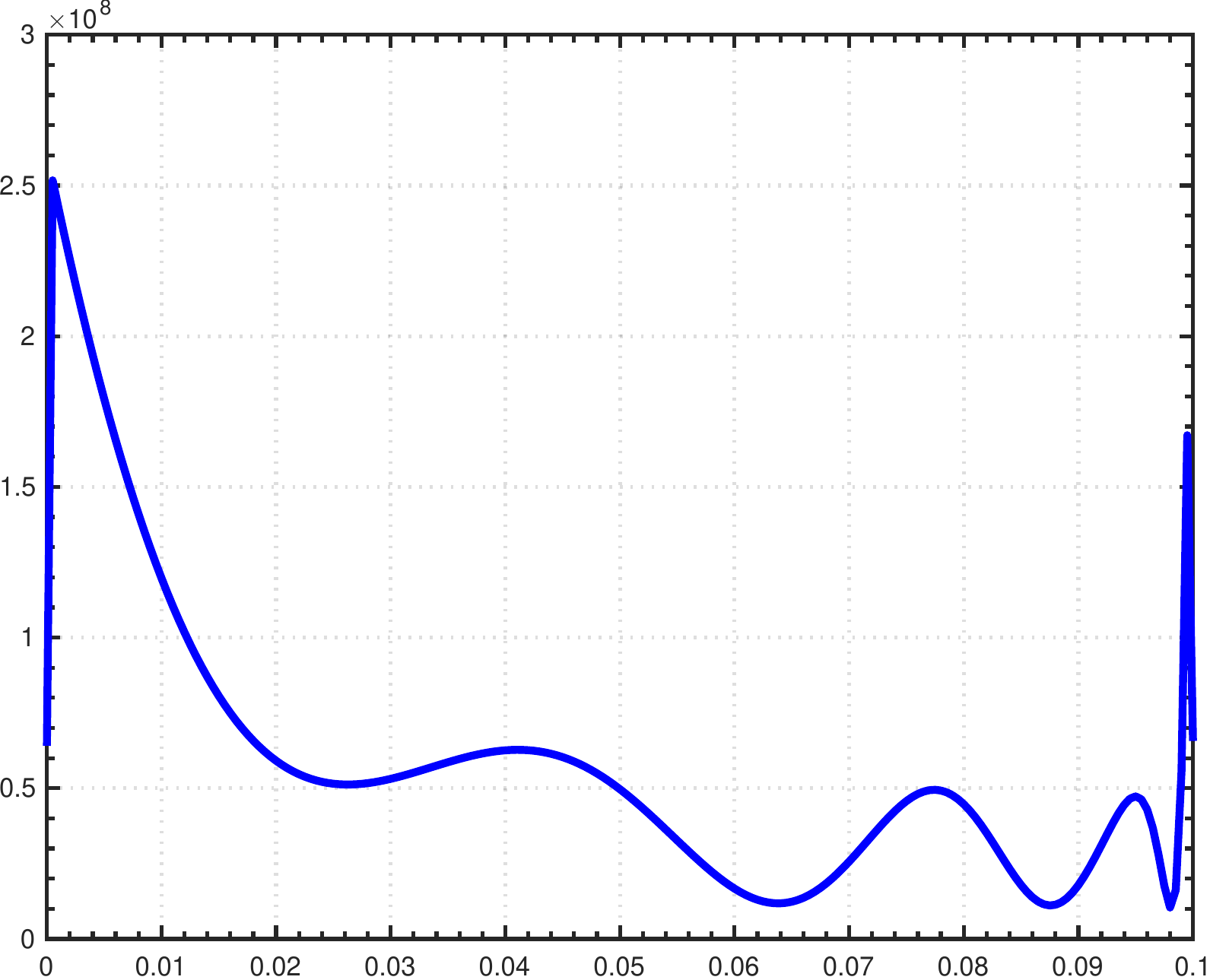}\\
	\vspace{0.2cm}
	\begin{tabular}{c}
	$t=0$ \\
	\includegraphics[scale=0.25]{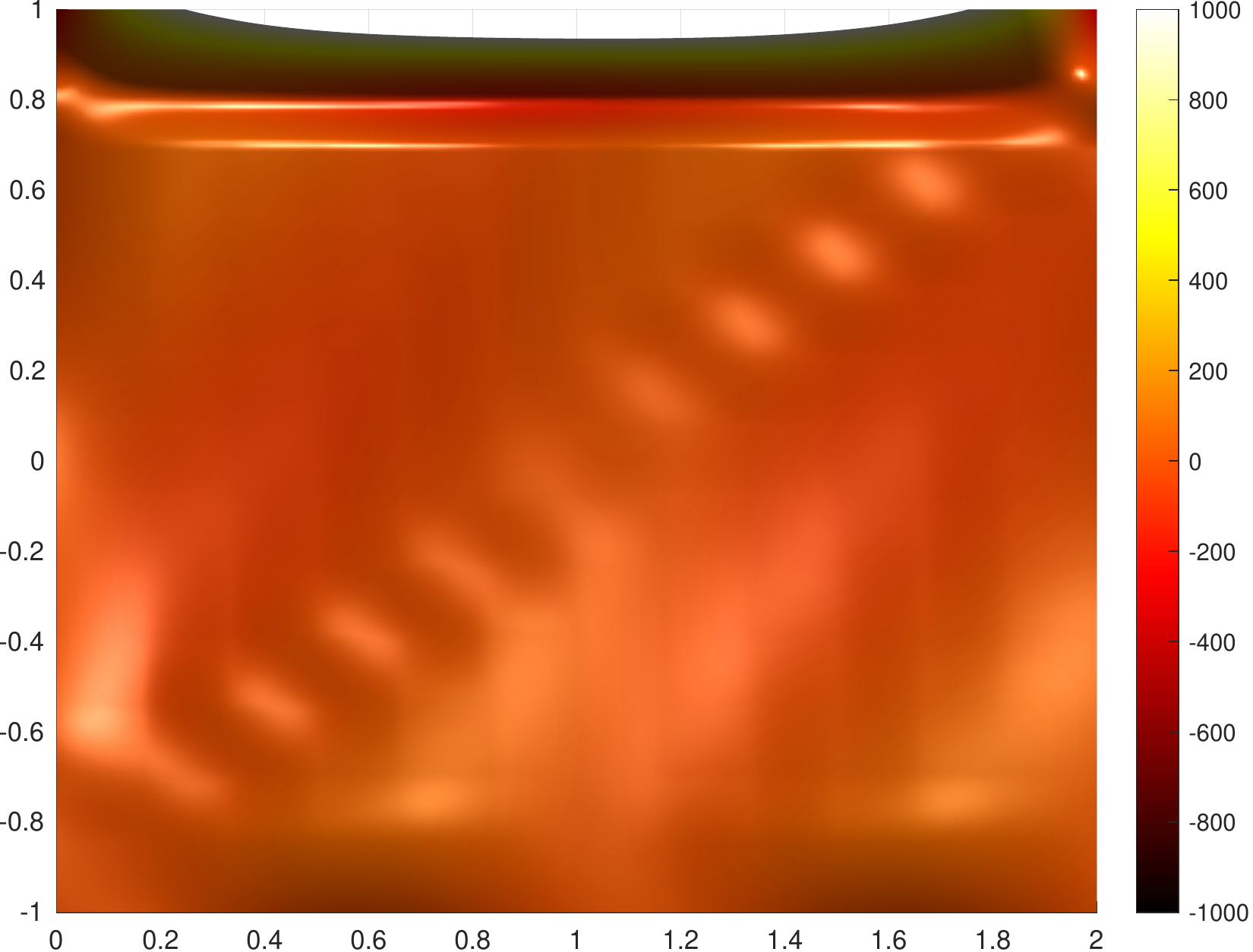}\\
	\includegraphics[scale=0.25]{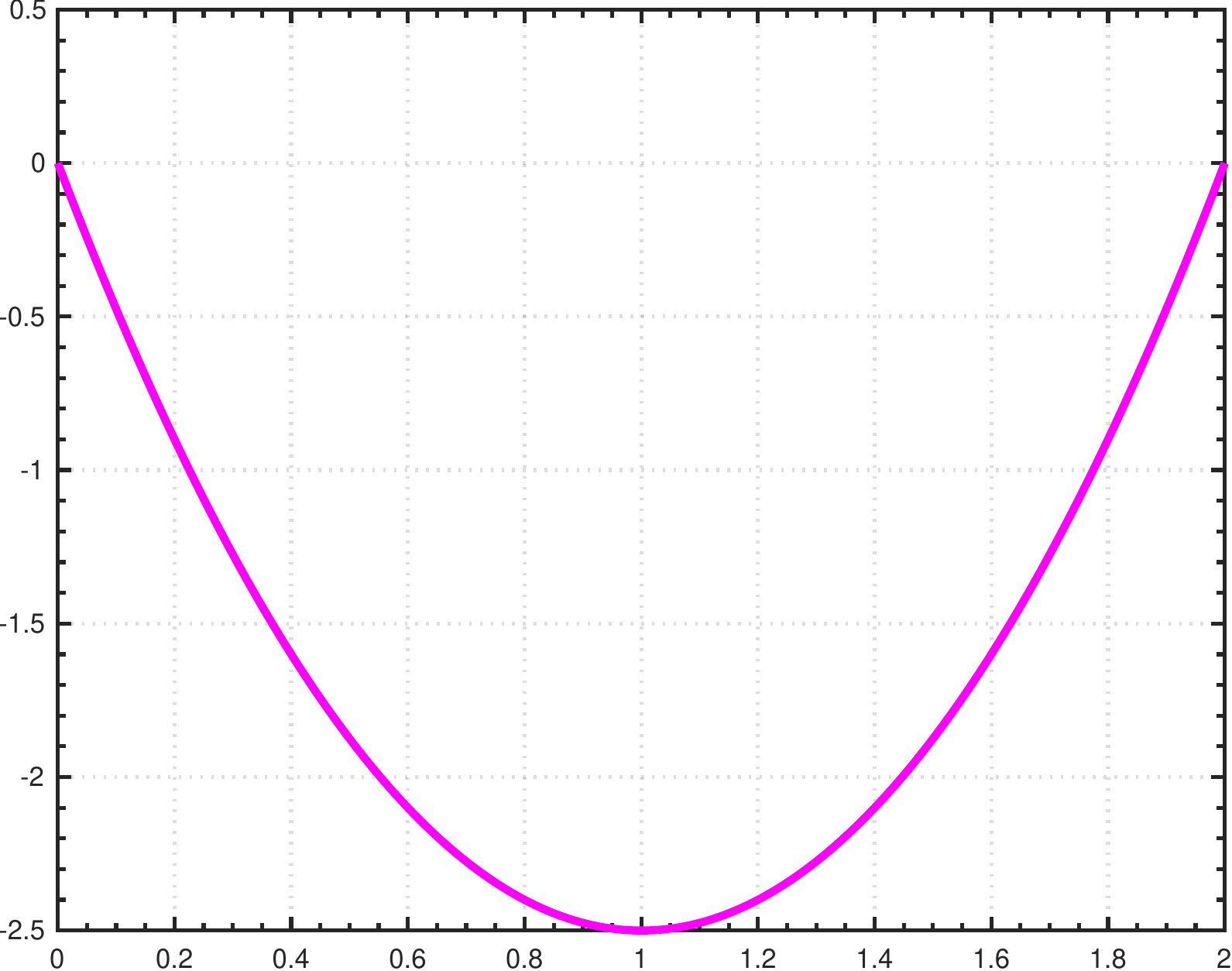}
	\end{tabular}
	\begin{tabular}{c}
	$t=0.025$ \\
	\includegraphics[scale=0.25]{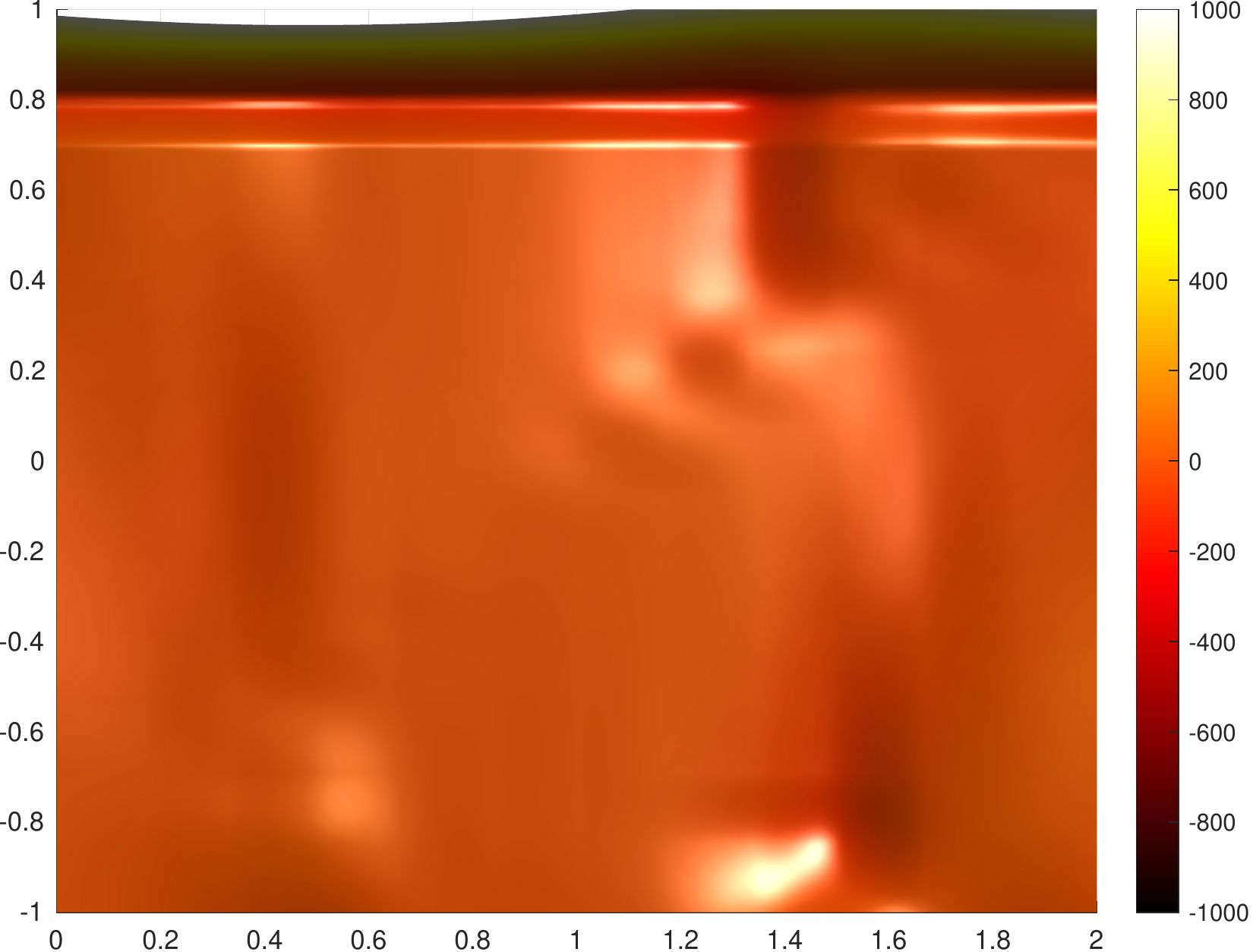}\\
	\includegraphics[scale=0.25]{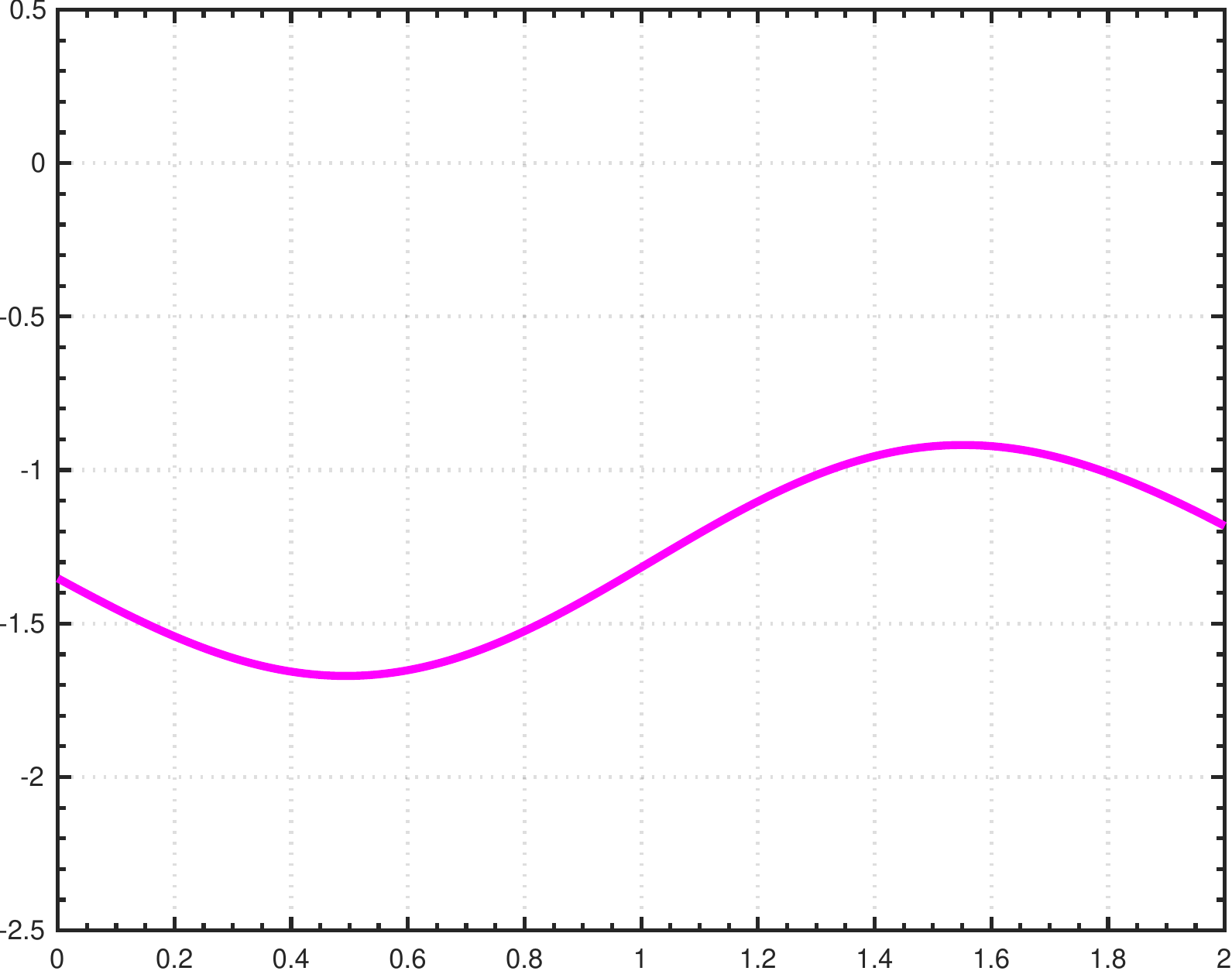}
	\end{tabular}
	\begin{tabular}{c}
	$t=0.1$ \\
	\includegraphics[scale=0.25]{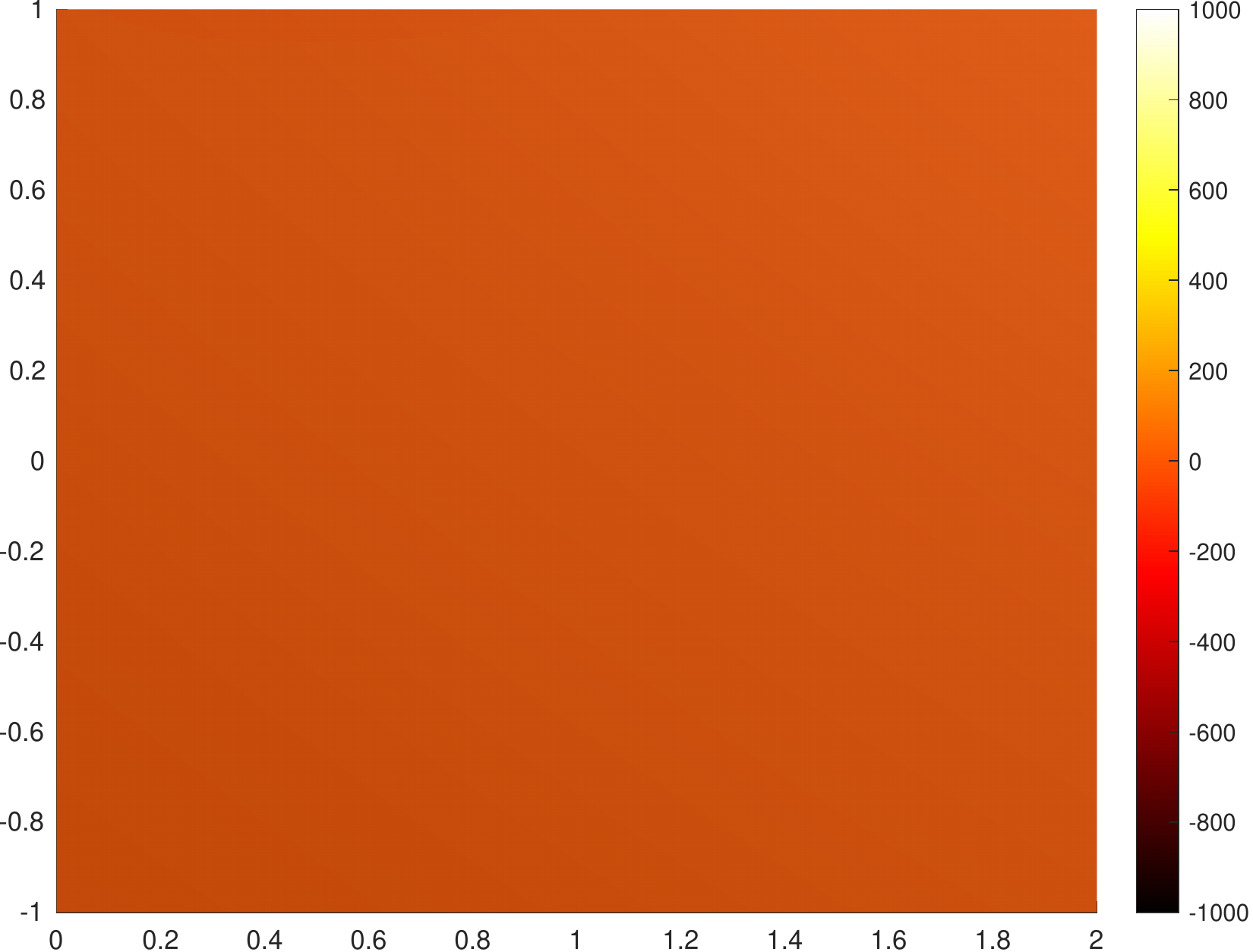}\\
	\includegraphics[scale=0.25]{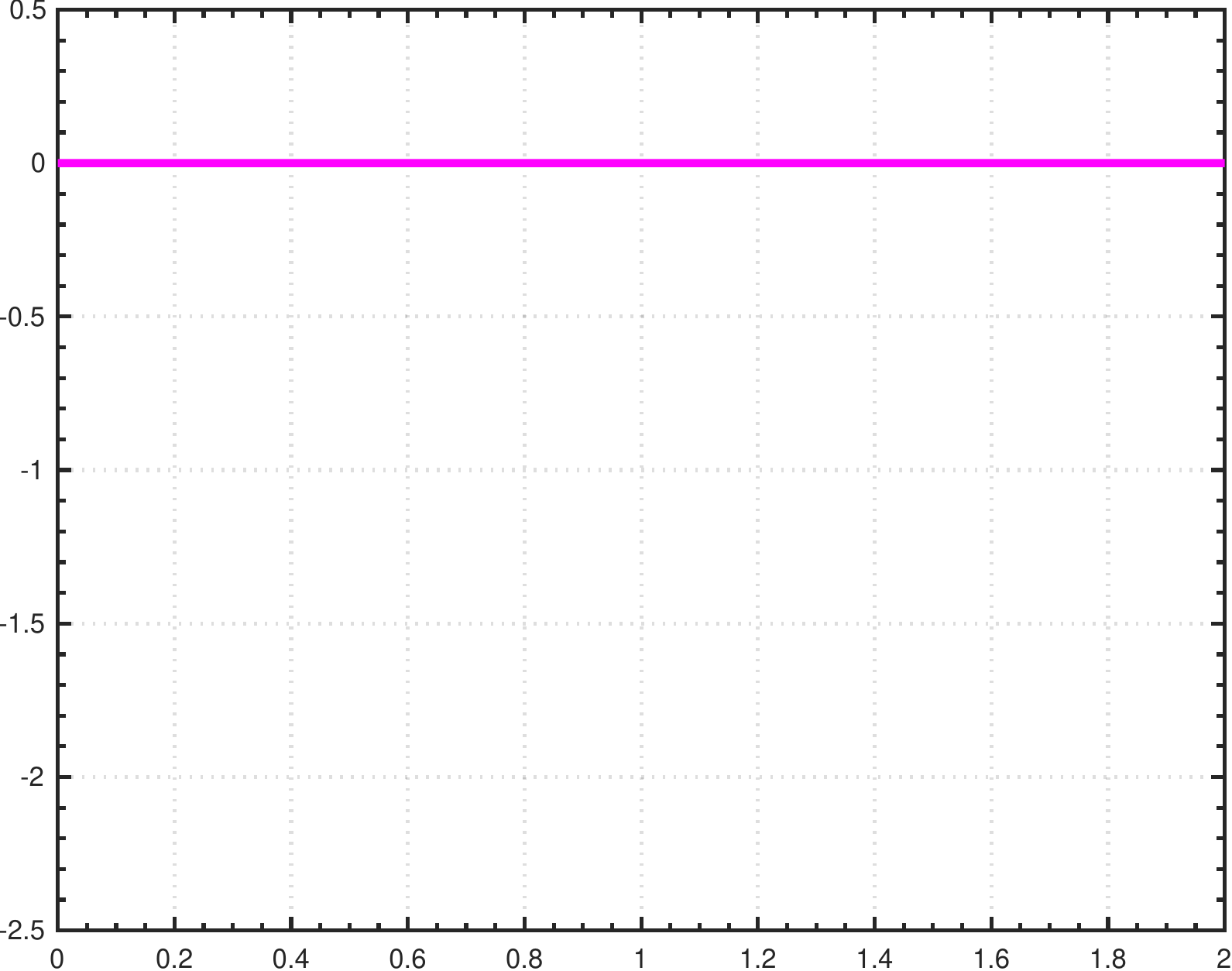}
	\end{tabular}
	\caption{{\bf Controllability of \eqref{sys-lin-con.intro} with $\sigma=10$.}
	 (\emph{Top}) The $L^2(\omega)$-norm, over time, of the minimal $L^2((0,T)\times\omega)$-norm control; (\emph{Middle}) The temperature $y(t,x_1,x_2)$; (\emph{Bottom}) the height $h(t,x_1)$. The numerical experiment reflects \Cref{thm: main.result.linear}.}
	\label{fig: hum.1}
	\end{figure}
	
	To motivate future work in this direction, we provide illustrations of numerical experiments for finding the minimal $L^2((0,T)\times\omega)$-norm control for \eqref{sys-lin-con.intro}, in the cases $\sigma>0$ (specifically, $\sigma=10$) and $\sigma=0$. The results are displayed in \Cref{fig: hum.1} and \Cref{fig: hum.2} respectively.
	Numerical discretization and computing details may be found in \Cref{sec: numerics}.  
	For simplicity of the implementation, we worked on a rescaled domain $\Omega=(0,2)\times(-1,1)$, whilst keeping periodic boundary conditions with respect to $x_1$.
	In both cases, we took $T=0.1$\footnote{Note that the simulations yield the expected results even when $T\gg1$. We present experiments with $T\ll1$ since this is somewhat the "interesting" regime in the context of null-controllability of heat-like equations.}, with  $h^0(x_1) = x_1(2-x_1)$ and $y^0(x_1,x_2)=70\sin(\pi x_1)\sin(\pi x_2)$ (up to $x_2=1$ when $\sigma=0$, whereas we impose the compatibility condition $y(0,x_1,1)=\sigma\partial_{x_1}^2 h^0(x_1)$ when $\sigma=10$). 
	The numerical experiments depicted in \Cref{fig: hum.2} insinuate that null-controllability might also hold when $\sigma=0$. For the time being, a rigorous analytical proof (or disproof) of such a result remains an open problem.
	
	\begin{figure}[h!]
	\includegraphics[scale=0.3]{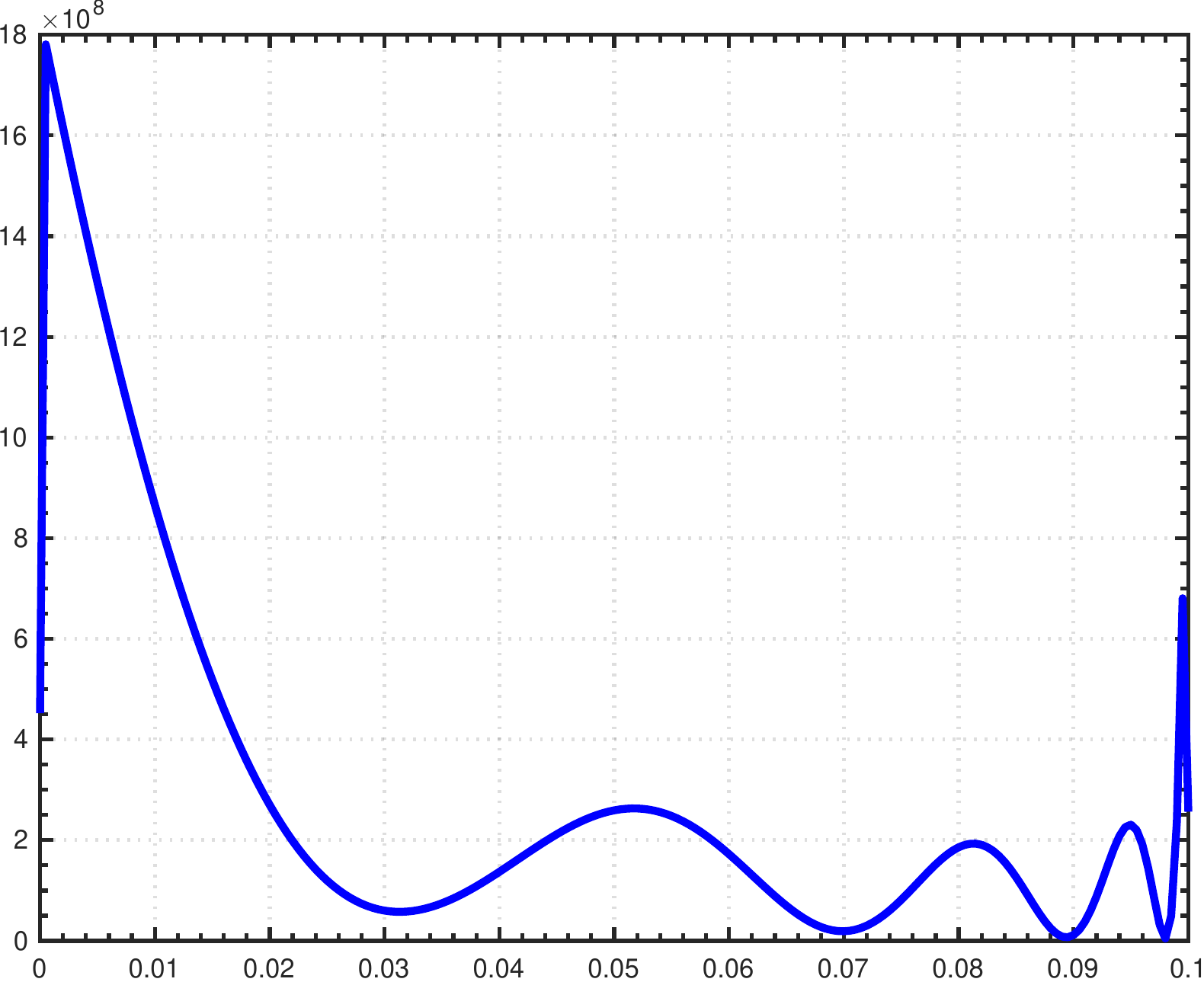}\\
	\vspace{0.2cm}
	\begin{tabular}{c}
	$t=0$ \\
	\includegraphics[scale=0.25]{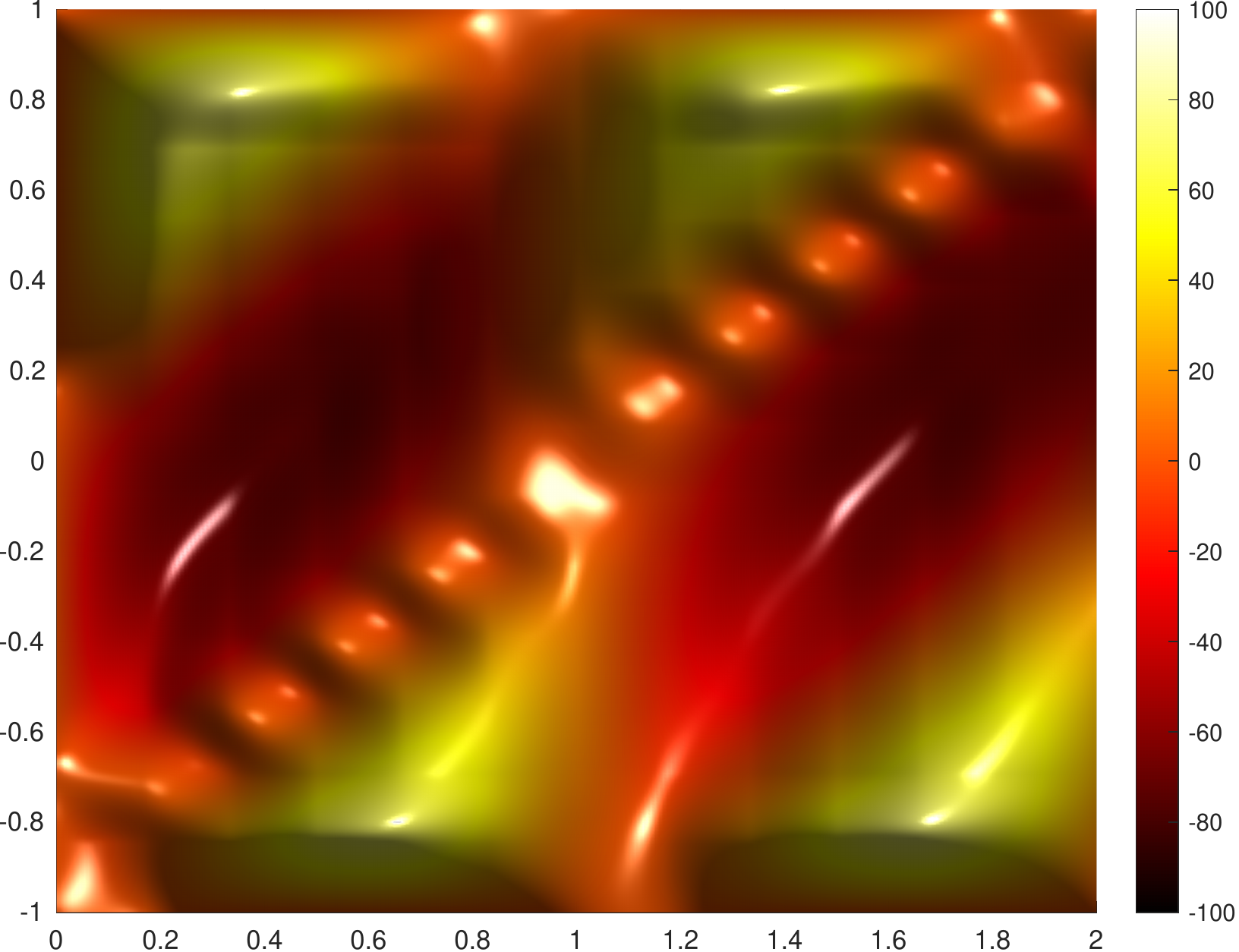}\\
	\includegraphics[scale=0.25]{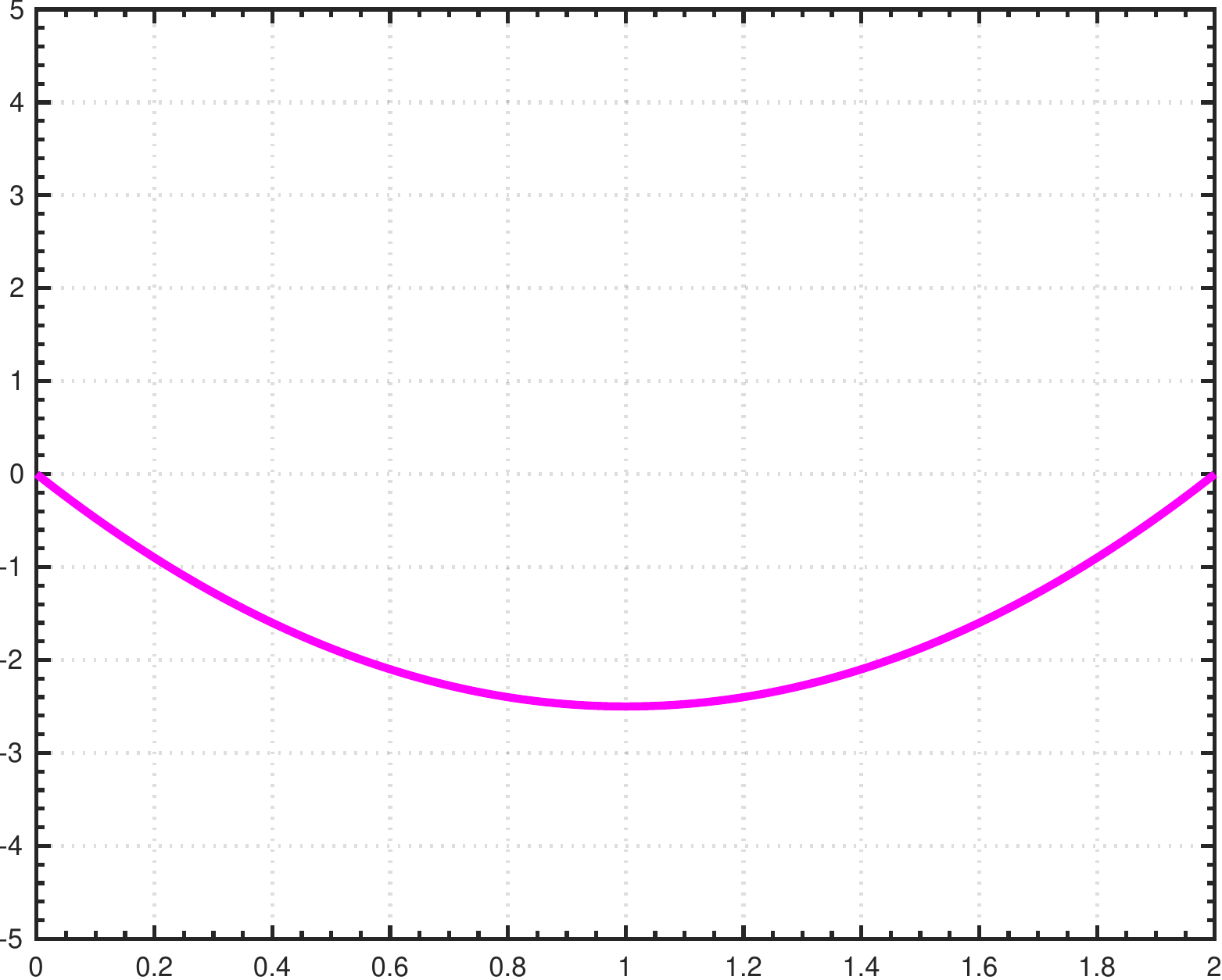}
	\end{tabular}
	\begin{tabular}{c}
	$t=0.05$ \\
	\includegraphics[scale=0.25]{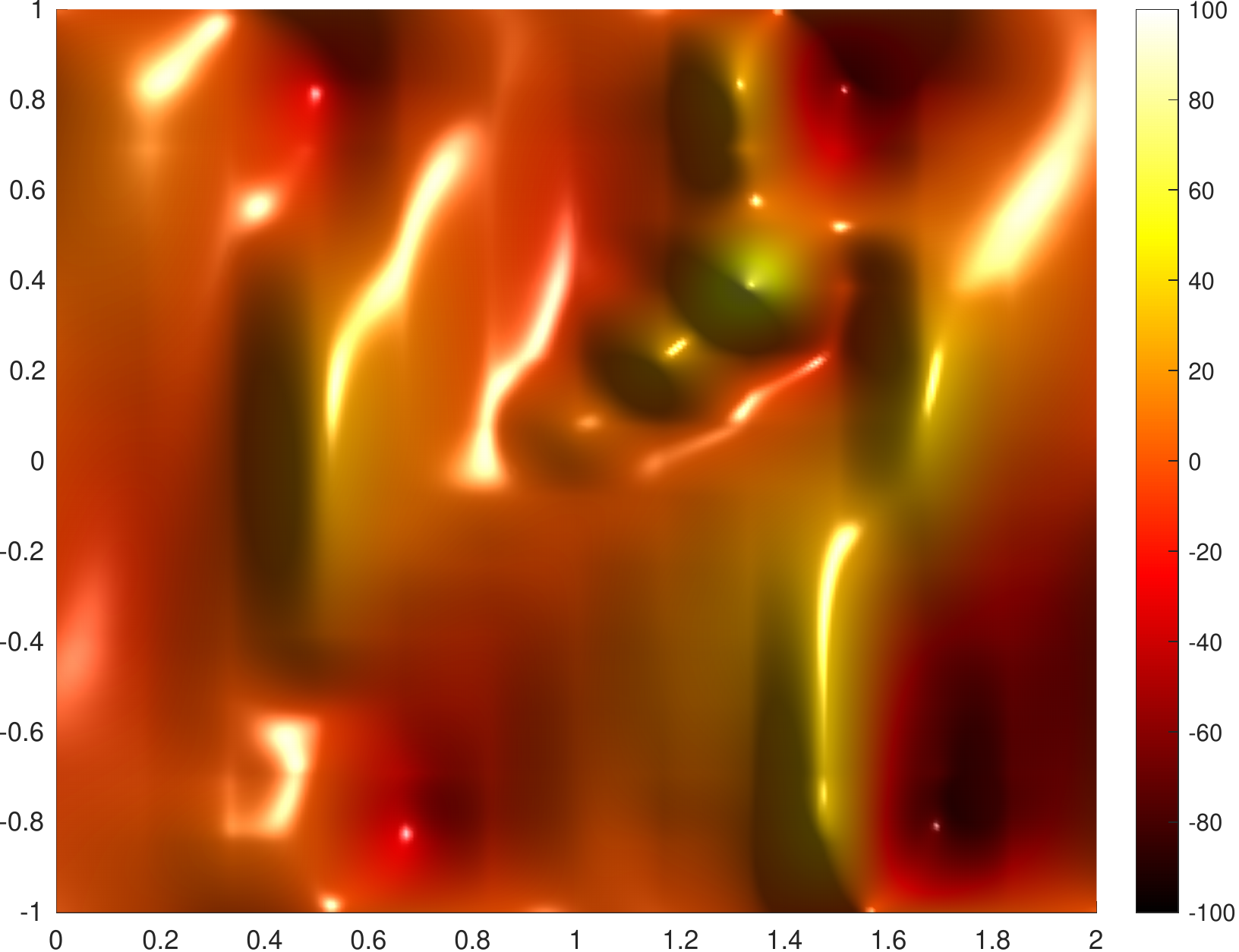}\\
	\includegraphics[scale=0.25]{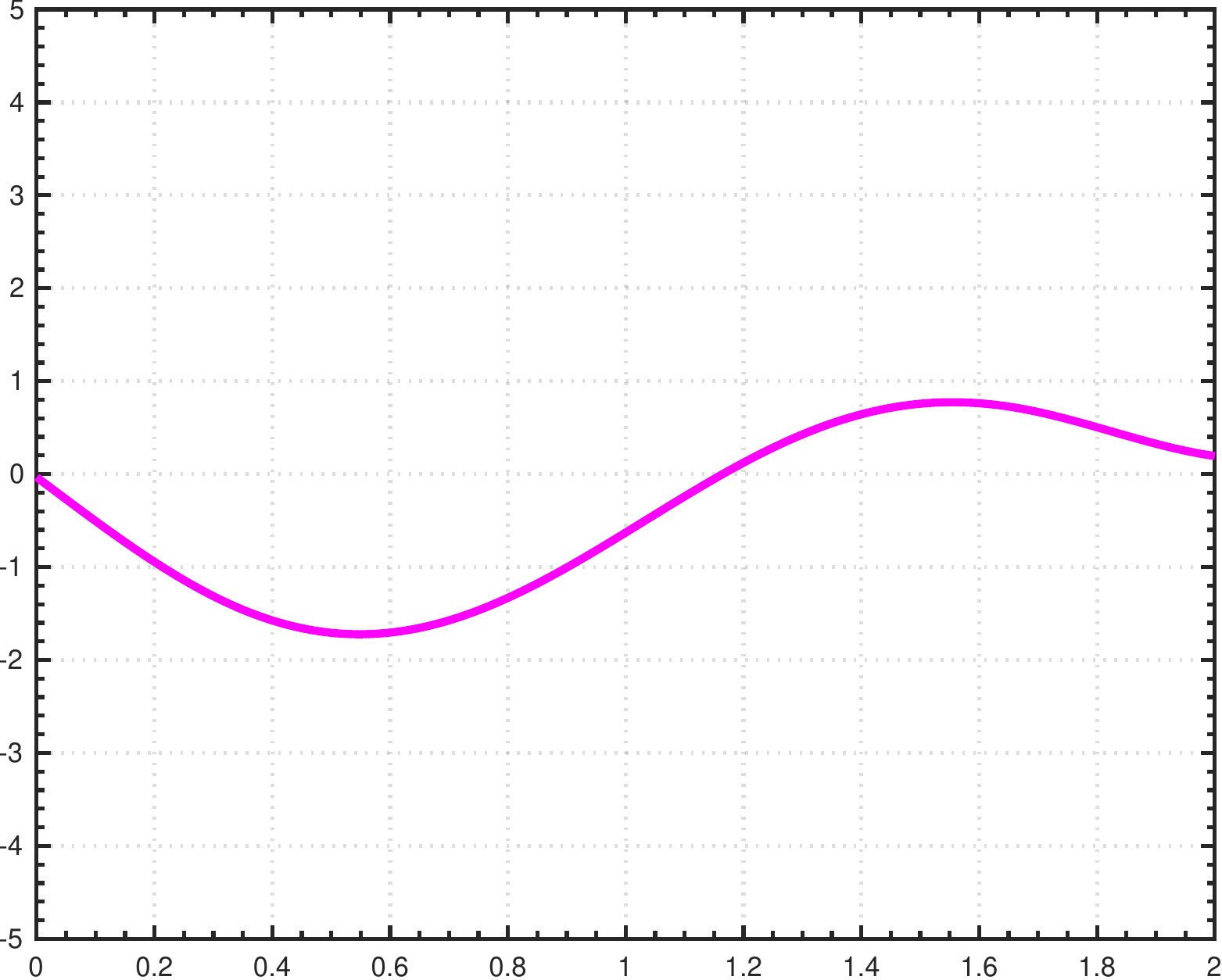}
	\end{tabular}
	\begin{tabular}{c}
	$t=0.1$ \\
	\includegraphics[scale=0.25]{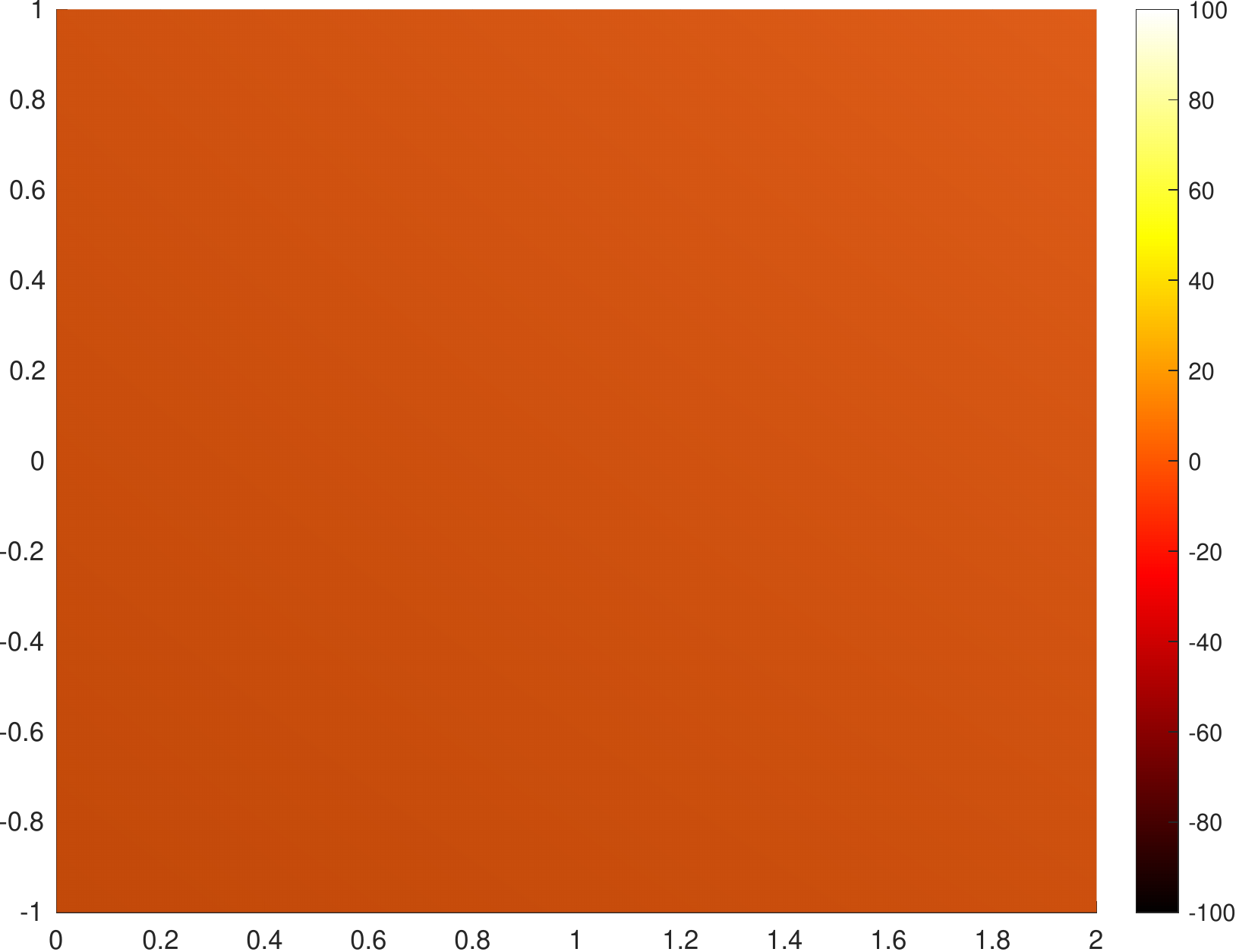}\\
	\includegraphics[scale=0.25]{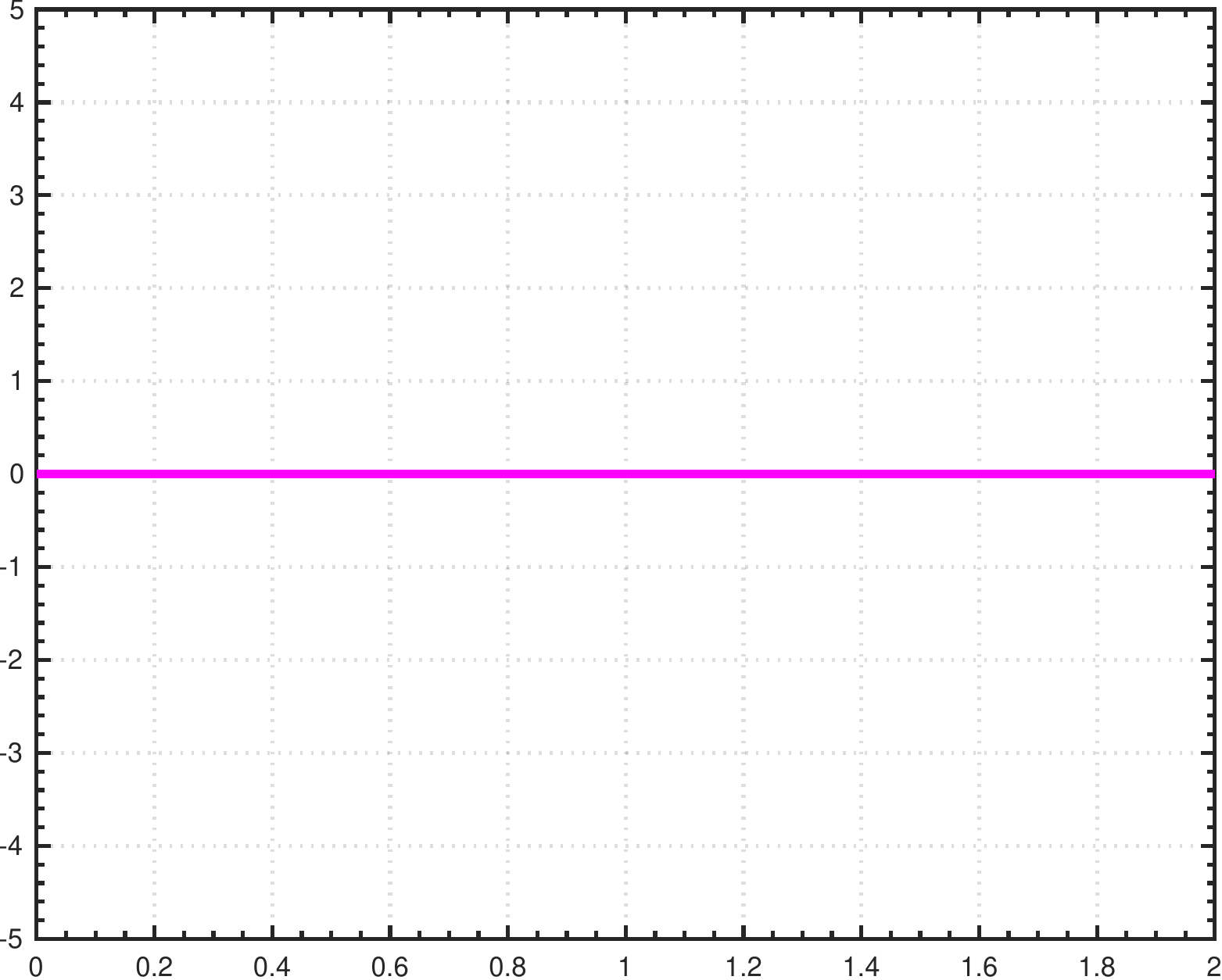}
	\end{tabular}
	\caption{{\bf Controllability of \eqref{sys-lin-con.intro} with $\sigma=0$.} 
	(\emph{Top}) The $L^2(\omega)$-norm of the minimal $L^2((0,T)\times\omega)$-norm control; (\emph{Middle}) The temperature $y(t,x_1,x_2)$; (\emph{Bottom}) the height $h(t,x_1)$. The numerical experiment stipulates that null-controllability holds even when the components are uncoupled, and thus $h(t,x_1)$ is only controlled indirectly, through the time-average of the Neumann trace of $y$, evaluated at $x_2=1$. Note that here, the scale is $10$ times lower than in \Cref{fig: hum.1}, to better capture the controlled evolution.}
	\label{fig: hum.2}
	\end{figure}

	\subsection{Extensions}
	Let us conclude this section with a brief discussion on the current limitations and possible extensions of our results.
	
	{\color{black}
	\begin{remark}[A global version of \Cref{thm:main-nonlinear}] 
	In the context of parabolic systems, it is typical to use a stabilizing property of the nonlinear dynamics to reach the smallness regime, beyond which the local null-controllability result can be applied, with the effect of obtaining a global null-controllability result (albeit in large time). This is much less obvious in the setting of \eqref{eq: stefan.surface.tension}. Indeed, the strongest results we are aware of regarding the well-posedness of \eqref{eq: stefan.surface.tension} are those of \cite{HS17}, in which the authors show local-in-time existence and uniqueness of solutions (under stronger regularity and sign assumptions on the initial data), and also \cite{hadvzic2012orthogonality}, in which global-in-time well-posedness and exponential decay is shown under smallness assumptions on the initial data. 
	In particular, we are not aware of any global-in-time or decay result for \eqref{eq: stefan.surface.tension} which could allow us to obtain a global analog of \Cref{thm:main-nonlinear}.
	\end{remark}
	}
	
	\begin{remark}[Geometric setup] 
	\begin{enumerate}
	\item We parametrize the horizontal variable $x_1$ of any $x=(x_1,x_2)\in\Omega(t)$ (or in $\Gamma(t)$) over the torus $\T$ for convenience.
	This is actually standard in the literature on free boundary problems.
	This choice allows us to avoid rather complicated ensuing arguments regarding the regularity of the moving domains. Moreover, our geometrical setting is also amenable to Fourier analysis, which gives us a natural blueprint, based on Fourier decomposition and spectral analysis for one-dimensional problems, for tackling the control problem. 
	\smallskip 
	\item There is no strict need of working in the two-dimensional setting regarding the linear problem \eqref{sys-lin-con.intro}. We do so mainly for convenience.
	The one-dimensional torus $\T$ may be replaced by $\T^{d-1}$, $d\geqslant3$, throughout, without changing the strategy in the slightest. Indeed, just like when $d=2$, Fourier decomposition can be performed with respect to the first $d-1$ variables $(x_1,\ldots,x_{d-1})$, when $d\geqslant3$ -- only the "vertical" one  (namely $x_d$) will remain in the projected, one-dimensional system, and the strategy of proof for \Cref{thm: main.result.linear} can be maintained.
	\smallskip  
	\item 
	The higher-dimensional ($d\geqslant3$) setting may pose an obstacle in a prospective extension of \Cref{thm: main.result.linear} to the nonlinear setting, where the dimension would play a role regarding the spatial regularity of solutions (in particular that of the height function $h(t,\cdot)$). More specifically, one might have to ensure that the controls for the linearized system are more regular than simply $L^2_{t}L^2_x$. We leave this open for a future study -- a possible roadmap would include using the penalized-HUM method, as done in \cite{le2020local}. 
	\end{enumerate}
	\end{remark}
	
\section{The linear semigroup} \label{sec-linear-sg}
		
	\subsection{Functional setting}	
	Recalling that $\Omega:=\T\times(-1,1)$, let us consider the Hilbert space 
\begin{equation*} \label{space-H}
\mathscr{H}:= L^{2}(\Omega) \times H^{1}(\T),
\end{equation*}
which, by Plancharel's theorem, may be endowed with the norm
\begin{equation} \label{eq: H.norm}
\left\|\big(f,g\big)\right\|_{\mathscr{H}}^2:=\sum_{n \in \mathbb{Z}} \Bigg\{\big\|f_{n}\big\|_{L^{2}(-1,1)}^{2} + \Big(1+ \sigma n^{2}\Big)\big|g_{n}\big|^{2}\Bigg\},
\end{equation}
where the Fourier coefficients $(f_n, g_n)$ are defined as in \eqref{eq: yn.fourier}.
The inner product on $\mathscr{H}$ can then be inferred from \eqref{eq: H.norm}.

We look to rewrite \eqref{sys-lin-con.intro} in a canonical first-order form evolving on $\mathscr{H}$. The well-posedness thereof will follow by studying the governing operator generating the semigroup of \eqref{sys-lin-con.intro}. 
We introduce the unbounded operator $\*A: \mathfrak{D}(\*A) \to \mathscr{H}$, defined by 
\begin{equation*}
\*A\big(y, h\big):= \big(\Delta y, \partial_{x_{2}} y(\cdot,1)\big),
\end{equation*}
with domain
\begin{align} \label{eq:DA}
\mathfrak{D}(\*A) = \Bigg\{(y, h) \in H^{2}(\Omega) \times H^{\sfrac{7}{2}}(\T)\,\,\Biggm|\,\,&y(x_1,1) = \sigma \del_{x_1}^2 h(x_1) \hspace{0.15cm} \text{ on } \T,\\
&y(x_1,-1) = 0 \hspace{0.15cm} \text{ on } \T,\nonumber \\
&\partial_{x_{2}} y(\cdot,1) \in H^{1}(\T) \Bigg\}\nonumber.
\end{align}
It can readily be seen that the operator $\*A: \mathfrak{D}(\*A) \to \mathscr{H}$ is closed and densely defined. 
By using the above definitions, we can rewrite \eqref{sys-lin-con.intro} (for a general source term $f$ instead of $u1_\omega$) as
\begin{equation} \label{eq: state.space.form}
\begin{cases}
\partial_t\big(y, h\big)=  \*A\big(y,h\big)+\big(f,0\big) &\text{ in } (0,T), \\ 
\big(y,h\big)_{\mid_{t=0}} = \left(y^{0}, h^{0}\right).
\end{cases}
\end{equation}

\subsection{Well-posedness of \eqref{eq: state.space.form}}

To study the well-posedness of \eqref{eq: state.space.form}, we shall rely on a Fourier decomposition in the horizontal ($x_1\in\T$) variable. This will lead to a system akin to \eqref{eq: fourier.projected.system.intro}. Specifically, for any $n\in\Z$ we consider the Hilbert space
		\begin{equation*}
			\mathscr{H}_{n}:=L^2(-1, 1)\times\R,
		\end{equation*}
		which we endow with the inner product
		\begin{equation*}
		\big\langle (f_1,g_1),(f_2, g_2)\big\rangle_{\mathscr{H}_{n}} := \big\langle f_1, f_2\big\rangle_{L^2(-1,1)} + \sigma n^2 g_1 g_2,
		\end{equation*}
		when $n\neq0$, and the canonical inner product when $n=0$. 
		We then define, for any $n\in \Z$, the operator $\*A_n: \mathfrak{D}(\*A_n) \rightarrow \mathscr{H}_{n}$ 
		by
		\begin{equation*}
			 \*A_n\big(y, h\big):= \Big(\del_{x_2}^2 y  -n^2 y,\, \del_{x_2} y(1)\Big),
		\end{equation*}
		with domain 
		\begin{equation*}
			\mathfrak{D}(\*A_n)= \left\{ \big(y, h\big)\in H^2(-1, 1) \times\R\, \Biggm| \, y(-1) = 0, \, y(1) = -\sigma n^2 h \right\}.
		\end{equation*}
We will show that the operators $\*A$ and $\*A_{n}$ ($n\in\Z$) generate analytic semigroups on $\mathscr{H}$, and $\mathscr{H}_{n}$, respectively. The fact that $\*A_{n}$ generates an analytic semigroup on $\mathscr{H}_{n}$ actually follows directly from \cite[Theorem 1.25]{maity_tucsnak_chapter}. 
To show that $\*A$ generates an analytic semigroup, we shall express its resolvent in terms of the resolvent of the operators $\*A_{n}$. Thus, we shall need resolvent estimates for $\*A_{n}$ which are uniform with respect to $n\in\Z$.

For $\theta \in\left(\frac{\pi}{2}, \pi\right)$ and $\beta>0$, we define the sector 
\begin{equation*}
\Sigma_{\theta, \beta} := \Big\{ \lambda \in \mathbb{C} \setminus \{ 0\} \, \Bigm| \, |\mathrm{arg}(\lambda) | < \theta, \, |\lambda|\geqslant \beta \Big\}.
\end{equation*}
The following result then holds. 

\begin{proposition}[Regarding $\*A_n$] 
\label{prop:resolvent-An} 
Suppose $n\in\Z$ and $\sigma>0$.
\begin{enumerate}
\item If $n\neq 0$, the operator $\*A_n: \mathfrak{D}(\*A_n)\to\mathscr{H}_{n}$ is self-adjoint, has compact resolvents, and its spectrum $ \mathrm{spec}(\*A_n)$ consists only of negative eigenvalues.
\smallskip

\item There exist $\theta \in\left(\frac{\pi}{2},\pi\right)$, $\beta > 0$ and $C_{\theta, \beta} > 0$, all independent of $n$, such that 
\begin{equation} \label{est:resolvent-An}
|\lambda| 
\left\|\left(\lambda\,\mathrm{Id} - \*A_{n}\right)^{-1}\right\|_{\mathscr{L}(\mathscr{H}_{n})}
 \leqslant C_{\theta, \beta}  
\end{equation}
holds for all $\lambda \in \Sigma_{\theta, \beta}$.
\end{enumerate}
Consequently, for any $n\in\mathbb{Z},$ the operator $\*A_{n}$ generates an analytic semigroup $\left\{e^{t\*A_{n}}\right\}_{t\geqslant0}$ on $\mathscr{H}_{n}$.
\end{proposition}

\begin{proof} The fact that both claims imply that $\*A_n$ generates an analytic semigroup follows from \cite[Theorem 2.11 (p. 112), Proposition 2.11 (p. 122)]{benoussan_book}. 

Let us begin by proving the first claim. If $n \neq 0,$ through standard integration by parts, is readily seen that $\*A_{n}$ is self-adjoint and has compact resolvents. Therefore, its spectrum $\mathrm{spec}(\*A_n)$ is a discrete subset of $\mathbb{R}$. Let us conclude the proof the first claim by showing that $\mathrm{spec}(\*A_n) \subset (-\infty,0)$. We argue by contradiction. 
		Let $\lambda \in\mathrm{spec}(\*A_n)$ with $\lambda \geqslant 0.$
		Thus there exists a vector $(y,h) \in \mathfrak{D}(\*A_n)\setminus\{0\}$ such that
		\begin{equation*}
		\begin{cases}
		\lambda y- \del_{x_2}^2 y + n^2 y = 0 &\text{ in }(-1, 1),\\
		\lambda h - \del_{x_2}y(1) = 0,\\
		y(-1) = 0,\\
		y(1) = -\sigma n^2 h.
		\end{cases}
		\end{equation*}
		We now multiply the first equation by $y$ and integrate by parts to obtain
		\begin{equation*}
		\lambda \int_{-1}^1 |y|^2 \diff x_2 + \int_{-1}^1 |\del_{x_2}y|^2 \diff x_2 - \del_{x_2}y(1)y(1) + n^2 \int_{-1}^1 |y|^2 \diff x_2 = 0.
		\end{equation*}
		Using the boundary conditions, this identity entails
		\begin{equation*} 
		\lambda \int_{-1}^1 |y|^2 \diff x_2 + \int_{-1}^1 |\del_{x_2} y|^2 \diff x_2 + \lambda \sigma n^2 h^2 + n^2 \int_{-1}^1 |y|^2 \diff x_2 = 0.
		\end{equation*}
		Since $\lambda\geqslant0$ and $n\neq0$, from this identity, we may readily conclude that $y\equiv h\equiv 0.$ This is a contradiction, and hence $\lambda\in(-\infty, 0)$.
		
		We now look to prove the second claim. Suppose that $n \in \mathbb{Z}$ and $\left(f_{n},g_{n}\right)\in\mathscr{H}_{n}$ are arbitrary. Let us consider the resolvent problem 
\begin{equation} \label{sys-resolvent-n.1}
\begin{cases}
\lambda y_n - \partial_{x_{2}}^2 y_n + n^{2} y_n =  f_{n} & \mbox{ in } (-1,1), \\
\lambda h_{n} =  \partial_{x_{2}} y_{n}(1)  +  g_{n}\\
y_{n}(-1) =  0 \\
y_{n}(1) = -\sigma n^{2} h_{n}. 
\end{cases}
\end{equation}
If $n =0,$ the operator $\*A_{0}$ generates an analytic semigroup (see \cite[Theorem 1.25]{maity_tucsnak_chapter}). Therefore, there exist $\theta \in \left(\frac{\pi}{2}, \pi\right)$, $\beta >0$, and $C_{\theta,\beta}>0$ such that 
\begin{equation} \label{est:res-0}
|\lambda| 
\left\|\big(y_0, h_0\big)\right\|_{L^2(-1,1)\times \mathbb{C}}
 \leqslant C_{\theta,\beta} \norm{\big(f_{0}, g_{0}\big)}_{L^{2}(-1,1) \times \mathbb{C}},
\end{equation}
holds for all $\lambda\in\Sigma_{\theta, \beta}$. Now let $n\neq0$ be fixed. Since $\mathrm{spec}(\*A_n)\subset(-\infty,0)$, we also have 
\begin{equation*}
\Sigma_{\theta, \beta} \subset \mathrm{res}(\*A_{n}), 
\end{equation*}
where $\mathrm{res}(\*A_{n})$ denotes the resolvent set of $\*A_{n}.$ Consequently, for any $n\neq 0,$ $\lambda \in \Sigma_{\theta, \beta}$, and $(f_{n},g_{n}) \in \mathscr{H}_{n},$ \eqref{sys-resolvent-n.1} admits a unique solution $(y_{n}, h_{n})\in\mathfrak{D}(\*A_{n}).$
Let us now take any $\lambda  = \beta_0 e^{i\theta_0} \in \Sigma_{\theta, \beta}.$
%
Multiplying the first equation in $\eqref{sys-resolvent-n.1}$ by $e^{-\sfrac{i\theta_0}{2}}$ and taking the inner product with $y_{n},$ we obtain 
\begin{align*}
\beta_0 e^{i\frac{\theta_0}{2}} \int_{-1}^{1} \left|y_{n}\right|^{2} \diff x_2 + e^{-i\frac{\theta_0}{2}} \int_{-1}^{1} \left|\partial_{x_{2}}y_{n}\right|^{2} \diff x_2 - e^{-\frac{i\theta_0}{2}} \partial_{x_{2}}y_{n}(1)y_{n}(1)\\
+ e^{-i\frac{\theta_0}{2}}n^2\int_{-1}^1 |y_n|^2\diff x_2
= e^{-\frac{i\theta_0}{2}} \int_{-1}^{1} f_{n} y_{n} \diff x_2.
\end{align*}
Using the boundary conditions, the above identity can be rewritten as 
\begin{align*}
\beta_0 e^{i\frac{\theta_0}{2}} \int_{-1}^{1} \left|y_{n}\right|^{2} \diff x_2 + e^{-i\frac{\theta_0}{2}} \int_{-1}^{1} \left|\partial_{x_{2}}y_{n}\right|^{2} \diff x_2 &+ \beta_0 e^{i\frac{\theta_0}{2}} \sigma n^{2} \left|h_{n}\right|^{2}+ e^{-i\frac{\theta_0}{2}}n^2\int_{-1}^1 |y_n|^2\diff x_2 \\
&= e^{-i\frac{\theta_0}{2}} \int_{-1}^{1} f_{n}y_{n} \diff x_2 + e^{-i\frac{\theta_0}{2}} \sigma n^{2} h_{n} g_{n}\\
&=e^{-i\frac{\theta_0}{2}}\Big\langle\big(y_n,h_n\big), \big(f_n, g_n\big)\Big\rangle_{\mathscr{H}_{n}}.
\end{align*}
By taking the real part on both sides in the above identity, and subsequently using Cauchy-Schwarz, we find 
\begin{multline*}
\beta_0 \left(  \int_{-1}^{1} \left|y_{n}\right|^{2} \diff x_2 + \sigma n^{2}  \left|h_{n}\right|^{2}  \right)+\int_{-1}^{1}\left|\partial_{x_{2}}y_{n}\right|^{2}+n^2\int_{-1}^1|y_n|^2\diff x_2 \\ 
\leqslant\left\|\big(y_n, h_n\big)\right\|_{\mathscr{H}_{n}}
 \left\|\big(f_n,g_n\big)\right\|_{\mathscr{H}_{n}}
\end{multline*}
Taking into account the fact that $|\lambda|=\beta_0$, we deduce that
\begin{equation*}
|\lambda| \left\|\big(y_{n}, h_{n}\big)\right\|_{\mathscr{H}_{n}}  \leqslant  \norm{\big(f_{n}, g_{n}\big)}_{\mathscr{H}_{n}}.
\end{equation*}
Since $n\neq0$ and $\lambda\in\Sigma_{\theta_0,\beta_0}$ were taken arbitrary, the above estimate in junction with \eqref{est:res-0} leads us to \eqref{est:resolvent-An}.
\end{proof}

The above result then leads us to the following.
			
\begin{proposition}[Regarding $\*A$] \label{prop: analytic.semigroup}
Suppose $\sigma>0$. There exist $\theta \in \left(\frac{\pi}{2}, \pi\right),$ $\beta > 0$ and $C_{\theta, \beta} > 0$ such that 
\begin{equation} \label{est:resolvent}
|\lambda| 
\left\|\left(\lambda\,\mathrm{Id}- \*A\right)^{-1}\right\|_{\mathscr{L}(\mathscr{H})}
\leqslant C_{\theta, \beta}
\end{equation}
holds for all $\lambda \in \Sigma_{\theta, \beta}$.
Consequently, the  operator $\*A:\mathfrak{D}(\*A)\to\mathscr{H}$ generates an analytic semigroup $\left\{e^{t\*A}\right\}_{t\geqslant0}$ on $\mathscr{H}$. {\color{black} Moreover,  the resolvent of $\*A$ is compact in $\mathscr{H}.$}
\end{proposition}

\begin{proof}[Proof of \Cref{prop: analytic.semigroup}]
Let $\theta \in \left(\frac{\pi}{2},\pi\right)$ and $\beta > 0$ be the constants stemming from \Cref{prop:resolvent-An}. 
For $\lambda \in \Sigma_{\theta, \beta}$ and $(f, g) \in \mathscr{H},$ we consider the eigenvalue problem 
\begin{equation} \label{sys-resolvent}
\begin{cases}
\lambda y- \Delta y =  f &\text{ on } \Omega,\\
\lambda  h - \partial_{x_{2}} y(\cdot,1)  +  g & \mbox{ on } \T,\\
y(\cdot,-1) = 0 &  \mbox{ on } \T,\\
y(\cdot,1) = \sigma \del_{x_1}^2 h(\cdot) & \mbox{ on } \T. 
\end{cases}
\end{equation}
We decompose all functions appearing in the above eigenvalue problem in Fourier series with respect to the periodic, $x_1$--variable, as in \eqref{eq: fourier.y} -- \eqref{eq: yn.fourier}.
We see that for any $n \in \mathbb{Z},$ the pair $\left(y_n, h_n\right)\in\mathscr{H}_n$ of Fourier coefficients solves \eqref{sys-resolvent-n.1}, and moreover, because of the Fourier series expansion of $(y,h)$ as in \eqref{eq: fourier.y}, 
\begin{equation*}
\left(\lambda\, \mathrm{Id}- \*A\right)^{-1} \big(f, g\big) =   \frac{1}{\sqrt{2\pi}}\sum_{n \in \mathbb{Z}} \left(\lambda\,\mathrm{Id} - \*A_{n}\right)^{-1} \big(f_{n}, g_{n}\big) e^{in \cdot}
\end{equation*}
also holds. Combining the above relation with \eqref{est:resolvent-An}, we immediately obtain \eqref{est:resolvent}. This completes the first part of the  proof. 
{\color{black} 
We consider the elliptic boundary value problem $\eqref{sys-resolvent}_{1-3}$ satisfied by $y$. Note that 
\begin{equation*}
\partial_{x_{2}}y(\cdot, 1) = \lambda h - g \in H^{1}(\mathbb{T}). 
\end{equation*}
Therefore, $y \in H^{2}(\Omega)$ and $y(\cdot, 1) \in H^{\sfrac{3}{2}}(\mathbb{T}).$ Finally, using $\eqref{sys-resolvent}_{4}$ we obtain $h \in H^{\sfrac{7}{2}}(\mathbb{T})$. Whence $(y,h)\in\mathfrak{D}(\*A)$, and the compactness readily follows.
}
\end{proof}

 \begin{remark}[Fourier decomposition of the semigroup]
In view of \Cref{prop:resolvent-An} and \Cref{prop: analytic.semigroup}, we have 
\begin{equation*} \label{semigroup.formula}
e^{t\*A} \left(y^{0}, h^{0}\right)  = \frac{1}{\sqrt{2\pi}}\sum_{n \in \mathbb{Z}}  e^{t\*A_{n}} \left(y_{n}^{0}, h_{n}^{0}\right) e^{in\cdot},
\end{equation*}
where $\left(y^{0}_{n}, h^{0}_{n}\right)$ denote the Fourier coefficients of an initial datum $\left(y^0,h^0\right)$, defined as in \eqref{eq: yn.fourier}.
\end{remark}
 
Taking stock of \Cref{prop: analytic.semigroup}, and using standard results from parabolic equations (see e.g. \cite[Thm. 2.12, Sect. 2]{benoussan_book}), we deduce the well-posedness of the linear system \eqref{eq: state.space.form}.

\begin{corollary} \label{cor: linear.well.posed}
Suppose $T>0$ and $\sigma>0$. 
For every $\left(y^0, h^0\right)\in L^2(\Omega)\times H^1(\T)$ and $f \in L^2(0,T; L^2(\Omega))$, there exists a unique solution $(y,h)\in C^0([0,T]; L^2(\Omega)\times H^1(\T))$ to \eqref{eq: state.space.form}.
\end{corollary}
 
{\color{black}
\subsection{Adjoint of $\*A$}
In this subsection we determine the adjoint of $\*A$. Let 
\begin{equation*}
\mathscr{H}^{*}=L^{2}(\Omega) \times H^{-1}(\T) 
\end{equation*}
denote the dual space of $\mathscr{H},$ with $H^{-1}(\T)$ being the dual of $H^{1}(\T)$ with respect to the pivot space $L^{2}(\T).$ We may endow this space with the norm 
\begin{equation*}
\left\|\big(f,g\big)\right\|_{\mathscr{H}^{*}}^2:=\sum_{n \in \mathbb{Z}} \Bigg\{\big\|f_{n}\big\|_{L^{2}(-1,1)}^{2} + \Big(1+ \sigma n^{2}\Big)^{-1}\big|g_{n}\big|^{2}\Bigg\}.
\end{equation*}

\begin{proposition} \label{prop-main-adj}
The adjoint $\*A^{*}$ of $\*A$ is defined by 
\begin{equation*}
\*A^{*}\big(\zeta, \ell\big):= \big(\Delta \zeta, \sigma \partial_{x_{1}}^{2}\partial_{x_{2}} \zeta(\cdot,1)\big),
\end{equation*}
with domain
\begin{align} \label{eq:DA*}
\mathfrak{D}(\*A^{*}) = \Bigg\{(\zeta, \ell) \in H^{2}(\Omega) \times H^{\sfrac{3}{2}}(\T)\,\,\Biggm|\,\,&\zeta(x_1,1) = \ell \hspace{0.15cm} \text{ on } \T, \quad \zeta(x_1,-1) = 0 \hspace{0.15cm} \text{ on } \T,\nonumber \\
&\sigma \partial_{x_{1}}^{2} \partial_{x_{2}} \zeta(\cdot,1) \in H^{-1}(\T) \Bigg\}\nonumber.
\end{align}
\end{proposition}

In order to prove the proposition, we first show that the operator $\*A^{*}$ defined above is surjective in $\mathscr{H}^{*}.$ Let us consider the associated resolvent problem 
\begin{equation} \label{sys-resolvent-adj}
\begin{cases}
\lambda \zeta - \Delta \zeta =  f &\text{ on } \Omega,\\
\lambda \ell - \sigma \partial_{x_{1}}^{2}\partial_{x_{2}} \zeta(\cdot,1)  =  g & \mbox{ on } \T,\\
\zeta(\cdot,-1) = 0 &  \mbox{ on } \T,\\
\zeta(\cdot,1) = \ell & \mbox{ on } \T. 
\end{cases}
\end{equation}
We first prove the following lemma. 

\begin{lemma}
Let $\lambda > \beta,$ where $\beta$ is the constant appearing in \Cref{prop:resolvent-An}. If $(f, g) \in \mathscr{H}^{*},$ then the system \eqref{sys-resolvent-adj} admits a unique solution $(\zeta, \ell) \in \mathfrak{D}(\*A^{*}).$ 
\end{lemma}

\begin{proof}
The proof is similar to that of \Cref{prop:resolvent-An}  and \Cref{prop: analytic.semigroup}. We first  decompose all functions appearing in the above eigenvalue problem in Fourier series with respect to the periodic, $x_1$--variable. Then for any $n \in \mathbb{Z},$ $(\zeta_{n}, \ell_{n}),$ the pair of Fourier coefficients of $(\zeta, \ell)$ solves
 \begin{equation*} 
\begin{cases}
\lambda \zeta_{n} -\partial_{x_{2}}^{2} \zeta_{n} + n^{2} \zeta_{n} =  f_{n} &\text{ on } \Omega,\\
\lambda \ell_{n} +  \sigma n^{2} \partial_{x_{2}} \zeta_{n}(1)  =  g_{n} & \mbox{ on } \T,\\
\zeta_{n}(-1) = 0 &  \mbox{ on } \T,\\
\zeta_{n}(1) = \ell_{n} & \mbox{ on } \T,
\end{cases}
\end{equation*}
where $(f_{n}, g_{n}) \in L^{2}(-1,1) \times \mathbb{C}.$ Following the arguments of \Cref{prop:resolvent-An}, we conclude that  for any $n \in \mathbb{Z}$ the above system admits a unique solution and 
\begin{equation*}
\norm{\zeta_{0}}^{2}_{H^{1}(-1,1)} + |\ell_{0}|^{2} \leqslant C \left( \norm{f_{0}}_{L^{2}(-1,1)}^{2} + |g_{0}|^{2} \right), 
\end{equation*}
and 
\begin{equation*}
\norm{\zeta_{n}}_{L^{2}(-1,1)}^{2} + \frac{|\ell_{n}|^{2}}{n^{2}} \leqslant  \norm{f_{n}}_{L^{2}(-1,1)}^{2} + \frac{|g_{n}|^{2}}{n^{2}}\qquad (n \neq 0),
\end{equation*}
for some $C>0$ independent of $n.$ The above two estimates imply that \eqref{sys-resolvent-adj} admits a unique solution $(\zeta, \ell) \in \mathscr{H}^{*}.$ Note that, from $\eqref{sys-resolvent-adj}_{2}$ we also have 
\begin{equation*}
- \sigma \partial_{x_{1}^{2}} \partial_{x_{2}} \zeta(\cdot, 1) = g - \lambda \ell \in H^{-1}(\T). 
\end{equation*}
Therefore, $\partial_{x_{2}} \zeta(\cdot, 1) \in H^{1}(\T).$ Solving the elliptic problem corresponding to $\zeta$ we get $\zeta \in H^{2}(\Omega),$ and from  $\eqref{sys-resolvent-adj}_{4},$ we infer that $\ell \in H^{3/2}(\T).$
\end{proof}

We are now in a position to prove \Cref{prop-main-adj}.
\begin{proof} [Proof of \Cref{prop-main-adj}]
Let $\lambda > \beta,$ and let us set 
\begin{align} 
\mathfrak{Z} = \Bigg\{(\zeta, \ell) \in H^{2}(\Omega) \times H^{\sfrac{3}{2}}(\T)\,\,\Biggm|\,\,&\zeta(x_1,1) = \ell \hspace{0.15cm} \text{ on } \T, \quad \zeta(x_1,-1) = 0 \hspace{0.15cm} \text{ on } \T,\nonumber \\
&\sigma \partial_{x_{1}}^{2} \partial_{x_{2}} \zeta(\cdot,1) \in H^{-1}(\T) \Bigg\}\nonumber.
\end{align}
First of all, by integrating by parts and using the Schwarz theorem vis-à-vis the symmetry of second derivatives, we may readily find 
\begin{equation*}
\Big\langle (\lambda{\mathrm{Id}} - \*A)\big(y, h\big), \big(z, r\big)\Big\rangle_{\mathscr{H}, \mathscr{H^{*}}} = \Big\langle \big(y, h\big), (\lambda{\mathrm{Id}} - \*A^{*}) \big(z, r\big)\Big\rangle_{\mathscr{H}, \mathscr{H^{*}}}  
\end{equation*}
for all $(y, h) \in  \mathfrak{D}(\*A)$  and $(\zeta, \ell) \in \mathfrak{Z}$. 
This shows that $\mathfrak{Z} \subset \mathfrak{D}(\*A^{*})$. To conclude, we show that $\mathfrak{D}(\*A^{*})\subset\mathfrak{Z}$. 
To this end, pick an arbitrary $(\zeta,\ell)\in\mathfrak{D}(\*A^{*})$. 
Then, there exists $(f, g) \in{\mathscr{H}^{*}}$ such that 
\begin{equation*}
\Big\langle (\lambda{\mathrm{Id}}  - \*A)\big(y, h\big), \big(\zeta, \ell\big)\Big\rangle_{\mathscr{H}, \mathscr{H^{*}}} = \Big\langle \big(y, h\big), \big(f, g\big)\Big\rangle_{\mathscr{H}, \mathscr{H^{*}}} 
\end{equation*}
holds for all $(y, h)  \in \mathfrak{D}(\*A)$.
Let us now set
\begin{equation*}
(\varphi,s):=(\lambda \mathrm{Id} - \*A^{*})^{-1}(f, g) \in \mathfrak{Z}. 
\end{equation*}
Then from the above identities, we may infer that
\begin{align*}
\Big\langle\big(y, h\big), \big(f, g\big)\Big\rangle_{\mathscr{H}, \mathscr{H^{*}}} &= \Big\langle \big(y, h\big),(\lambda \mathrm{Id} -  \*A^{*}) \big(\varphi,s \big)\Big\rangle_{{\mathscr{H}, \mathscr{H^{*}}}} \\
&=   \Big\langle (\lambda{\mathrm{Id}} - \*A)\big(y, h\big), \big(\varphi,s \big)\Big\rangle_{{\mathscr{H}, \mathscr{H^{*}}}},
\end{align*}
holds for all $(y, h)\in\mathfrak{D}(\*A) $.
Therefore, we may also conclude that 
\begin{equation*}
\Big\langle (\lambda{\mathrm{Id}} - \*A) \big(y, h\big), \big(\varphi, s \big) -  \big(\zeta,\ell \big) \Big\rangle_{\mathscr{H}} = 0,
\end{equation*}
for all $(y, h)  \in \mathfrak{D}(\*A) $,
whence $(\zeta,\ell) = (\varphi,s) \in \mathfrak{Z}$.
 This completes the proof. 
\end{proof}

}

\section{Proof of \Cref{thm: frequency.obs}} \label{sec: projected.systems}

	{\color{black}
		The full proof of \Cref{thm: frequency.obs} may be found in \Cref{sec: prooof.prop31}. More specifically, when $n\neq 0$, we use the observation made for \eqref{eq: frequency.by.frequency.obs.2}, and the observability inequality can be shown by means of spectral arguments (based on results presented in \Cref{sec: section31}, which come with some degree of difficulty). 
		For the zeroth mode $n=0$, we shall note that the eigenfunctions of the governing linear operator are not orthogonal (the operator is not self-adjoint), but the system is of cascade type and falls into the setting of \cite{geshkovski2021controllability}. The estimate of $\mathcal{K}(T,\sigma)$ for the zeroth mode comes from an adapted moment argument.
		}
		
		\subsection{The spectrum of $\*A_n$} \label{sec: section31}
		Recall that, when $n\in \Z\setminus\{0\}$, the operator $\*A_n$ is self-adjoint due to the specific inner product we endowed to $\mathscr{H}_{n}$ (see \Cref{prop:resolvent-An}). And since $\*A_n$ has compact resolvents, by the Hilbert-Schmidt theorem, it may be diagonalized to find an orthonormal basis of $\mathscr{H}_{n}$ consisting of eigenfunctions of $\*A_n$, associated to a decreasing sequence of eigenvalues. 
		
		{\color{black} 
		\begin{lemma}[The spectrum of $\*A_{0}$] \label{lem-spec-A0}
		The spectrum of $\* A_{0}$ is contained in $(-\infty,0]$. 
		\end{lemma}
		
		\begin{proof}
		Let $\lambda \in \mathbb{C},$ and we consider the eigenvalue problem:
			\begin{equation*} 
			\begin{cases}
			\lambda \varphi_{0} - \del_{x_2}^2 \varphi_0 = 0 &\text{ in }(-1,1), \\
			\lambda \ell_{0} = \del_{x_2}\varphi_0(1) \\
			\varphi_0(\pm1) = 0.
			\end{cases}
			\end{equation*}
			The conclusion follows by multiplying the first equation by $\varphi_0$ and integrating by parts.
		\end{proof}
		

		}
		
		To prove the observability of \eqref{eq: fourier.projected.system.intro} when $n\neq0$ by using spectral arguments, we need to explicitly characterize the spectrum of $\*A_n$.
		This is the goal of the following result.
		
		\begin{lemma}[The spectrum of $\*A_n$] \label{lem: spectrum}
			Let $\sigma>0$ and $n \in \Z\setminus\{0\}$ be fixed. Then, the sequence $\{ \lambda_{n, k} \}_{k=0}^{+\infty}$, with $\lambda_{n,k}<0$, of eigenvalues of $\*A_n$, reads as follows:
			\begin{equation*}
			\{\lambda_{n,k}\}_{k=0}^{+\infty}=\left\{-\Bigg(\frac{(k+1)\pi}{2}+\frac{\pi}{4}-\varepsilon_{k+1}\Bigg)^2-n^2\right\}_{k=0}^{+\infty}
			\end{equation*}
			if $\sigma>2$, for some $\varepsilon_k\in\left(0,\frac{\pi}{4}\right)$, and
			\begin{equation*}
			\{ \lambda_{n, k} \}_{k=0}^{+\infty}=\left\{-\Bigg(\frac{k\pi}{2}+\frac{\pi}{4}-\varepsilon_k\Bigg)^2-n^2\right\}_{k=1}^{+\infty} \bigcup\,
			\begin{cases}
			\nu_0^2-n^2 &\text{ if } \sigma<2 \\
			-n^2 &\text{ if } \sigma=2
			\end{cases}
			\end{equation*}
			otherwise, for some $\nu_0\in\left(0,\frac{1}{\sigma}\right)$ independent of $k$\footnote{{\color{black}The constant $\nu_0$ might depend on $n$, but this dependence is at worst benign: the quantities of interest, namely the separation constant $\gamma$ in \eqref{eq: regular.eigs} and the lower bound $C$ in \eqref{eq: eigenfunc.lower.bound}, do not depend on $n$. Furthermore, should $\nu_0$ depend on $n$, it ought to converge to some value $\nu_0^\star\in\left[0,\frac{1}{\sigma}\right]$, which again implies nothing of relevance vis-à-vis the conclusions of \Cref{lem: spectrum}.}}. (Here, we imply that the second "set" in the union is indexed as $k=0$.)
			
			Furthermore, the following properties hold.
			\begin{enumerate}
			\item
			The sequence $\{ \lambda_{n, k} \}_{k=0}^{+\infty}$ is separated uniformly with respect to $n \in \Z\setminus\{0\}$, in the sense that
			\begin{equation} \label{eq: regular.eigs}
				\inf_{k\geqslant0} \Big|\lambda_{n, k+1}-\lambda_{n, k}\Big|>\gamma
			\end{equation} 
			holds for some $\gamma>0$ independent of $n$.
			\smallskip
			\item
			Moreover,
			\begin{equation} \label{eq: eigenvals}
				-\lambda_{n, k} = rk^2 + n^2 + \underset{k\rightarrow+\infty}{\mathcal{O}}(k)
			\end{equation}
			for some $r>0$ independent of $n$ and $\sigma$.
			\smallskip
			
			\item
			Suppose that $(c,d)\subset(-1,1)$ is fixed, with $c<d$. Then, there exists some constant $C=C(\sigma,c-d)>0$ such that for any $n \in \Z\setminus\{0\}$ and $k \geqslant0$, the normalized eigenfunctions 
			\begin{equation*}
			\Phi_{n, k}:=\left(\varphi_{n, k},\, -\frac{1}{\sigma n^2}\varphi_{n,k}(1)\right)
			\end{equation*}
			of $ \*A_n$ are such that
			\begin{equation} \label{eq: eigenfunc.lower.bound}
				\| \varphi_{n, k}\|_{L^2\left(c,d\right)} \geqslant C
			\end{equation}
			holds.
			\end{enumerate}
		\end{lemma}
		
		Before proceeding with the proof of \Cref{lem: spectrum}, let us make the following important observation.

\begin{lemma}[Spectral gap] \label{rem: spectral.gap}
Suppose $\sigma>0$ and $n\in\mathbb{Z}\setminus\{0\}$. Then $\*A_n$ has a spectral gap -- namely, 
\begin{equation*}
\lambda\leqslant-\min\left\{\frac{\sigma}{2}, 1\right\}n^2 
\end{equation*}
for all $\lambda\in\mathrm{spec}(\*A_n)$. 
\end{lemma}

In particular, $\lambda\leqslant-n^2$ if $\sigma\geqslant2$.
We shall make use of this property in the proof of \Cref{lem: spectrum}, in particular distinguishing the cases $\sigma\in(0,2)$ and $\sigma\geqslant2$, for convenience.
The factor $2$ in the denominator refers to the Lebesgue measure of the interval $[-1,1]$, so more general intervals can be envisaged.

\begin{proof}[Proof of \Cref{rem: spectral.gap}]
Let $\lambda<0$ be such that 
\begin{equation} \label{eq: eig.system}
\begin{cases}
\lambda y -\del_{x_2}^2 y + n^2 y = 0 &\text{ in }(-1,1),\\
\lambda h -\del_{x_2}y(1) = 0,\\
y(-1) = 0,\\
y(1)=-\sigma n^2 h,
\end{cases}
\end{equation}
holds for some $(y,h)\in\mathfrak{D}(\*A_n)\setminus\{(0,0)\}$. This entails the identity
\begin{equation} \label{eq: identity.eig}
-\lambda\left(\int_{-1}^1 |y|^2\diff x_2 + \sigma n^2 h^2\right) = \int_{-1}^1\left|\del_{x_2} y\right|^2\diff x_2 + n^2 \int_{-1}^1|y|^2\diff x_2.
\end{equation}
Now note that from the boundary condition $y(1)=-\sigma n^2 h$ and an elementary Sobolev embedding\footnote{We simply use $y(x)=y(x)-y(-1)=\int_{-1}^x\del_{z}y(z)\diff z$ combined with the Cauchy-Schwarz inequality. The factor $2$ occurs as $\mathrm{meas}([-1,1])=2$.} for solutions to \eqref{eq: eig.system}, we derive
\begin{equation}\label{eq: sobolev.used}
\sigma^2 n^4 h^2 = |y(1)|^2\leqslant 2\int_{-1}^1\left|\del_{x_2} y\right|^2\diff x_2.
\end{equation}
Plugging \eqref{eq: sobolev.used} in \eqref{eq: identity.eig}, we find
\begin{align*}
-\lambda\left(\int_{-1}^1 |y|^2\diff x_2 + \sigma n^2 h^2\right)&\geqslant\frac{\sigma}{2}\sigma n^4 h^2 + n^2\int_{-1}^1 |y|^2\diff x_2\\
&\geqslant\min\left\{\frac{\sigma}{2},1\right\} n^2\left(\sigma n^2 h^2+\int_{-1}^1 |y|^2\diff x_2 \right),
\end{align*}
as desired. 
\end{proof}
		
		\begin{proof}[Proof of \Cref{lem: spectrum}]
			
			We recall that $\*A_n: \mathfrak{D}(\*A_n)\to \mathscr{H}_{n}$ is self-adjoint, has compact resolvents, and its spectrum consists of a decreasing sequence of negative eigenvalues, namely a sequence $\{\lambda_{n,k}\}_{k=0}^{+\infty}$ with $-\infty < \ldots \leqslant \lambda_{n,k} \leqslant \ldots \leqslant \lambda_{n,0} < 0$. 
			We shall distinguish three different scenarios for the computation of these eigenvalues.
			\smallskip
			
		\noindent
		\textbf{Case 1: \fbox{$\lambda < -n^2$.}} 
			Suppose that $\lambda\in(-\infty,0)$ is an eigenvalue of $\*A_n$ which also satisfies $\lambda<-n^2$. So, there must exist a vector $(\varphi,\ell)\in \mathfrak{D}(\*A_n)\setminus\{(0,0)\}$ such that
			\begin{equation} \label{eq: eigensystem}
			\begin{cases}
			\del_{x_2}^2 \varphi+(-\lambda-n^2)\varphi = 0 &\text{ in }(-1,1), \\
			\del_{x_2}\varphi(1) = \lambda\ell\\
			\varphi(-1) = 0 \\
			\varphi(1) = -\sigma n^2\ell.
			\end{cases}
			\end{equation}
			In other words, $\varphi$ would solve the mixed Dirichlet-Robin problem
			\begin{equation} \label{eq: eigensystem.2}
			\begin{cases}
			\del_{x_2}^2 \varphi+(-\lambda-n^2)\varphi = 0 &\text{ in } (-1,1),\\
			\varphi(-1) = 0 \\
			\varphi(1)+\frac{\sigma n^2}{\lambda} \del_{x_2}\varphi(1)=0.
			\end{cases}
			\end{equation}
			Since $-\lambda-n^2>0$, one may readily see that the solutions to \eqref{eq: eigensystem.2} are of the form
			\begin{equation*}
			\varphi(x_2) = C\sin\Big(\nu(1+x_2)\Big),
			\end{equation*} 
			with $C>0$, where $\nu:=\sqrt{-\lambda-n^2}$ is the positive root of the transcendental equation 
			\begin{equation} \label{eq: transcendental.1}
			\left(\frac{\nu^2}{n^2}+1\right)\tan(2\nu) = \sigma\nu.
			\end{equation}
			Locating the positive roots of this equation suggest a study of the fixed points of the function $f(\nu) = \left(\frac{\nu^2}{n^2}+1\right)\tan(2\nu)$, defined and increasing on the union of consecutive intervals of the form
			\begin{equation*}
			\bigcup_{k=1}^{+\infty} \left(\frac{\pi}{4}+\frac{(k-1)\pi}{2}, \,\frac{\pi}{4} + \frac{k\pi}{2}\right).
			\end{equation*}
			
			\begin{figure}[h!]
			\centering
				\includegraphics[scale=0.42]{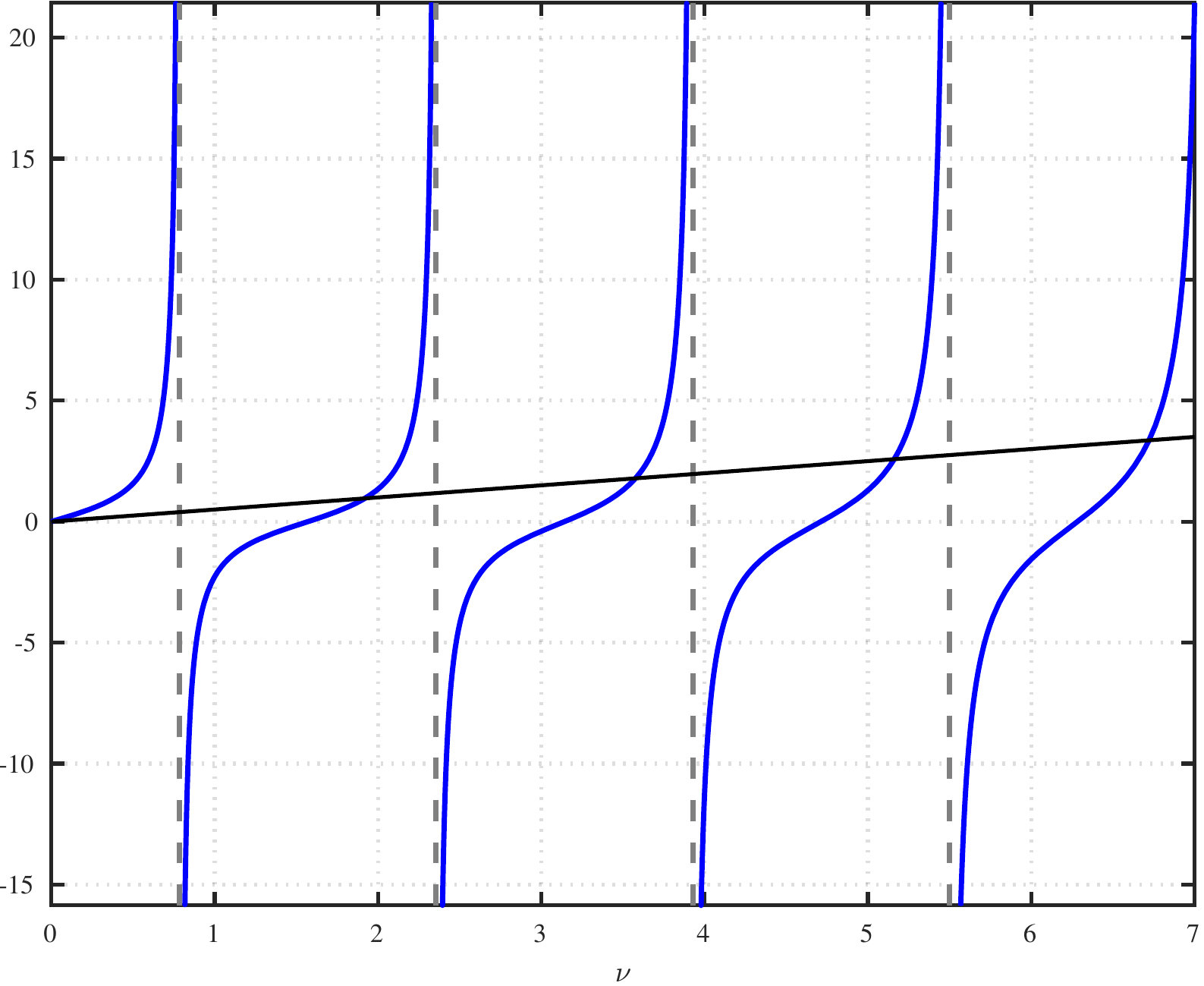}
				\hspace{0.15cm}
				\includegraphics[scale=0.42]{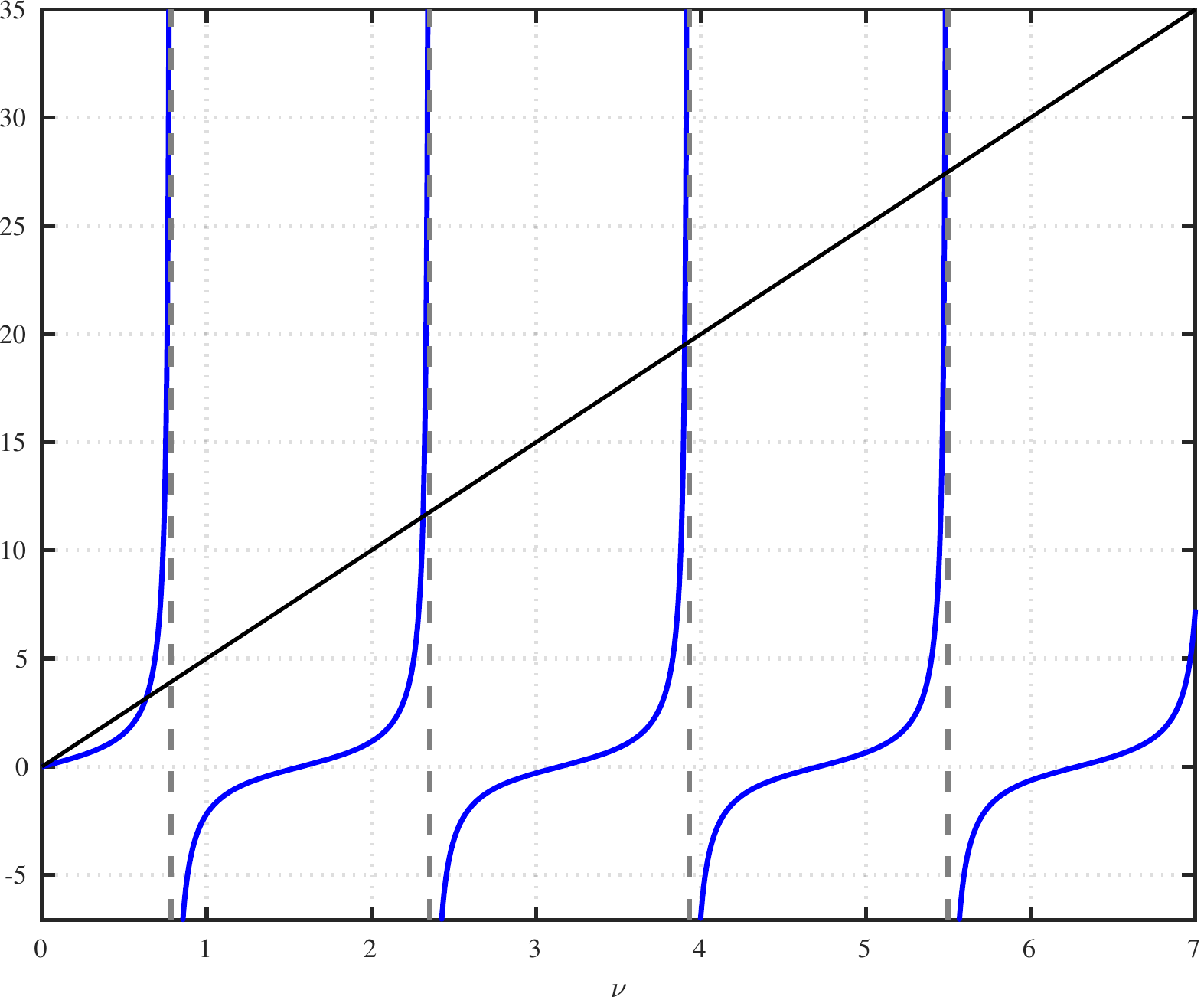}
				\caption{The function $f(\nu) =  \left(\sfrac{\nu^2}{n^2}+1\right)\tan(2\nu)$ (\emph{blue}), with $\nu\mapsto\sigma x$ superposed (\emph{black}), with $n=10^4$ and $\sigma=\sfrac12$ (\emph{left}) and $\sigma=5$ (\emph{right}). We see how the fixed points of $f$ are localized over each subinterval.}
			\end{figure}
			
			\noindent
			Moreover, for $k \geqslant 1$,
			\begin{equation*}
			\lim_{\nu\searrow \frac{\pi}{4}+\frac{(k-1)\pi}{2} } f(\nu) = -\infty, \quad f\left(\frac{k\pi}{2}\right) = 0, \quad \lim_{\nu \nearrow \frac{\pi}{4} + \frac{k\pi}{2}} f(\nu) = +\infty.
			\end{equation*}
			Thus, \eqref{eq: transcendental.1} has a sequence of positive roots $\{\nu_k\}_{k=1}^{+\infty}$ of the form
			\begin{equation} \label{eq: nu_k}
			\nu_k = \frac{k\pi}{2}+ \frac{\pi}{4}-\varepsilon_k
			\end{equation}
			for $k\geqslant 1$, where $\varepsilon_k \in \left(0, \frac{\pi}{4}\right)$ may a priori depend on $\sigma$ and $n$. 
			Consequently, the eigenvalues $\lambda_{n,k}<0$ in this case are of the form
			\begin{equation} \label{eq: lambda.k}
			-\lambda_{n,k} = \left(\frac{k\pi}{2}+ \frac{\pi}{4}-\varepsilon_k\right)^2 +n^2
			\end{equation}
			for $k\geqslant1$.	
			\smallskip
			
			\noindent
			\textbf{Case 2: \fbox{$\lambda=-n^2$.}} We note that $\lambda=-n^2$ is an eigenvalue if and only if $\sigma=2$. Indeed, should $\lambda=-n^2$ be an eigenvalue, then $\varphi$ in \eqref{eq: eigensystem} would be harmonic, and thus an affine function: $\varphi(x_2) = C_0x_2 + C_1$, for some $C_0,C_1\in\R$. Using the boundary conditions, we moreover find that $C_0=C_1$, as well as $2C_0=\lambda\sigma\ell$. Since $\del_{x_2}\varphi(1) =C_0=\lambda\ell$, we are led to $\sigma=2$. Summarizing, we find that when $\lambda=-n^2$,
			\begin{equation*}
			\varphi(x_2) = C_0(1+x_2).
			\end{equation*}
			
			\noindent
			\textbf{Case 3: \fbox{$\lambda \in (-n^2, 0)$.}} Following \Cref{rem: spectral.gap} regarding the spectral gap, this case may only occur if $\sigma<2$. 
			So let us henceforth suppose that $\sigma<2$, and that $\lambda \in (-n^2, 0)$ is an eigenvalue of $ A_{n}$. Consequently, there must exist a vector $(\varphi, \ell) \in \mathfrak{D}(\*A_n)\setminus\{(0,0)\}$ such that \eqref{eq: eigensystem} holds.
			Then $\varphi$ would again solve the mixed Dirichlet-Robin problem \eqref{eq: eigensystem.2}. 
			Since now $-\lambda - n^2 <0$, one may readily see that the solutions to \eqref{eq: eigensystem.2} are now of the form 
			\begin{equation*}
			\varphi(x_2) = C\Big(e^{\nu x_2} - e^{-\nu(2+x_2)}\Big),
			\end{equation*}
			for some $C>0$, where $\nu:=\sqrt{n^2+\lambda}$ denotes the positive root(s) of the transcendental equation
			\begin{equation*}
			e^{\nu}-e^{-3\nu} -\frac{\sigma n^2}{n^2-\nu^2}\Big(\nu e^{\nu} + \nu e^{-3\nu}\Big) = 0
			\end{equation*}
			in $(0, |n|)$. We may equivalently rewrite the above equation as
			\begin{equation} \label{eq: f(nu)}
			\Big(-\nu^2-\sigma n^2\nu+n^2\Big)+\Big(\nu^2-\sigma n^2 \nu-n^2\Big) e^{-4\nu} = 0.
			\end{equation}
			We designate \eqref{eq: f(nu)} as $f(\nu)=0$, and we claim that $f$ has a unique root\footnote{One may try to compute this root by using the Lambert $W$ function and its generalizations (\cite{mezHo2017generalization}). We omit this from our work as it is not needed for our analysis.} in $(0, |n|)$. Let us henceforth focus on proving this claim. Existence follows from the fact that $f$ is increasing and positive near $0$, and decreasing and negative near $|n|$. To ensure uniqueness, 
			 we will look to show that $f$ is strictly concave in $\big[0,|n|\big]$ (see \Cref{fig.zoom}).
			We shall designate
			\begin{equation*}
			f_1(\nu):=-\nu^2-\sigma n^2\nu+n^2
			\end{equation*} 
			and 
			\begin{equation*}
			f_2(\nu):=\Big(\nu^2-\sigma n^2 \nu-n^2\Big) e^{-4\nu},
			\end{equation*} 
			so that $f=f_1+f_2$. 
			We see that 
			\begin{equation} \label{eq: f1''}
			f_1''(\nu) = -2.
			\end{equation}
			On another hand, we also have 
			\begin{equation} \label{eq: f2''}
			f_2''(\nu) = \Big(2-16\nu+8\sigma n^2+16\nu^2-16\sigma n^2\nu-16n^2\Big)e^{-4\nu}.
			\end{equation}
			As $f=f_1+f_2$, and taking \eqref{eq: f1''} into account, we see that it suffices to ensure that 
			$f_2''(\nu)<2$ in $[0,|n|]$. This may be shown without too much difficulty. 
			Indeed, whenever $\nu\in\left[0,\frac{1}{2}\right]$ we see that $8\sigma\nu^2-16n^2<0$ as well as $16\nu^2-16\nu<0$ in \eqref{eq: f2''}. 
			Similarly, whenever $\nu\in\left[\frac12,|n|\right]$, we find that $8\sigma n^2-16\sigma n^2\nu<0$ as well as $16\nu^2-16n^2<0$ in \eqref{eq: f2''}. This yields $f_2''(\nu)<2$ in $[0,|n|]$, whence $f$ is strictly concave in $[0,|n|]$. Consequently, it follows that $f$ has at most $2$ roots in $[0,|n|]$. (This elementary property is readily shown by arguing by contradiction and Rolle's theorem.) 
			And since $f(0)=0$, we conclude that any root of $f$ in $(0,|n|)$ is unique. 
			On another hand, we may readily see that
			\begin{equation*}
			f_2(\nu)<0 \hspace{1cm} \text{ for } \nu\in\big[0,|n|\big],
			\end{equation*}
			as well as 
			\begin{equation*}
			f_1(\nu)\leqslant 0 \hspace{1cm} \text{ for } \nu\in\left[\frac12 \left(\sqrt{n^4\sigma^2+4n^2}-n^2\sigma\right), |n|\right).
			\end{equation*}
			Consequently we find that
			\begin{equation*}
			f(\nu)<0 \hspace{1cm} \text{ for } \nu\in\left[\frac12 \left(\sqrt{n^4\sigma^2+4n^2}-n^2\sigma\right), |n|\right),
			\end{equation*}
			and therefore the unique root $\nu_0$ of $f$ must be located in the complement of the above interval, namely $\nu_0\in\left(0,\frac12 \left(\sqrt{n^4\sigma^2+4n^2}-n^2\sigma\right)\right)$. 
			Hence, in this case, the first eigenvalue $\lambda_{n,0}<0$ of $\*A_n$ will have the form
			\begin{equation} \label{eq: lambda.zero}
			\lambda_{n,0} = \nu_0^2 - n^2.
			\end{equation} 
			Note that, moreover,
			\begin{equation*}
			\nu_0 \leqslant \frac12\left(\sqrt{n^4\sigma^2+4n^2}-n^2\sigma\right) = \frac{4n^2}{2\left(\sqrt{n^4\sigma^2+4n^2}+n^2\sigma\right)} \leqslant \frac{1}{\sigma},
			\end{equation*}
			for all $n \in \Z\setminus\{0\}$.
			
			\begin{figure}[h!]
			\includegraphics[scale=0.385]{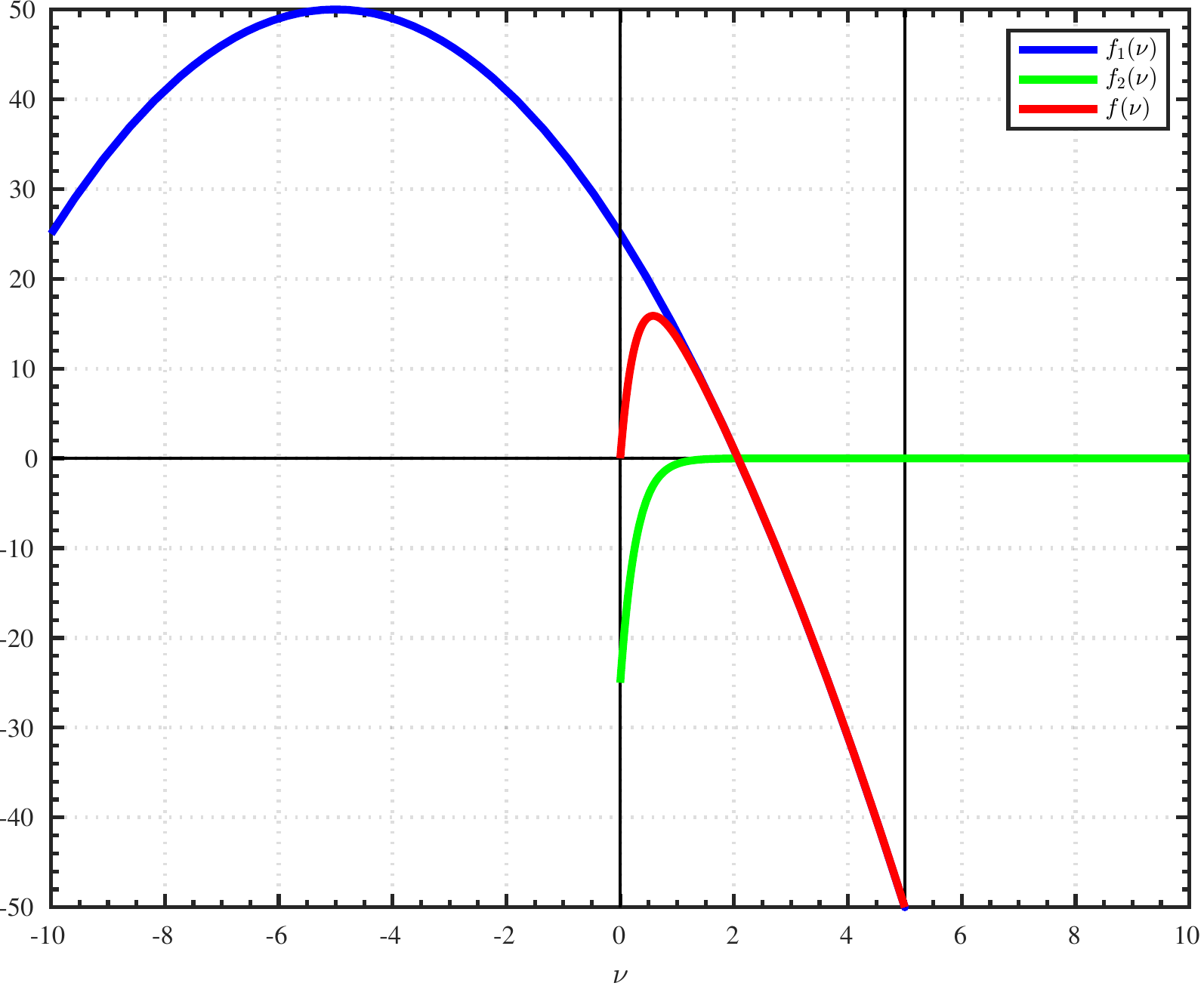}
			\hspace{0.15cm}
			\includegraphics[scale=0.385]{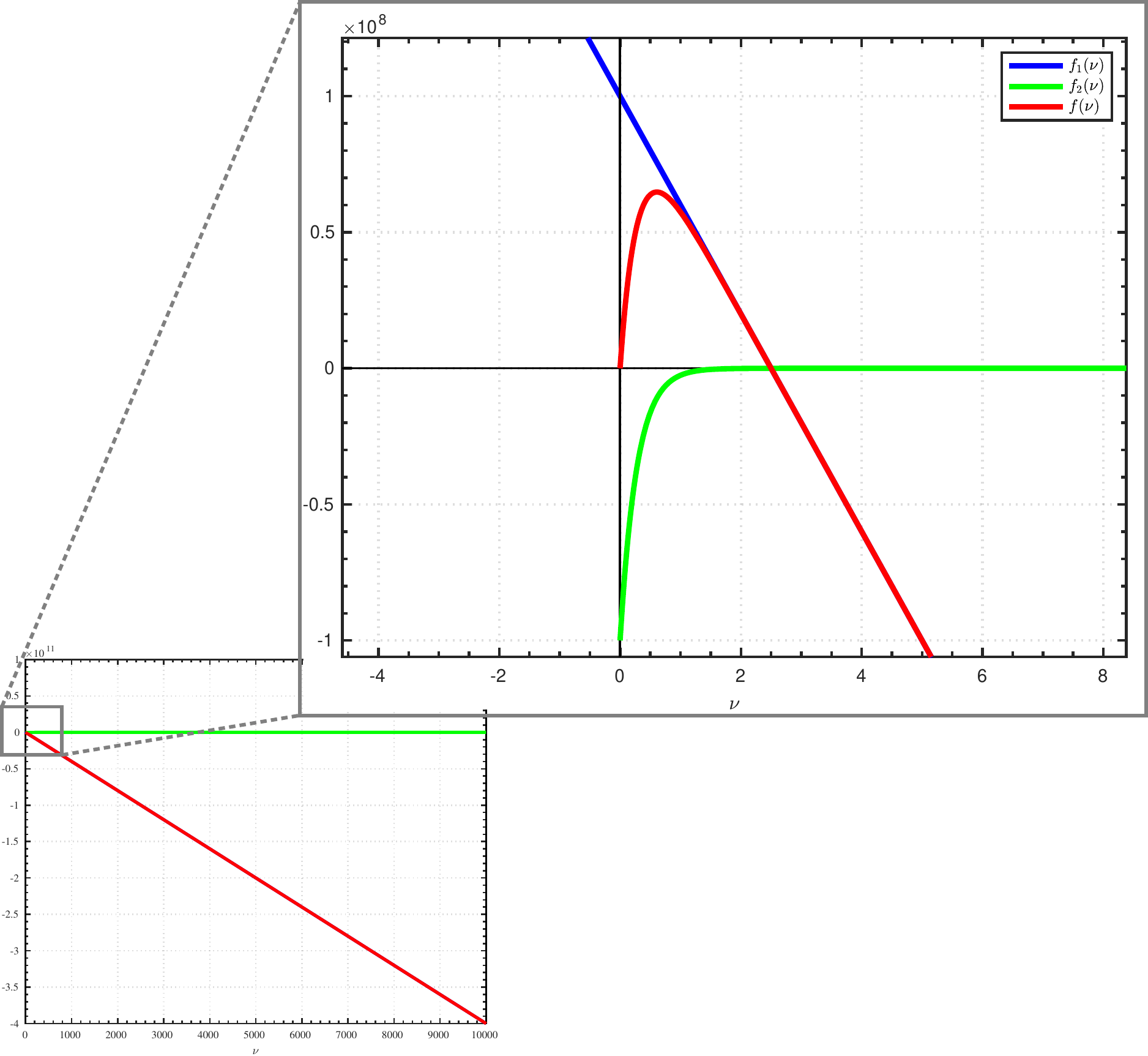}
			\caption{Plot of $f_1$ (\emph{blue}), $f_2$ (\emph{green}), and $f=f_1+f_2$ (\emph{red}), for $\sigma=\frac{2}{5}$, with $n=5$ (\emph{left}) and $n=10^4$ (\emph{right}). We see that $f$ has a unique root $\nu_0\in(0,|n|)$, and this root seemingly does not collapse to $0$ as $n\to+\infty$.}
			\label{fig.zoom}
			\end{figure}
			
			Having analyzed all the possible scenarios, we collect the sequence of eigenvalues 
			$\{\lambda_{n,k}\}_{k=0}^{+\infty}$, which if $\sigma>2$ may actually be indexed as $\{\lambda_{n,k}\}_{k=1}^{+\infty}$ with $\lambda_{n,k}$ with $k\geqslant1$ defined in \eqref{eq: lambda.k}; if $\sigma=2$, we also have $\lambda_{n,0}=-n^2$, and if $\sigma\in(0,2)$, we have $\lambda_{n,0}<0$ defined in \eqref{eq: lambda.zero}. 
			One thus readily sees that \eqref{eq: eigenvals} holds. 
			On another hand, since $\varepsilon_k \in\left(0,\frac{\pi}{4}\right)$, we see that for $k \geqslant 1$,
			\begin{align*}
			-\lambda_{n,k+1} + \lambda_{n,k} &= \Bigg(\frac{(k+1)\pi}{2}+\frac{\pi}{4}-\varepsilon_{k+1}\Bigg)^2 - \Bigg(\frac{k\pi}{2}+\frac{\pi}{4}-\varepsilon_{k}\Bigg)^2 \\
			&= \Bigg(k\pi + \pi - \varepsilon_{k+1} - \varepsilon_k\Bigg)\Bigg(\frac{\pi}{2} - \varepsilon_{k+1} + \varepsilon_k\Bigg) \\
			&\geqslant\frac{3\pi}{2}\cdot\frac{\pi}{4} = \frac{3\pi^2}{8}.
			\end{align*}
			Furthermore, if $\sigma\in(0,2)$, 
			\begin{align*}
			-\lambda_{n,1} + \lambda_{n,0} = \Bigg(\frac{\pi}{2} + \frac{\pi}{4} - \varepsilon_k\Bigg)^2+ \nu_0^2 \geqslant \frac{\pi^2}{4},
			\end{align*}
			whereas if $\sigma=2$, we similarly find
			\begin{align*}
			-\lambda_{n,1} + \lambda_{n,0} = \Bigg(\frac{\pi}{2} + \frac{\pi}{4} - \varepsilon_k\Bigg)^2 \geqslant \frac{\pi^2}{4}.
			\end{align*}
			Hence, the separation condition \eqref{eq: regular.eigs} holds as well (actually, also uniformly in $\sigma>0$).
			
			Let us finally prove \eqref{eq: eigenfunc.lower.bound}. 
			We recall that the normalized eigenfunctions $\Phi_{n, k}$ have the form
			\begin{equation*}
				\Phi_{n,k} = \left(
				\varphi_{n,k},\, -\frac{1}{\sigma n^2}\varphi_{n,k}(1)
				\right),
			\end{equation*}
			where $\varphi_{n,k}$ is given by
			\begin{equation*}
				\varphi_{n,k}(x_2) = C_{n,k} \sin\left(\sqrt{-\lambda_{n,k}-n^2}\, (1+x_2)\right)
			\end{equation*}
			for $k\geqslant1$ (following the indexing of the eigenvalues depending on where $\sigma$ is located), while 
			\begin{equation*}
			\varphi_{n,0}(x_2) = C_{n,0} \left(e^{\sqrt{n^2+\lambda_{n,0}}\, x_2} - e^{-\sqrt{n^2+\lambda_{n,0}}(2+x_2)} \right)
			\end{equation*}
			if $\sigma\in(0,2)$, and 
			\begin{equation*}
			\varphi_{n,0}(x_2) = C_{n,0}\Big(1+x_2\Big)
			\end{equation*}
			if $\sigma=2$.
			Let us first suppose $k\geqslant 1$. Reusing the notation $\nu_k:= \sqrt{-\lambda_{n,k} - n^2}>0$, we note that in order to ensure that the eigenfunctions $\Phi_{n, k}$ are of norm $1$ in $\mathscr{H}_{n}$, namely $\|\Phi_{n,k}\|_{\mathscr{H}_{n}}=1$, we see that $C_{n,k}>0$ needs to satisfy
			\begin{equation*}
				C_{n,k}^2 \left(1-\frac{\sin(4\nu_k)}{4\nu_k}\right) + \frac{C_{n,k}^2}{n^2\sigma} \sin^2(2\nu_k) = 1
			\end{equation*}
			for all $k \geqslant 1$. 
			On the other hand, using elementary trigonometric identities, we may also find
			\begin{equation} \label{eq: expr.bphik}
				\|\varphi_{n,k} \|_{L^2\left(c,d\right)}^2 = \frac{C_{n,k}^2}{2} \left((d-c)+\frac{\sin\left(2\nu_k(c+1)\right)}{2\nu_k}-\frac{\sin\left(2\nu_k(d+1)\right)}{2\nu_k}\right).
			\end{equation}
			In view of \eqref{eq: nu_k}, namely the asymptotics of $\nu_k$ when $k\to+\infty$, we find that there exists $\delta=\delta(d-c)>0$ independent of $n \in \Z\setminus\{0\}$ and $k$ such that 
			\begin{equation} \label{eq: expr.bphik2}
			(d-c)+\frac{\sin\left(2\nu_k(c+1)\right)}{2\nu_k}-\frac{\sin\left(2\nu_k(d+1)\right)}{2\nu_k}\geqslant\delta
			\end{equation}
			holds for all $k\geqslant1$.
			Therefore, we see from \eqref{eq: expr.bphik} and \eqref{eq: expr.bphik2} that, in order to obtain \eqref{eq: eigenfunc.lower.bound}, it suffices to have an appropriate lower bound on $C_{n,k}$ for all $k \geqslant 1$. 
			To this end, we note that
			\begin{equation*}
				C_{n,k}^2 = \left( \left(1-\frac{\sin(4\nu_k)}{4\nu_k}\right) + \frac{\sin^2(2\nu_k)}{n^2\sigma} \right)^{-1}.
			\end{equation*}
			By virtue of \eqref{eq: nu_k}, we see that 
			\begin{equation*}
				\frac{\sin(4\nu_k)}{4\nu_k} \underset{k\rightarrow+\infty}{\longrightarrow} 0 \hspace{1cm} \text{ and } \hspace{1cm} \sin^2(2\nu_k) \underset{k\rightarrow+\infty}{\longrightarrow} 0, 
			\end{equation*}
			hence 
			\begin{equation*}
				C_{n,k}^2 \geqslant C_*
			\end{equation*}
			for some $C_*>0$ independent of $n, k$ and $\sigma$. This concludes the proof of \eqref{eq: eigenfunc.lower.bound} when 
			$k\geqslant1$.
			
			Let us now consider the case $k=0$ and $\sigma\in(0,2]$. First suppose that $\sigma\in(0,2)$. 
			We see that to ensure orthonormality, $C_{n,0}>0$ needs to satisfy
			\begin{equation*}
			C_{n,0}^2 \,\frac{\sinh(2\nu_0)-4\nu_0}{\nu_0} e^{-2\nu_0} + \frac{C_{n,0}^2}{n^2\sigma} \Big(e^{\nu_0} - e^{-3\nu_0}\Big)^2=1,
			\end{equation*}
			thus 
			\begin{equation} \label{eq: c02}
			C_{n,0}^2 = \left(\frac{\sinh(2\nu_0)-4\nu_0}{\nu_0} e^{-2\nu_0} + \frac{\Big(e^{\nu_0} - e^{-3\nu_0}\Big)^2}{n^2\sigma}\right)^{-1}.
			\end{equation}
			Taking \eqref{eq: c02} into stock, we see that since $\nu_0\in\left(0,\frac{1}{\sigma}\right)$ and $\frac{\sinh(2x)}{x}$ is positive and continuous for $x \in \left[0, \frac{1}{\sigma}\right]$, 
			\begin{equation*}
			\frac{\sinh(2\nu_0)-4\nu_0}{\nu_0} e^{-2\nu_0} + \frac{\Big(e^{\nu_0} - e^{-3\nu_0}\Big)^2}{n^2\sigma}\leqslant\frac{\sinh(2\nu_0)}{\nu_0}+\frac{1}{\sigma}e^{\frac{1}{\sigma}}\leqslant C_1(\sigma)+\frac{1}{\sigma}e^{\frac{1}{\sigma}}
			\end{equation*}
			holds for some $C_1(\sigma)>0$ independent of $n$. Consequently, 
			\begin{equation}\label{eq: Cn0}
			C_{n,0}^2\geqslant\frac{1}{C_1(\sigma)+\frac{1}{\sigma}e^{\frac{1}{\sigma}}}.
			\end{equation}
			We also have
			\begin{equation*}
			\|\varphi_{n,0}\|^2_{L^2\left(c,d\right)} = C_{n,0}^2\, \frac{\Big(-2\nu_0(d-c)-\sinh\big(2(c+1)\nu_0\big)+\sinh\big(2(d+1)\nu_0\big)\Big)}{\nu_0} e^{-2\nu_0}.
			\end{equation*}
			Similarly, using the continuity and the positivity of the function 
			\begin{equation*}
			x \mapsto \frac{\sinh(2(d+1)x)-\sinh(2(c+1)x)-2(d-c)x}{x} 
			\end{equation*}
			on $\left(0, \frac{1}{\sigma}\right)$, and using \eqref{eq: Cn0}, we conclude that there exists $C(\sigma,d-c)>0$ independent of $n$ such that
			\begin{equation*}
			\|\varphi_{n,0}\|^2_{L^2\left(c,d\right)} \geqslant C(\sigma,d-c),
			\end{equation*}
			holds for all $n\in\Z\setminus\{0\}$. This is precisely \eqref{eq: eigenfunc.lower.bound}.
			
			Finally, in the case $k=0$ and $\sigma=2$, since the (normalized) eigenfunction is an affine function, the proof of \eqref{eq: eigenfunc.lower.bound} is straightforward. 
			This concludes the proof.
		\end{proof}

	\subsection{Proof of \Cref{thm: frequency.obs}} \label{sec: prooof.prop31}
		We may conclude the study of the family of projected systems. 
							
		\begin{proof}[Proof of \Cref{thm: frequency.obs}]
			
		We distinguish two separate cases.
		\smallskip
		
		{\color{black}
		\noindent
		\textbf{Case 1:} $n=0.$
		By virtue of the results in \cite{geshkovski2021controllability, Er19}, we know that \eqref{eq: frequency.by.frequency.obs} holds when $n=0$.
		The proof of this result consists in looking at the dual controllability problem: since $n=0$, the resulting control system is of cascade type, with 
		\begin{equation} \label{eq: expression.h}
		h_0(t) = h^0_0 + \int_0^t \del_{x_2}y_0(s, 1)\diff s,
		\end{equation}
		with $y_0$ solving the linear heat equation with Dirichlet boundary conditions.
		Controlling $h_0(t)$ to $0$ in time $T$ can then be seen as a one-dimensional constraint on the heat control found by HUM (\cite{lions1988exact}), and may be achieved by a compactness-uniqueness argument within the global Carleman inequality for the heat operator. We refer to \cite{geshkovski2021controllability, Er19} for all the necessary details.
		
		The monotonicity properties of the constant ${\mathcal{K}}(T,\sigma)$ are presented in \cite{miller2010direct}. But due to the compactness argument, the proof of controllability does not yield an explicit exponential bound of the form \eqref{eq:  zero.mode.estimate.cost}, which is of essence for using the source term method argument. 
		We present a brief moment method argument which yields the desired bound. 
		Let $\phi_j(x_2) = \sin(j\pi x_2)$ and $\lambda_j=j^2\pi^2$ denote the $j$-th eigenfunction and eigenvalue of the Dirichlet Laplacian on $(-1,1)$ for $j\geqslant1$.
		We have
		\begin{equation*}
		y_0(t, x_2) = \sum_{j=1}^{+\infty}\left(e^{-\lambda_j t}\left\langle y^0_0, \phi_j\right\rangle_{L^2(-1,1)} + \int_0^t e^{-\lambda_j(t-s)}\langle u_0(s), \phi_j\rangle_{L^2(c,d)}\diff s\right)\phi_j(x_2).
		\end{equation*}
		We see that the condition $y_0(T)=0$ (which holds true) is equivalent to
		\begin{equation} \label{eq: cond.1}
		\int_0^T e^{\lambda_j s} \langle u_0(s), \phi_j\rangle_{L^2(c,d)}\diff s = -\left\langle y_0^0,\phi_j\right\rangle_{L^2(-1,1)} \hspace{0.5cm} \text{ for all } j\geqslant 1.
		\end{equation}
		By a similar argument, using the expression for $h_0(t)$ in \eqref{eq: expression.h}, we see that the condition $h_0(T)=0$ (which also holds true) is equivalent to
		\begin{align*}
		&\sum_{j=1}^{+\infty} \Bigg(\left\langle y_0^0,\phi_j\right\rangle_{L^2(-1,1)}\int_0^T e^{-\lambda_j t}\diff t+\int_0^T \int_0^t e^{-\lambda_j(t-s)}\langle u_0(s), \phi_j\rangle_{L^2(c,d)}\diff s\diff t\Bigg)\partial_{x_2}\phi_j(1) = -h^0_0.
		\end{align*}
		By integrating by parts in the above identity, we find
		\begin{align*}
		\sum_{j=1}^{+\infty}\frac{\partial_{x_2}\phi_j(1)}{\lambda_j}\Bigg(&\left\langle y_0^0, \phi_j\right\rangle_{L^2(-1,1)}\left(1-e^{-\lambda_jT}\right)\\
		&+\int_0^T\langle u_0(t),\phi_j\rangle_{L^2(c,d)}\diff t - e^{-\lambda_jT}\int_0^T e^{\lambda_j s}\langle u_0(s),\phi_j\rangle_{L^2(c,d)}\diff s\Bigg) = -h_0.
		\end{align*}
	Then using \eqref{eq: cond.1} in the above identity, we find 
		\begin{equation} \label{eq: cond.2}
		\sum_{j=1}^{+\infty} \frac{\partial_{x_2}\phi_j(1)}{\lambda_j}\int_0^T\langle u_0(t),\phi_j\rangle_{L^2(c,d)}\diff t = -h^0_0 - \sum_{j=1}^{+\infty}\frac{\partial_{x_2}\phi_j(1)}{\lambda_j}\left\langle y_0^0,\phi_j\right\rangle_{L^2(-1,1)}.
		\end{equation}
		Now let $\{\theta_j\}_{j=0}^{+\infty}$ be a biorthogonal sequence to $\left\{e^{(\lambda_j+1) \cdot}\right\}_{j=0}^{+\infty}$ in $L^2\left(-\frac{T}{2},\frac{T}{2}\right)$, where we've set $\lambda_0:=0$. (The existence of such a sequence is classical, see \cite[Section 3]{fattorini1971exact} and also \cite{tucsnaktenenbaum}).  
		Then we see that we can take $u_0$ as 
		\begin{equation*}
		u_0(t,x_2):= v_0 \theta_0\left(t-\frac{T}{2}\right) e^t + \sum_{j=1}^{+\infty} v_j(x_2) e^{-(\lambda_j+1)\frac{T}{2}} \theta_j\left(t-\frac{T}{2}\right)e^t, 
		\end{equation*}
		where $v_0$ and $v_j$ are such that \eqref{eq: cond.1} and \eqref{eq: cond.2} hold for this $u_0$, namely:
		\begin{equation} \label{eq: v0}
		v_0:=\frac{\displaystyle-h^0_0-\sum_{j=1}^{+\infty}\partial_{x_2}\phi_j(1){\lambda_j}^{-1}\left\langle y^0_0,\phi_j\right\rangle_{L^2(-1,1)}}{\displaystyle\sum_{j=1}^{+\infty} \partial_{x_2}\phi_j(1){\lambda_j}^{-1}\left\langle 1, \phi_j\right\rangle_{L^2(c,d)}}
		\end{equation}
		and 
		\begin{equation*}
		v_j(x_2) = -\frac{\left\langle y^0_0, \phi_j\right\rangle_{L^2(-1,1)}}{\|\phi_j\|^2_{L^2(c,d)}}\phi_j(x_2).
		\end{equation*}
		Then from well-known estimates on this biorthogonal family (see  \cite[Proof of Thm. 3.4]{tucsnaktenenbaum}), we gather that
		\begin{align} \label{eq: c0.u0}
		\|u_0\|_{C^0([0,T]; L^2(c,d))} \leqslant M_1\left(|v_0|e^{\frac{M_2}{T}} e^T + e^{\frac{M_2}{T}} e^T \sum_{j=1}^{+\infty} \left|\left\langle y_0^0, \phi_j\right\rangle_{L^2(-1,1)}\right| \frac{e^{-(\lambda_j+1)\frac{T}{2}}}{(\lambda_j+1)}\right)
		\end{align}
		holds for some $M_1, M_2>0$ independent of $T$. In view of \eqref{eq: c0.u0} and the form of $v_0$ in \eqref{eq: v0}, and since $\del_{x_2}\phi_j(1)\lambda_j^{-1} = (-1)^j (j\pi)^{-2}$, to conclude the proof of \eqref{eq: zero.mode.estimate.cost}, it suffices to show that the series in the denominator in \eqref{eq: v0} is absolutely convergent to some non-zero constant. This follows from \Cref{lem: convergence.series} when $c\neq\pm d$. {\color{black} Finally, by duality \eqref{eq: frequency.by.frequency.obs} holds for $n =0.$}
		}
		
		\smallskip
		\noindent
		\textbf{Case 2:} $n\neq0$.
{\color{black}
		As discussed in \Cref{sec: strategy.proof}, showing \eqref{eq: frequency.by.frequency.obs} for \eqref{eq: fourier.projected.system.intro} is equivalent to proving \Cref{thm: frequency.obs} in the case $n\neq0$. We thus focus on showing the former; namely, we seek to show that
			\begin{equation} \label{eq: obs.ineq}
			Me^{\frac{M}{T}} \int_0^T \left\|\zeta(t,\cdot)\right\|_{L^2\left(c,d\right)}^2 \diff t \geqslant\left\|\big(\zeta(0, \cdot), \ell(0)\big)\right\|_{\mathscr{H}_{n}}^2
			\end{equation}
			holds for some $M>0$ independent of $n$ and $T$, and for all $(\zeta_T, \ell_T) \in \mathscr{H}_{n}$, where $(\zeta,\ell)$ is the unique solution to \eqref{eq: fourier.projected.system.intro}. (We shall stick to the notation $\ell$ for the second component even for \eqref{eq: fourier.projected.system.intro}, and drop the subscripts $n$, for simplicity.)}
			As the governing operator $\*A_n$ of \eqref{eq: fourier.projected.system.intro} has an orthonormal basis of eigenfunctions $\{\Phi_{n, k}\}_{k=0}^{+\infty}$ and corresponding decreasing sequence of negative eigenvalues $\{-\lambda_{n,k}\}_{k=0}^{+\infty}$, we may write the Fourier series decomposition of $\zeta$ as 
			\begin{equation*}
			\zeta(t,x_2) = \sum_{k=0}^{+\infty} e^{-\lambda_{n,k}(T-t)} \Big\langle\big(\zeta_T, \ell_T\big), \Phi_{n,k}\Big\rangle_{\mathscr{H}_{n}} \varphi_{n,k}(x_2).
			\end{equation*} 
			Denoting by $\{\psi_j\}_{j=0}^{+\infty}$ the orthonormal basis of $L^2\left(c,d\right)$, and via the shift $T-t\mapsto t$, we obtain 
			\begin{align} \label{eq: 2.8}
			&\int_0^T \left\|\zeta(t, \cdot)\right\|_{L^2\left(c,d\right)}^2 \diff t = \sum_{j=0}^{+\infty} \int_0^T \left| \sum_{k=0}^{+\infty} e^{-\lambda_{n, k} t} \Big\langle\big(\zeta_T, \ell_T\big), \Phi_{n,k}\Big\rangle_{\mathscr{H}_{n}}  \langle\varphi_{n, k}, \psi_j \rangle_{L^2\left(c,d\right)} \right|^2 \diff t.
			\end{align}
			Now, making use of \eqref{eq: regular.eigs} and \eqref{eq: eigenvals}, we deduce from \cite[Cor. 3.6]{tucsnaktenenbaum} that there exists a constant $M>0$ depending only on $r>0$ and $\gamma>0$ such that 
			\begin{equation} \label{eq: marius.inequality}
			Me^{\frac{M}{T}} \int_0^T \left|\sum_{k=0}^{+\infty}a_k e^{-\big(\lambda_{n,k}-n^2\big)t}\right|^2\diff t \geqslant \sum_{k=0}^{+\infty} |a_k|^2 e^{-2\big(\lambda_{n,k}-n^2\big)T}
			\end{equation}
			for any $\{a_k\}_{k=0}^{+\infty} \in \ell^2(\R)$, and hence
			\begin{align*}
		Me^{\frac{M}{T}} \int_0^T \left|\sum_{k=0}^{+\infty}a_k e^{-\lambda_{n,k}t}\right|^2\diff t &\geqslant Ce^{\frac{C}{T}} e^{-n^2 T} \int_0^T \left|\sum_{k=0}^{+\infty}a_k e^{-\big(\lambda_{n,k}-n^2\big)t}\right|^2\diff t  \\
			&\mathop{\geqslant}_{\eqref{eq: marius.inequality}} e^{-n^2 T} \sum_{k=0}^{+\infty} |a_k|^2 e^{-2\big(\lambda_{n,k}-n^2\big)T} \\
			&= \sum_{k=0}^{+\infty} |a_k|^2 e^{-2\lambda_{n,k}T}.
			\end{align*}
			The above estimate, combined with \eqref{eq: 2.8}, implies that
			\begin{align*}
			M&e^{\frac{M}{T}} \int_0^T \left\|\zeta(t, \cdot)\right\|_{L^2\left(c,d\right)}^2 \diff t\geqslant \sum_{j=0}^{+\infty} \sum_{k=0}^{+\infty} e^{-2\lambda_{n,k}T} \left|\Big\langle\big(\zeta_T, \ell_T\big), \Phi_{n,k}\Big\rangle_{\mathscr{H}_{n}}\right|^2 \left|\left\langle \varphi_{n,k}, \psi_j\right\rangle_{L^2\left(c,d\right)}\right|^2.
			\end{align*}
			After employing the Fubini theorem, we may apply \eqref{eq: eigenfunc.lower.bound} to the above estimate, revert the time shift, and deduce that 
			\begin{equation*}
			Me^{\frac{M}{T}} \int_0^T \left\|\zeta(t, \cdot)\right\|_{L^2\left(c,d\right)}^2 \diff t \geqslant C_2\left\|\big(\zeta(0, \cdot), \ell(0)\big)\right\|^2_{\mathscr{H}_{n}},
			\end{equation*}
			which holds for all $(\zeta_T,\ell_T) \in \mathscr{H}_{n}$. Here, the constant $C_2=C_2(\sigma,c-d)>0$ stems from \eqref{eq: eigenfunc.lower.bound} in \Cref{lem: spectrum}.
			This concludes the proof of \eqref{eq: obs.ineq}, and thus the proof of \Cref{thm: frequency.obs} altogether.
		\end{proof}

\section{Proof of \Cref{thm: main.result.linear}} \label{sec: proof.main.result.linear.new}

Let us now consider $\omega=(a, b)\times(c,d)$, where $(a, b)\subset\mathbb{T}$ and $(c, d)\subset(-1,1)$.  We consider the adjoint system
{\color{black}
\begin{equation} \label{sys-lin-adj}
	\begin{cases}
	-\del_t \zeta - \Delta \zeta = 0 &\text{ in } (0, T) \times \Omega,\\
	-\del_t \ell(t, x_1) -\sigma \partial_{x_{1}}^{2} \del_{x_2} \zeta(t, x_1, 1)=0 &\text{ on } (0, T) \times \T,\\
	\zeta(t, x_1, 0) = 0 &\text{ on } (0, T) \times \T,\\
	\zeta(t, x_1, 1) =  \ell(t, x_1) &\text{ on } (0, T) \times \T,\\
	(\zeta, \ell)_{\mid_{t=T}} = \left(\zeta_T, \ell_T\right) &\text{ in } \Omega \times \T.
	\end{cases}
	\end{equation}
	}
{\color{black}
Using the shorthand
\begin{align*}
\mathscr{H}_n\otimes\left\{\frac{1}{\sqrt{2\pi}}e^{in\cdot}\right\}&=\Bigg\{(\zeta,\ell)\in L^2(\Omega)\times H^1(\T)\, \Biggm|\,\\
&\hspace{0.75cm}\left(\zeta(x_1,x_2),\ell(x_1)\right)=\frac{1}{2\pi}\left(\int_{\T}\zeta(s,x_2)e^{ins}\diff s,\int_{\T}\ell(s)e^{in s}\diff s\right)e^{inx_1}\Bigg\},
\end{align*}
}we also define the space of low frequencies
\begin{align*}
E_\mu&:= \bigoplus_{\substack{|n| \leqslant \mu}} \mathscr{H}_{n} \otimes\left\{\frac{1}{\sqrt{2\pi}}e^{in\cdot}\right\}\\
&=\left\{ \frac{1}{\sqrt{2\pi}}\sum_{|n| \leqslant \mu} (f_{n}(x_{2}), g_{n}) e^{inx_{1}} \, \Biggm|\, (f_{n}, g_{n}) \in L^{2}(-1,1) \times \mathbb{R} \right\}.
\end{align*}
{\color{black}Note that $E_\mu$ is a closed subspace of ${\mathscr{H}}$.}
Let $\mathbf{\Pi}_{\mu}:{\mathscr{H}} \to{\mathscr{H}}$ denote the orthogonal projection from ${\mathscr{H}}$ onto $E_\mu$. 

\subsection{Control of the low frequencies} We first recall the following version of the Lebeau-Robbiano spectral inequality for the eigenfunctions of the Laplacian on $\mathbb{T}$ with periodic boundary conditions. 

	\begin{lemma}[Spectral inequality] \label{lem:LR}
	Let $a,b\in\mathbb{R}$ be such that $a<b$. 
	There exists $C >0$ such that for every $\mu>0$ and $\{a_{n}\}_{n \in \mathbb{Z}}\subset\mathbb{C}$, the inequality
	\begin{equation*}
	\sum_{|n| \leqslant \mu} |a_{n}|^{2} \leqslant Ce^{C\mu}\int_{a}^{b}\left|\sum_{|n| \leqslant \mu} a_{n} e^{inx_{1}}\right|^{2}\,\rd x_{1}
	\end{equation*}
	holds.
\end{lemma}

\begin{proof} See \cite{lebeau1995controle}, and also \cite[Theorem 5.4]{le2012carleman}, \cite[Proposition 5]{beauchard2014null}.
\end{proof}

Using \Cref{lem:LR}, we may derive the following observability inequality for the low frequencies.

{\color{black}
\begin{proposition}[Low-frequency observability] \label{thm:obs-filtered} 
Let $\sigma>0$ be fixed, and suppose that $\omega=(a, b)\times(c,d)$, where $(a, b)\subset\mathbb{T}$ and $(c, d)\subset (-1,1)$, $c\neq\pm d$. Then, there exists a constant $C>0$ such that for every $0<T\leqslant1$, for every $\mu>0$, and for every $(\zeta_{T},\ell_{T})\in E_\mu$, the unique solution $(\zeta,\ell)\in C^0([0,T]; L^2(\Omega)\times H^{-1}(\T))$ to \eqref{sys-lin-adj} satisfies 
\begin{align*}
\int_{\Omega} \big|\zeta(0)\big|^{2}\diff x & + \norm{\ell(0)}_{H^{-1}(\T)}^{2} \leqslant C e^{C\left(\frac{1}{T} + \mu\right)} \int_{0}^{T} \int_{\omega} |\zeta(t)|^{2} \diff x\diff t.
\end{align*}
\end{proposition}
}

{\color{black}
\begin{proof}
By assumption, we can write
\begin{equation*}
\Big(\zeta_T(x_{1}, x_{2}),\ell_T(x_{1})\Big) = \frac{1}{\sqrt{2\pi}}
\sum_{\substack{|n| \leqslant \mu }}\Big(\zeta_{T,n}(x_{2}),\ell_{T,n}\Big) e^{inx_{1}},
\end{equation*}
with the Fourier coefficients being defined as in \eqref{eq: yn.fourier}.
The solution $(\zeta,\ell)$ to \eqref{sys-lin-adj} is then given by 
\begin{equation*}
\Big(\zeta(t, x_{1}, x_{2}),\ell(t,x_{1})\Big) = \frac{1}{\sqrt{2\pi}}
\sum_{\substack{|n| \leqslant \mu}}\Big(\zeta_{n}(t, x_{2}),\ell_{n}(t)\Big) e^{inx_{1}}, 
\end{equation*}
where for any $n,$ $(\zeta_{n}, \ell_{n})$ solves \eqref{eq: fourier.projected.system.new}. 
Combining \eqref{eq: frequency.by.frequency.obs} (noting that $T\leqslant1$) with \Cref{lem:LR}, we find that 
\begin{align*}
\int_{\Omega} |\zeta(0)|^{2} \diff x + \norm{\ell(0)}_{H^{-1}(\T)}^{2}  &=\sum_{\substack{|n| \leqslant \mu}} \left(  \big\|\zeta_{n}(0)\big\|_{L^{2}(-1,1)}^{2} + \Big(1+ \sigma n^{2}\Big)^{-1}\big|\ell_{n}(0)\big|^{2} \right) \\
&\mathop{\leqslant}_{\eqref{eq: frequency.by.frequency.obs}} Me^{\frac{M}{T}} \sum_{\substack{|n| \leqslant \mu}} \int_{0}^{T} \int_{c}^{d} \Big|\zeta_{n}(t, x_{2})\Big|^{2} \ \rd x_{2} \rd t \\
&\mathop{\leqslant}_{\mathrm{Lem. \ref{lem:LR}}} C_1 e^{C_1\big(\frac{1}{T} + \mu\big)}  \int_{0}^{T} \int_{c}^{d} \int_{a}^{b} \left|\sum_{\substack{|n|\leqslant\mu}}  \zeta_{n}(t, x_{2}) e^{inx_{1}}  \right|^{2} \ \rd x_{1} \rd x_{2} \rd t \\
& = C_1 e^{C_1\big(\frac{1}{T} + \mu\big)}  \int_{0}^{T} \int_{\omega} \big|\zeta(t, x_{1}, x_{2})\big|^{2} \ \rd x_{1} \rd x_{2} \rd t
\end{align*}
holds for some constant $C_1>0$ independent of $T$ and $\mu$. This is the desired conclusion. 
\end{proof}
}
By customary HUM arguments, we also deduce the following result. 

\begin{proposition}[Low-frequency controllability] \label{thm:control-filter}
Let $\sigma>0$ be fixed, and suppose that $\omega=(a,b)\times(c,d)$, where $(a,b)\subset\T$ and $(c,d)\subset(-1,1)$, $c\neq\pm d$. Then, there exists a constant $C > 0$ such that for any $0<T\leqslant1$, for every $\mu>0$, and for every $(y^{0}, h^{0})\in{\mathscr{H}}$, there exists $u_{\mu} \in L^{2}((0,T) \times \omega)$ such that the unique solution $(y, h)\in C^0([0,T];{\mathscr{H}})$ to \eqref{sys-lin-con.intro} with control $u_{\mu}$ satisfies 
\begin{equation} \label{eq: project.control}
\mathbf{\Pi}_{\mu} \left(y(T), h(T)\right) = 0,
\end{equation}
and 
\begin{equation} \label{eq: 48}
\big\|u_{\mu}\big\|_{L^{2}((0,T) \times \omega)} \leqslant C e^{C\big(\frac{1}{T} + \mu\big)} \norm{\Big(y^{0}, h^{0}\Big)}_{{\mathscr{H}}}.
\end{equation}
\end{proposition}

\begin{proof} We proceed by using customary HUM arguments,  closely following \cite{zuazua1997finite}. Let us define the functional
\begin{equation*}
\mathscr{J}(\zeta_T,\ell_T):=\frac12\int_0^T\int_\omega |\zeta|^2\diff x\diff t-\left\langle\Big(\zeta(0),\ell(0)\Big),\Big(y^0,h^0\Big)\right\rangle_{{\mathscr{H}^{*}, \mathscr{H}}},
\end{equation*}
for $(\zeta_T,\ell_T)\in E_\mu$. The functional $\mathscr{J}$ has a unique minimizer $(\zeta_T^*,\ell_T^*)\in E_\mu$ by virtue of the observability inequality of \Cref{thm:obs-filtered}. Let $(\zeta^*,\ell^*)$ denote the corresponding solution to \eqref{sys-lin-adj}.
By writing down the Euler-Lagrange equation, one quickly finds that the minimizer $(\zeta_T^*,\ell_T^*)$ is such that
\begin{equation*}
\left\langle\Big(y(T),h(T)\Big),\Big(\varphi_T,s_T\Big)\right\rangle_{{\mathscr{H}}}=0
\end{equation*}
for all $(\varphi_T,s_T)\in E_\mu$. This yields \eqref{eq: project.control}, by definition of $\*\Pi_\mu$. Estimate \eqref{eq: 48} follows similarly by virtue of \Cref{thm:obs-filtered}.
\end{proof}

\subsection{Decay of the high frequencies}
To go beyond the above theorem, as is common with the Lebeau-Robbiano method, we will need to make use of the following exponential decay result for the high frequency components.

\begin{lemma}[Exponential decay of high-frequencies] \label{lem:exp-projection}
Suppose $\sigma>0$. The following statements hold.
\begin{enumerate} 
\item 
The operator $\*A$ generates a bounded semigroup on $\mathscr{H}$. More precisely, there exists a constant $M \geqslant 1$ such that
\begin{equation*}
\Big\|e^{t\*A}\Big\|_{\mathscr{L}({\mathscr{H}})} \leqslant M,
\end{equation*}
holds for all $t\geqslant 0$.
\smallskip
\item 
Let $\mu>0$ be fixed, and suppose that $\left(y^{0}, h^{0}\right)\in\mathscr{H}$ is such that $\mathbf{\Pi}_{\mu} \left(y^{0}, h^{0}\right) = 0$. Then the solution of \eqref{sys-lin-con.intro} with $u \equiv 0$ satisfies 
\begin{equation*}
\Big\|\Big(y(t), h(t)\Big)\Big\|_{ \mathscr{H}} \leqslant e^{-\min\big\{\frac{\sigma}{2},1\big\}\mu^{2} t} \Big\|\Big(y^{0}, h^{0}\Big)\Big\|_{ \mathscr{H}}
\end{equation*}
for all $t\geqslant 0$.
\end{enumerate}
\end{lemma}

\begin{proof} 
{\color{black} Regarding the first claim, from \Cref{lem-spec-A0} and \Cref{lem: spectrum}, we infer that supremum of the spectrum of $\* A$ is zero. Thus $\* A$ generates a bounded semigroup on $\mathscr{H}.$}

Regarding the second claim, since $\mathbf{\Pi}_{\mu} \left(y^{0}, h^{0}\right) = 0$, we may write
\begin{equation*}
\Big(y^{0}(x_{1}, x_{2}),h^0(x_1)\Big) = \frac{1}{\sqrt{2\pi}}
\sum_{ |n| > \mu} \Big(y^0_{n}(x_{2}), h^0_{n}\Big) e^{inx_{1}}, 
\end{equation*}
with the Fourier coefficients being again defined as in \eqref{eq: yn.fourier}.
The solution $(y,h)$ to \eqref{sys-lin-con.intro} with $u \equiv 0$ is then given by 
\begin{equation*}
\Big(y(t, x_{1}, x_{2}),h(t,x_1)\Big) = \frac{1}{\sqrt{2\pi}}
\sum_{|n| > \mu} \Big(y_{n}(t, x_{2}),h_{n}(t)\Big)e^{inx_{1}},
\end{equation*}
where $(y_{n},h_{n})$ is the unique solution to \eqref{eq: fourier.projected.system.intro} with $u_{n}\equiv0$. Since $n\neq0$, the latter reads
\begin{equation} \label{eq: 50}
\Big(y_n(t,x_2),h_n(t)\Big)=\sum_{k=0}^{+\infty}e^{\lambda_{n,k}t}\left\langle\Big(y_n^0,h_n^0\Big),\Phi_{n,k}\right\rangle_{\mathscr{H}_n}\Phi_{n,k}(x_2).
\end{equation}
(Here $\{\Phi_{n,k}\}_{k=0}^{+\infty}$ is the sequence of eigenfunctions of $\*A_n$ forming an orthobasis of $\mathscr{H}_n$.)
We make use of \eqref{eq: 50} and \Cref{rem: spectral.gap} to find that
\begin{align*}
\norm{\Big(y_{n}(t), h_{n}(t)\Big)}_{\mathscr{H}_{n}} \leqslant e^{\lambda_{0, n}t}\norm{\Big(y_{n}^{0}, h_{n}^{0}\Big)}_{\mathscr{H}_{n}} &\leqslant e^{-\min\big\{\frac{\sigma}{2},1\big\} n^{2} t} \norm{\Big(y_{n}^{0}, h_{n}^{0}\Big)}_{\mathscr{H}_{n}} \\
 &\leqslant e^{-\min\big\{\frac{\sigma}{2},1\big\} \mu^{2} t} \norm{\Big(y_{n}^{0}, h_{n}^{0}\Big)}_{\mathscr{H}_{n}}
\end{align*}
holds for all $t\geqslant0$. Summing up over $|n|>\mu$ yields the desired conclusion. 
\end{proof}

Combining \Cref{thm:control-filter} and \Cref{lem:exp-projection}, we obtain the following result.

\begin{proposition} \label{pr:control-filter-estimate}
Let $\sigma>0$ be fixed, and suppose $\omega=(a,b)\times(c,d)$, where $(a,b)\subset\T$ and $(c,d)\subset(-1,1)$, $c\neq\pm d$. Then, there exists $C > 0$ such that for any $0<T\leqslant1$, for any $(y^{0}, h^{0}) \in{\mathscr{H}},$ and for any $\mu > 0,$ there exists a control $u_{\mu} \in L^{2}((0,T) \times \omega)$ such that the unique solution $(y, h)\in C^0([0,T];{\mathscr{H}})$ to \eqref{sys-lin-con.intro} with control $u_{\mu}$ satisfies 
\begin{equation} \label{eq:est-control-filter}
\big\|u_{\mu}\big\|_{L^{2}((0,T) \times \omega)} \leqslant C e^{C\big(\frac{1}{T} + \mu\big)} \Big\|\Big(y^{0}, h^{0}\Big)\Big\|_{{\mathscr{H}}},
\end{equation}
and 
\begin{equation} \label{eq:est-state-filter}
\Big\|\Big(y(T), h(T)\Big)\Big\|_{{\mathscr{H}}} \leqslant C \left(1 + \sqrt{T} \right)  e^{C\big(\frac{1}{T} + \mu\big) - \min\big\{\frac{\sigma}{2},1\big\} \mu^{2} \frac{T}{2}} \Big\|\Big(y^{0}, h^{0}\Big)\Big\|_{{\mathscr{H}}}.
\end{equation}
\end{proposition}

\begin{proof}
By virtue of \Cref{thm:control-filter}, there exists a control $\overline{u}_{\mu} \in L^{2}\left(\left(0,\frac{T}{2}\right)\times\omega\right)$ such that the solution $(\overline{y}, \overline{h})$ to \eqref{sys-lin-con.intro} with control $\overline{u}_{\mu}$ satisfies
\begin{equation} \label{control-T/2}
\mathbf{\Pi}_{\mu} \left(\overline{y}\left(\frac{T}{2}\right),\overline{h}\left(\frac{T}{2}\right)\right) = 0,
\end{equation}
and 
\begin{equation} \label{eq: 55.5}
\norm{\overline{u}_{\mu}}_{L^{2}\left(\left(0,\frac{T}{2}\right)\times\omega\right)} \leqslant C e^{C\big(\frac{1}{T} + \mu\big)} \Big\|\Big(y^{0}, h^{0}\Big)\Big\|_{{\mathscr{H}}}.
\end{equation}
The constant $C>0$, stemming from \Cref{thm:control-filter}, is independent of both $T>0$ and $\mu>0$.
By virtue of the Duhamel formula, the semigroup boundedness per \Cref{lem:exp-projection}, the Cauchy-Schwarz inequality, and \eqref{eq: 55.5}, we find 
\begin{align} \label{est-state-T/2}
\left\|\left(\overline{y}\left(\frac{T}{2}\right), \overline{h}\left(\frac{T}{2}\right)\right)\right\|_{{\mathscr{H}}} & \leqslant M \Big\|\Big(y^{0},h^{0}\Big)\Big\|_{{\mathscr{H}}} + M \sqrt{T} \left\|\overline{u}_{\mu}\right\|_{L^{2}\left(\left(0,\frac{T}{2}\right)\times \omega\right)} \notag \\
& \leqslant C \left(1 + \sqrt{T} \right)  e^{C\big(\frac{1}{T} + \mu\big)}  \Big\|\Big(y^{0}, h^{0}\Big)\Big\|_{{\mathscr{H}}}.
\end{align}
Between $\frac{T}{2}$ and $T,$ we let the system dissipate according to \Cref{lem:exp-projection}. Namely, let us define 
\begin{equation*}
u_{\mu}(t) := \begin{cases}
\overline{u}_{\mu}(t) & \mbox{ if } t \in\left(0,\frac{T}{2}\right)\\
0 & \mbox{ if } t\in\left(\frac{T}{2}, T\right).
\end{cases}
\end{equation*}
From the above definition, we immediately have \eqref{eq:est-control-filter}. Finally, using the second claim in \Cref{lem:exp-projection} combined with \eqref{control-T/2}, and subsequently \eqref{est-state-T/2}, 
we deduce 
\begin{align*}
\Big\|\Big(y(T), h(T)\Big)\Big\|_{{\mathscr{H}}} & \leqslant e^{-\min\big\{\frac{\sigma}{2},1\big\}\mu^{2} \frac{T}{2}}\left\|\left(\overline{y}\left(\frac{T}{2}\right),\overline{h}\left(\frac{T}{2}\right)\right)\right\|_{{\mathscr{H}}}  \\
& \leqslant C \left(1 +\sqrt{T} \right)   e^{C\big(\frac{1}{T} + \mu\big)}  e^{-\min\big\{\frac{\sigma}{2},1\big\}\mu^{2} \frac{T}{2}} \Big\|\Big(y^{0}, h^{0}\Big)\Big\|_{{\mathscr{H}}}.
\end{align*}
This concludes the proof.
\end{proof}

\begin{remark}
In the constants appearing in the estimates \eqref{eq:est-control-filter} and \eqref{eq:est-state-filter}, the term $T$ actually indicates the length of the interval, rather than the endpoint. In other words, the interval $(0,T)$ may be replaced by an arbitrary interval $(\tau_1,\tau_2)$ in the above result, in which case the $T$ appearing in the constants of these estimates would be replaced by $\tau_2-\tau_1$.
\end{remark}

\subsection{Proof of \Cref{thm: main.result.linear}}
We may now complete the proof of our second main result.

\begin{proof}[Proof of \Cref{thm: main.result.linear}]

{\color{black}
Without loss of generality, we can suppose $0<T\leqslant1$ (indeed, otherwise, simply set the control equal to $0$ beyond time $1$).
Let $\beta > 0$ be fixed and to be chosen suitably later on. 
For $j\geqslant1$, we consider the dyadic sequences 
\begin{equation*}
T_{j} := 2^{-j}T \hspace{0.5cm} \text{ and } \hspace{0.5cm} \mu_{j}:= 2^{j}\beta,
\end{equation*}
and we define the sequence
\begin{equation*}
\tau_j:= 
\begin{dcases}
0 &\text{ for } j=0,\\
\sum_{k=1}^j T_k &\text{ for } j\geqslant1.
\end{dcases}
\end{equation*}
We note that
\begin{equation*}
(0,T) = \bigcup_{j=0}^{+\infty}(\tau_{j}, \tau_{j+1}).
\end{equation*}
For $j\geqslant1$, on any interval $(\tau_{j-1}, \tau_{j})$, by virtue of \Cref{pr:control-filter-estimate}, with $\mu=\mu_{j}$, we may build a control $u_j\in L^2((\tau_{j-1},\tau_{j})\times\omega)$ with corresponding state $(y_{j}, h_{j}) \in C^0([\tau_{j-1},\tau_j];{\mathscr{H}})$ such that 
\begin{equation} \label{eq: 6.12-bis}
\|u_j\|_{L^2((\tau_{j-1},\tau_{j})\times\omega)} \leqslant Ce^{C\left(\frac{1}{T_{j}}+\mu_{j}\right)}\left\|\Big(y_{j}(\tau_{j-1}), h_{j}(\tau_{j-1})\Big)\right\|_{{\mathscr{H}}},
\end{equation} 
as well as 
\begin{align} \label{eq: 6.13-bis}
&\left\|\Big(y_{j}(\tau_{j}), h_{j}(\tau_{j})\Big)\right\|_{{\mathscr{H}}}\nonumber\\
&\qquad\leqslant C\left(1+\sqrt{T_j}\right)e^{C\left(\frac{1}{T_j}+\mu_j\right)-\min\left\{\frac{\sigma}{2},1\right\}\mu_j^2\frac{T_j}{2}}\left\|\Big(y_{j}(\tau_{j-1}),h_{j}(\tau_{j-1})\Big)\right\|_{{\mathscr{H}}}
\end{align}
hold, with the convention $(y_{1}(0), h_{1}(0)) = \left(y^{0}, h^{0}\right).$ 
From \eqref{eq: 6.12-bis} and \eqref{eq: 6.13-bis}, we gather 
\begin{align} 
&\norm{u_{j}}_{L^{2}\left(\left(\tau_{j-1},\tau_{j}\right) \times \omega\right)}\label{eq:control-fj-1} \\ 
& \leqslant C^{j} \left(1+\sqrt{T}\right)^{j-1} 
\exp\left(\,\sum_{k=1}^{j} C\left(\frac{1}{T_{k}} + \mu_{k} \right)   - \min\left\{\frac{\sigma}{2},1\right\} \sum_{k=1}^{j-1} \mu_{k}^{2} \frac{T_{k}}{2} \right)\left\|\Big(y^{0}, h^{0}\Big)\right\|_{{\mathscr{H}}}\notag.
\end{align}
as well as
\begin{align} \label{eq:sol-fj}
&\left\|\Big(y_{j}(\tau_{j}), h_{j}(\tau_{j})\Big)\right\|_{{\mathscr{H}}} \\
 &\leqslant  C^{j}\left(1+\sqrt{T}\right)^{j}\exp\left(\,\sum_{k=1}^{j} \left[ C\left(\frac{1}{T_{k}} +\mu_{k}\right)   - 
 \min\left\{\frac{\sigma}{2},1\right\} 
 \mu_{k}^{2} \frac{T_{k}}{2} \right]\right)\left\|\Big(y^{0}, h^{0}\Big)\right\|_{{\mathscr{H}}}\notag.
\end{align}
Let us now define 
\begin{align*}
\beta_{0}:=-\frac{2C}{T}-2\beta C+\min\left\{\frac{\sigma}{2},1\right\} \beta^{2} T,
\end{align*}
and
\begin{equation*}
\beta_{1}:=-\frac{2C}{T}-2\beta C+\min\left\{\frac{\sigma}{2},1\right\} \frac{\beta^{2} T}{2}.
\end{equation*}
We may observe that 
\begin{align*}
\sum_{k=1}^{j}\left[C\left(\frac{1}{T_{k}} + \mu_{k} \right) -  \min\left\{\frac{\sigma}{2}, 1\right\} \mu_{k}^{2} \frac{T_{k}}{2} \right]&=\left(\frac{2C}{T}+2\beta C- \min\left\{\frac{\sigma}{2},1\right\} T\right)\left(2^{j} - 1\right) \\
&= -\beta_{0}(2^{j}-1),
\end{align*}
as well as
\begin{align*}
 \sum_{k=1}^{j}  C\left(\frac{1}{T_{k}} + \mu_{k}\right)   - \min\left\{\frac{\sigma}{2},1\right\} \sum_{k=1}^{j-1}  
 \mu_{k}^{2} \frac{T_{k}}{2} 
 =  - 2^{j}\beta_{1} + \beta_{0}.
\end{align*}
Thus, from \eqref{eq:control-fj-1} and \eqref{eq:sol-fj}, we infer that  
\begin{equation} \label{eq:control-fj-2}
\norm{u_{j}}_{L^{2}((\tau_{j-1},\tau_{j}) \times \omega)}  \leqslant \frac{C^{j}\left(1+\sqrt{T}\right)^{j}}{1+\sqrt{T}}  e^{\beta_{0}} e^{-2^{j}\beta_1} \left\|\Big(y^{0}, h^{0}\Big)\right\|_{{\mathscr{H}}}.
\end{equation}
and 
\begin{equation} \label{eq:sol-fj-2}
\left\|\Big(y_{j}(\tau_{j}), h_{j}(\tau_{j})\Big)\right\|_{{\mathscr{H}}} \leqslant C^{j} \left(1+\sqrt{T}\right)^{j}  e^{\beta_{0}} e^{-\beta_{0} 2^{j}} \left\|\Big(y^{0}, h^{0}\Big)\right\|_{{\mathscr{H}}},
\end{equation}
We select $m=m(T)\in\mathbb{N}$ sufficiently large so that
\begin{equation} \label{cond-m}
C \left(1+\sqrt{T}\right) < 2^{m}.
\end{equation}
Moreover, without loss of generality let us suppose $\sigma\in(0,2]$, and set
\begin{equation*}
\beta:=\frac{\alpha C}{\sigma T},
\end{equation*}
where $\alpha>0$ is large enough so that $\beta_1\geqslant1$ (The case $\sigma>2$ will follow analogously, as then $\min\{\sfrac{\sigma}{2},1\}=1$, and we can choose $\beta:=\sfrac{\alpha C}{T}$.)
We see that
\begin{equation*}
\beta_0 = \frac{C^2}{T\sigma}\left(\frac{\alpha^2}{2}-2\alpha T-\frac{2\sigma}{C}\right)
\end{equation*}
and
\begin{equation*}
\beta_1 = \frac{C^2}{T\sigma}\left(\frac{\alpha^2}{4}-2\alpha T-\frac{2\sigma}{C}\right).
\end{equation*}
Hence, from \eqref{eq:control-fj-2} and since $\beta_1\geqslant1$, we have 
\begin{align} \label{eq:control-fj-3} 
\norm{u_{j}}_{L^{2}((\tau_{j-1},\tau_{j}) \times \omega)}  
& \leqslant e^{\beta_{0}} \beta_{1}^{-m}  \left(C\left(1+\sqrt{T}\right) 2^{-m}\right)^{j}  (\beta_{1} 2^{j})^{m} e^{-\beta_{1} 2^{j}} \left\|\Big(y^{0}, h^{0}\Big)\right\|_{\mathscr{H}} \notag \\
& \leqslant C_{1} e^{\frac{C_1}{T}} \left(\left(1+C\sqrt{T}\right) 2^{-m}\right)^{j} \left\|\Big(y^{0}, h^{0}\Big)\right\|_{\mathscr{H}},
\end{align}
where $C_{1}=C_1(T,\sigma)>0$ is independent of $j.$  We define the control $u:(0,T) \times \omega \to \mathbb{R}$ by pasting all the $u_j$ as
\begin{equation*} \label{def-control}
u = \sum_{j=1}^{\infty} u_{j} 1_{(\tau_{j-1}, \tau_{j})}.
\end{equation*}
By virtue of \eqref{eq:control-fj-3} and \eqref{cond-m}, we get
\begin{align*}
\norm{u}^{2}_{L^{2}(0,T;L^{2}(\omega))} = \sum_{j=1}^{\infty} \norm{u_{j}}_{L^{2}((\tau_{j-1},\tau_{j}) \times \omega)}^{2} & = C_{1}^{2}e^{\frac{2C_1}{T}} \left\|\Big(y^{0}, h^{0}\Big)\right\|_{\mathscr{H}}^2  \sum_{j=1}^{\infty} 
 \left(\left(1+C\sqrt{T}\right) 2^{-m}\right)^{2j} \\
 &= C_{2} e^{\frac{C_2}{T}}  \left\|\Big(y^{0}, h^{0}\Big)\right\|_{\mathscr{H}}^2,
\end{align*}
for some
$C_{2}=C_2(T,\sigma)>0$, and therefore $u \in L^{2}((0,T)\times\omega).$ 
Now consider the unique solution $(y,h)\in C^0([0,T];L^{2}(\Omega)\times H^1(\T))$ to \eqref{sys-lin-con.intro}, with control $u$. 
Clearly 
\begin{equation*}
\big(y(t), h(t)\big) = \big(y_{j}(t), h_{j}(t)\big) \hspace{0.5cm} \text{ for } t\in[\tau_{j-1},\tau_j].
\end{equation*}
In particular, 
\begin{equation*}
\big(y(\tau_{j}), h(\tau_{j})\big) = \big(y_{j}(\tau_{j}), h_{j}(\tau_{j})\big),
\end{equation*} 
so from \eqref{eq:sol-fj-2} we deduce that $\big(y(T,\cdot), h(T,\cdot)\big)\equiv0$ in $\Omega\times\T$. 
This concludes the proof.
}
\end{proof}

\section{Proof of \Cref{thm:main-nonlinear}} \label{sec: nonlinear.section}

\subsection{Control in spite of source terms} \label{sec: source.term.method}

{\color{black}
In view of tackling the controllability of the nonlinear system, we look to add the source terms over which we aim to apply a fixed point argument. 
	Let us therefore consider the following linear control system with non-homogeneous source terms 
\begin{equation} \label{sys-lin}
\begin{dcases}
\partial_{t} y - \Delta y =  u 1_{\omega} + f_1 & \mbox{ in } (0,T) \times \Omega, \\
\partial_{t}  h(t,x_1) -  \partial_{x_{2}} y(t,x_1,1) = f_{2}(t,x_1)  & \mbox{ on } (0,T) \times \mathbb{T}, \\
y(t,x_{1}, -1) = 0 &  \mbox{ on } (0,T) \times \mathbb{T}, \\
y(t, x_{1}, 1) = \sigma \del_{x_1}^2 h(t,x_1) + f_{3}(t,x_1) & \mbox{ on } (0,T) \times \mathbb{T}, \\
\left(y, h\right)_{|_{t=0}} = \left(y^0, h^0\right) & \mbox{ in } \Omega \times \T.
\end{dcases}
\end{equation}

\subsubsection{Improved regularity for \eqref{sys-lin}}
Before proceeding with the control analysis, let us provide a necessary regularity result for \eqref{sys-lin}. 
We consider the subset of refined initial data 
\begin{equation*}
\mathscr{I} := \left\{\begin{bmatrix}y^0\\h^0\end{bmatrix}
 \in H^{1}(\Omega) \times H^{\sfrac{5}{2}}(\T) \,\,\Biggm|\,\, y^{0}(x_{1}, -1) = 0 \mbox{ in } \T \right\},
\end{equation*}
as well as
\begin{multline*}
\mathfrak{E}_{f}(0,T) := \left\{ (f_{1}, f_{2}, f_{3}) \in L^{2}((0,T)\times\Omega)\times H^{\sfrac{1}{2}, \sfrac{1}{4}}((0,T) \times \T) \times H^{\sfrac{3}{2}, \sfrac{3}{4}}((0,T) \times \T) \right\},
\end{multline*}
where we used 
\begin{equation*}
H^{r,s}((0,T) \times \T) : = L^{2}((0, T); H^{r}(\T)) \cap H^{s}((0, T);L^{2}(\T))
\end{equation*}
for $r,s\geqslant0$.
We also introduce the (higher order) energy spaces
\begin{equation} \label{eq: Ey}
\mathfrak{E}_{y}(0,T) := L^{2}\Big((0,T);H^{2}(\Omega)\Big) \cap H^{1}\Big((0,T);L^{2}(\Omega)\Big) \cap C^0\Big([0,T];H^{1}(\Omega)\Big),
\end{equation}
and
\begin{align} \label{eq: Eh}
\mathfrak{E}_{h}(0,T) := L^{2}\Big((0,T);H^{\sfrac{7}{2}}(\T)\Big) &\cap H^{\sfrac{3}{4}}\Big((0,T);H^{2}(\T)\Big) \cap H^{1}\Big((0,T);H^{1}(\T)\Big) \\
&\cap H^{\sfrac{5}{4}}\Big((0,T);L^{2}(\T)\Big) \cap C^0\Big([0,T];H^{\sfrac{5}{2}}(\T)\Big).\notag
\end{align}
The following improved well-posedness result then holds.

\begin{proposition} \label{thm:lin-source}
Suppose $\sigma>0$ and $T>0$. 
For any $\left(y^{0}, h^{0}\right)\in L^{2}(\Omega) \times H^{1}(\T)$ and $(f_{1}, f_{2}, f_{3}) \in \mathfrak{E}_{f}(0,T)$, System \eqref{sys-lin} with control $u\equiv0$ admits a unique mild solution $(y, h) \in C^0([0,T];L^{2}(\Omega) \times H^{1}(\T))$, and there exists
some constant $C(T,\sigma)>0$ such that 
\begin{equation} \label{est-linear-main}
\norm{\left(y, h\right)}_{ C^0([0,T];L^{2}(\Omega) \times H^{1}(\T))} \leqslant C(T,\sigma) \left( \norm{\left(y^{0}, h^{0}\right)}_{L^{2}(\Omega) \times H^{1}(\T)} + \norm{\left(f_{1}, f_{2}, f_{3}\right)}_{\mathfrak{E}_{f}(0,T)}\right).
\end{equation}
If moreover $\left(y^0, h^0\right) \in \mathscr{I}$, and satisfies the compatibility condition 
\begin{equation} \label{cc1}
y^{0}(x_{1}, 1) = \sigma \del_{x_1}^2 h^{0}(x_{1}) + f_{3}(0, x_{1}) \hspace{0.5cm} \mbox{ in } \T,
\end{equation}
then \eqref{sys-lin} admits a unique strong solution $(y, h) \in \mathfrak{E}_{y}(0,T) \times \mathfrak{E}_{h}(0,T)$, and there exists some constant $C(T,\sigma)>0$ such that
\begin{equation*} 
\norm{\left(y, h\right)}_{ \mathfrak{E}_{y}(0,T) \times \mathfrak{E}_{h}(0,T)} \leqslant C(T,\sigma) \left(\norm{\left(y^0, h^{0}\right)}_{\mathscr{I}} + \norm{\left(f_{1}, f_{2}, f_{3}\right)}_{\mathfrak{E}_{f}(0,T)}\right).
\end{equation*}
\end{proposition}

\begin{remark}
Note that the constant $C(T,\sigma)$ in \Cref{thm:lin-source} is of the form $\mathcal{O}(e^{T})$ with respect to $T$. In particular, it does not blow up if $T$ goes to zero.
\end{remark}

\begin{proof}[Proof of \Cref{thm:lin-source}]
Uniqueness of solutions follows easily. Thus we focus on showing existence. From well-known trace results (see for instance \cite{lions2012non}), there exists some $y^{[2]} \in \mathfrak{E}_{y}(0,T)$ such that 
\begin{equation*}
y^{[2]}(t, x_{1}, 1) = f_{3}, \hspace{0.25cm} y^{[2]}(t, x_{1}, -1) = 0, \quad \mbox{ and } \quad  \partial_{x_{2}}y^{[2]}(t, x_{1}, 1) = f_{2} \quad \mbox{ on }  (0, T) \times \T.
\end{equation*}
Moreover, there exits a positive constant $C_{1}=C_1(T) > 0$ such that 
\begin{equation} \label{est-trace}
\norm{y^{[2]}}_{\mathfrak{E}_{y}(0,T)} \leqslant C_{1}\norm{\left(f_{1}, f_{2}, f_{3}\right)}_{\mathfrak{E}_{f}(0,T)}.
\end{equation}
We now look for a decomposition of $y$ of the form $y  = y^{[1]} + y^{[2]}.$ In turn, $(y^{[1]}, h)$ must solve
\begin{equation*} \label{sys-lin-bis}
\begin{dcases}
\partial_{t} y^{[1]} - \Delta y^{[1]} =  f_{1}^{[1]} & \mbox{ in } (0,T) \times \Omega, \\
\partial_{t}  h(t,x_1) -  \partial_{x_{2}} y^{[1]}(t,x_1,1) = 0 & \mbox{ on } (0,T) \times \mathbb{T}, \\
y^{[1]}(t,x_{1}, -1) = 0 &  \mbox{ on } (0,T) \times \mathbb{T}, \\
y^{[1]}(t, x_{1}, 1) = \sigma \del_{x_1}^2 h(t,x_1)  & \mbox{ on } (0,T) \times \mathbb{T}, \\
\left(y^{[1]}, h\right)_{|_{t=0}} = \left(y^{{[1]}, 0}, h^0\right) & \mbox{ in } \Omega \times \T,
\end{dcases}
\end{equation*}
where 
\begin{equation*}
f^{[1]}_{1}:=f_{1} - \partial_{t} y^{[2]} + \Delta y^{[2]}, \quad y^{{[1]},0}:= y^0 - y^{[2]}(0, \cdot).
\end{equation*}
From \eqref{est-trace}, we gather that there exists a positive constant $C_{2}=C_2(T,\sigma)>0$ such that 
\begin{equation} \label{eq: above.estimate}
\norm{ f^{[1]}_{1} }_{L^{2}((0,T)\times\Omega)} + \norm{y^{{[1]},{0}}}_{H^{1}(\Omega)} \leqslant C_{2} \left(\norm{\left(y^{0}, h^{0}\right)}_{\mathscr{I}} + \norm{\left(f_{1}, f_{2}, f_{3}\right)}_{\mathfrak{E}_{f}(0,T)}\right).
\end{equation}
Moreover, the compatibility condition \eqref{cc1} implies that the corrected initial datum lives in an interpolation space: 
\begin{equation*}
\left(y^{{[1]},0}, h^0\right) \in \left[\mathfrak{D}(\mathbf{A}), \mathscr{H} \right]_{\frac{1}{2}}.
\end{equation*}
(Recall the definition of $(\mathbf{A},\mathfrak{D}(\mathbf{A}))$ in \eqref{eq:DA}.)
Therefore, by standard maximal regularity results (\cite[Theorem 3.1, pp. 143]{benoussan_book}), we have 
\begin{equation*}
\left(y^{[1]}, h\right) \in L^{2}((0,T);\mathfrak{D}(\mathbf{A})) \cap H^{1}((0,T);\mathscr{H}).
\end{equation*}
Combining estimate \eqref{eq: above.estimate} with \eqref{est-trace} and well-known interpolation estimates (\cite{lions2012non}), we deduce \eqref{est-linear-main}.
\end{proof}	

\subsubsection{Adding the source terms} 

We are now in a position to provide an adaptation of the source term method first introduced in \cite{tucsnak_burgers} (see also \cite{le2020local, geshkovski2019null}), in the specific setting of the problem we consider containing boundary source terms, which will allow us to then apply a fixed point method for tackling the nonlinear system. 

Herein, we shall work specifically in the regime $T\leqslant1$ for simplicity, to exploit the special, exponential character of the controllability cost $\mathcal{K}(T,\sigma)$. 
Now fix 
\begin{equation*}
q\in\left(1,\sqrt[4]{2}\right),
\end{equation*}
(recall that $\sqrt[4]{2}\approx1.1892\ldots$) and consider 
\begin{equation*} \label{eq: rho.F.def}
\rho_{\mathfrak{F}}(t):=\exp\left(-\frac{\beta}{(T-t)^2}\right)
\hspace{1cm} \text{ for } t \in [0,T],
\end{equation*}
and 
\begin{equation*} \label{eq: rho.F.def}
\rho_{0}(t):= M \exp\left(\frac{M}{(q-1)(T-t)}-\frac{\beta}{q^{4}(T-t)^2}\right)
\hspace{1cm} \text{ for } t \in [0,T],
\end{equation*}
where $M$ is the constant appearing in the control cost $\mathcal{K}(T, \sigma)$ is \Cref{thm: main.result.linear}, and 
\begin{equation*}
\beta > \frac{Mq^{4}}{2(q-1)}
\end{equation*}
is fixed.
Note that $\rho_{0}$ and $\rho_{\mathfrak{F}}$ are decreasing functions.
We also consider
\begin{equation} \label{eq: rho.alpha.def}
\rho_\alpha(t):=\exp\left(-\frac{\alpha}{(T-t)^2}\right) \hspace{1cm} \text{ for } t\in[0,T],
\end{equation}	
where 
\begin{equation*}
\alpha\in\left(\frac{\beta}{2},\frac{\beta}{q^4}\right)
\end{equation*}
is fixed. Most importantly, we also have $\rho_\alpha(T)=0$.
We now define the weighted space of source terms 
\begin{align*}
\mathfrak{F}&:= \left\{f=(f_1,f_2,f_3) \in \mathfrak{E}_{f}(0,T) \,\Biggm|\, \frac{f}{\rho_{\mathfrak{F}}}\in \mathfrak{E}_{f}(0,T) \right\}. \\
\end{align*}

The following lemma will be central in what follows.

\begin{lemma} \label{lem: weights.facts}
Let $T>0$, $\beta > \frac{Mq^{4}}{2(q-1)},$ and  $\alpha\in\left(\frac{\beta}{2},\frac{\beta}{q^4}\right)$ be fixed. The following facts hold true.
\begin{enumerate}
\item $\displaystyle \frac{\rho_{\mathfrak{F}}}{\rho_\alpha}\in W^{1,\infty}(0,T)$.
\smallskip
\item  $\displaystyle \frac{\rho_\alpha^2}{\rho_{\mathfrak{F}}}\in W^{1,\infty}(0,T)$.
\smallskip
\item For $j\geqslant1$, set $T_j:=T\left(1-q^{-j}\right)$. Then 
\begin{align*}
\rho_{0}(T_{j+2}) = \rho_{\mathfrak{F}}(T_{j})M\exp\left( \frac{M}{T_{j+2}-T_{j+1}}\right)
\end{align*}
and
\begin{align*}
\norm{\frac{1}{\rho_{\alpha}}}_{W^{1,\infty}(T_{j}, T_{j+1})}^{2}\frac{\rho_{0}^2(T_{j+1})}{(T - T_{j+1})^{6}}  \leqslant C(T),
\end{align*}
hold for some constant $C(T)>0$ independent of $j$.
\end{enumerate}
\end{lemma}


The proof readily follows from the explicit form of the weights.
The following version of the source term method then holds.

\begin{theorem}[Source term method]  \label{thm: source.term.method} 
Suppose $0<T\leqslant1$ and $\sigma>0$. 
There exist a constant $C(T,\sigma)>0$ and a continuous linear map 
$\mathfrak{L}: L^{2}(\Omega) \times H^{1}(\T)\times\mathfrak{F}\to\mathfrak{U}$ such that for any initial data $\left(y^0, h^0\right)\in L^{2}(\Omega) \times H^{1}(\T)$ and $f=(f_1, f_2, f_3)\in\mathfrak{F}$, the unique solution $(y,h)\in C^0([0,T];L^{2}(\Omega) \times H^{1}(\T))$ to \eqref{sys-lin} with control $u=\mathfrak{L}(y^0, h^0, f)$ satisfies
\begin{align*}
&\norm{\left(\frac{y}{\rho_{\alpha} }, \frac{h}{\rho_{\alpha} }\right)}_{ C^0([0,T];L^{2}(\Omega) \times H^{1}(\T))}^{2} + \norm{\frac{u}{\sqrt{\rho_{0}}}}^{2}_{L^{2}((0, T)\times\omega)} \\
&\qquad\leqslant C(T,\sigma) \Bigg(  \norm{\left(y^{0}, h^{0}\right)}^{2}_{L^{2}(\Omega) \times H^{1}(\T)}  + \norm{\left(\frac{f_{1}}{\rho_{\mathfrak{F}}}, \frac{f_{2}}{\rho_{\mathfrak{F}}}, \frac{f_{3}}{\rho_{\mathfrak{F}}}\right)}_{\mathfrak{E}_{f}(0,T)}^{2} \Bigg).
\end{align*}
In particular, $y(T, \cdot)\equiv0$ in $\Omega$ and $h(T,\cdot)\equiv0$ on $\T$.

If moreover $\left(y^0, h^0\right) \in \mathscr{I}$, and satisfies the compatibility condition \eqref{cc1}, then $(y, h)$ additionally satisfies 
\begin{align} 
&\norm{\left(\frac{y}{\rho_{\alpha} }, \frac{h}{\rho_{\alpha} }\right)}_{ \mathfrak{E}_{y}(0,T) \times \mathfrak{E}_{h}(0,T)} + \norm{\frac{u}{\sqrt{\rho_{0}}}}^{2}_{L^{2}((0, T)\times\omega)} \notag \\
&\qquad\leqslant C(T,\sigma) \Bigg( \norm{\left(y^{0}, h^{0}\right)}_{\mathscr{I}}  + \norm{\left(\frac{f_{1}}{\rho_{\mathfrak{F}}}, \frac{f_{2}}{\rho_{\mathfrak{F}}}, \frac{f_{3}}{\rho_{\mathfrak{F}}}\right)}_{\mathfrak{E}_{f}(0,T)} \Bigg). \label{est:main-nh-source}
\end{align}
\end{theorem}

\begin{proof}[Proof of \Cref{thm: source.term.method}]
We shall focus on initial data $\left(y^0, h^0\right) \in \mathscr{I}$, which satisfy the compatibility condition \eqref{cc1}, as the proof of the first part of the statement is  transparent throughout the arguments. 
We shall split the proof in five steps.

{\bf Step 1.} \emph{Lifting traces.}
First of all, from well-known trace results (\cite{lions2012non}) there exists some $y^{[2]} \in \mathfrak{E}_{y}(0,T)$ such that 
\begin{equation*}
y^{[2]}(t, x_{1}, 1) = \frac{f_{3}}{\rho_{\mathfrak{F}}},  \quad  y^{[2]}(t, x_{1}, -1) = 0, \quad \mbox{ and } \quad \partial_{x_{2}}y^{[2]}(t, x_{1}, 1) = \frac{f_{2}}{\rho_{\mathfrak{F}}} \quad \mbox{ on }  (0, T) \times \T.
\end{equation*}
Moreover, there exits a constant $C_1(T) > 0$ such that 
\begin{equation} \label{est-trace-new}
\norm{y^{[2]}}_{\mathfrak{E}_{y}(0,T)} \leqslant C_{1}(T) \norm{\left(\frac{f_{1}}{\rho_{\mathfrak{F}}}, \frac{f_{2}}{\rho_{\mathfrak{F}}}, \frac{f_{3}}{\rho_{\mathfrak{F}}}\right)}_{\mathfrak{E}_{f}(0,T)}.
\end{equation}
Per \Cref{lem: weights.facts}, we have $\frac{\rho_{\mathfrak{F}}}{\rho_{\alpha}}\in W^{1,\infty}(0,T)$, and so from \eqref{est-trace-new} we derive
\begin{equation} \label{est:y-sharp-weight}
\norm{\frac{\rho_{\mathfrak{F}} y^{[2]} }{\rho_{\alpha}}}_{\mathfrak{E}_{y}(0,T)} 
\leqslant C_2(T) \norm{\left(\frac{f_{1}}{\rho_{\mathfrak{F}}}, \frac{f_{2}}{\rho_{\mathfrak{F}}}, \frac{f_{3}}{\rho_{\mathfrak{F}}}\right)}_{\mathfrak{E}_{f}(0,T)}.
\end{equation}
We now look for a decomposition of $y$ of the form 
\begin{equation*}
y = y^{[1]} + \rho_{\mathfrak{F}} y^{[2]}.
\end{equation*}
In turn $(y^{[1]}, h)$ must satisfy
\begin{equation}\label{sys-lin-bis-00}
\begin{dcases}
\partial_{t} y^{[1]} - \Delta y^{[1]} =  f_{1}^{[1]} + u 1_{\omega} & \mbox{ in } (0,T) \times \Omega, \\
\partial_{t}  h(t,x_1) -  \partial_{x_{2}} y^{[1]}(t,x_1,1) = 0 & \mbox{ on } (0,T) \times \mathbb{T}, \\
y^{[1]}(t,x_{1}, -1) = 0 &  \mbox{ on } (0,T) \times \mathbb{T}, \\
y^{[1]}(t, x_{1}, 1) = \sigma \del_{x_1}^2 h(t,x_1)  & \mbox{ on } (0,T) \times \mathbb{T}, \\
\left(y^{[1]}, h\right)_{|_{t=0}} = \left(y^{{[1]}, 0}, h^0\right) & \mbox{ in } \Omega \times \T,
\end{dcases}
\end{equation}
where 
\begin{equation} \label{def-f-dag}
f^{[1]}_{1} = f_{1} - \partial_{t} \left(\rho_{\mathfrak{F}} y^{[2]} \right) + \Delta  \left(\rho_{\mathfrak{F}} y^{[2]} \right) , \quad y^{{[1]},0} = y^0 - \rho_{\mathfrak{F}}(0) y^{[2]}(0, \cdot).
\end{equation}
Now observe that
\begin{align} \label{est-y-dag-0}
\left\|y^{{[1]},0}\right\|_{H^1(\Omega)}&\stackrel{\eqref{est-trace-new}}{\leqslant} \left\|y^0\right\|_{H^1(\Omega)}+C_1(T)|\rho_{\mathfrak{F}}(0)|\norm{\left(\frac{f_{1}}{\rho_{\mathfrak{F}}}, \frac{f_{2}}{\rho_{\mathfrak{F}}}, \frac{f_{3}}{\rho_{\mathfrak{F}}}\right)}_{\mathfrak{E}_{f}(0,T)}\nonumber\\
&\leqslant C_3(T)\left( \norm{\left(y^{0}, h^{0}\right)}_{\mathscr{I}}  + \norm{\left(\frac{f_{1}}{\rho_{\mathfrak{F}}}, \frac{f_{2}}{\rho_{\mathfrak{F}}}, \frac{f_{3}}{\rho_{\mathfrak{F}}}\right)}_{\mathfrak{E}_{f}(0,T)}\right),
\end{align}
where $C_3(T):=\max\{1,C_1(T)|\rho_{\mathfrak{F}}(0)|\}$. If $(y^0,h^0)$ also satisfies the compatibility condition \eqref{cc1}, we additionally gather that
\begin{equation*}
y^{{[1]}, 0}(x_{1}, 1) = \sigma \del_{x_1}^2 h^{0}(x_{1})  \hspace{1cm} \mbox{ in } \T.
\end{equation*}

{\bf{Step 2.}} \emph{Splitting over geometrically-shrinking time intervals.}
We now study the controllability of the lifted system \eqref{sys-lin-bis-00}.  
Let us set
\begin{equation*}
a_{0}=\left(y^{{[1]}, 0}, h^{0}\right),
\end{equation*}
and for $j\geqslant0$, we define
\begin{equation*}
T_j:=T\left(1-q^{-j}\right),
\end{equation*}
as well as 
\begin{equation*}
a_{j+1}:=\left(y_{f}\left(T_{j+1}^{-}, \cdot\right), h_{f}\left(T_{j+1}^{-}, \cdot\right)\right),
\end{equation*}
where $\left(y_{f}, h_{f}\right)$ is the unique solution to 
\begin{equation*}  
\begin{dcases}
\partial_{t} y_{f} - \Delta y_{f} =  f_{1}^{{[1]}}  & \mbox{ in } (T_{j},T_{j+1}) \times \Omega, \\
\partial_{t}  h_{f}(t,x_1) -  \partial_{x_{2}} y_{f}(t,x_1,1) = 0  & \mbox{ on } (T_{j},T_{j+1})  \times \mathbb{T}, \\
y_{f}(t,x_{1}, -1) = 0 &  \mbox{ on } (T_{j},T_{j+1}) \times \mathbb{T}, \\
y_{f}(t, x_{1}, 1) = \sigma \del_{x_1}^2 h_f(t,x_1) & \mbox{ on } (T_{j},T_{j+1}) \times \mathbb{T}, \\
\left(y_{f}, h_{f}\right)_{|{t=T_{j}^{+}}} = 0 & \mbox{ in } \Omega \times \T.
\end{dcases}
\end{equation*}
Fix $j\geqslant0$. From \Cref{thm:lin-source}, we gather that there exists some $C(T,\sigma)>0$ (independent of $j$) such that
\begin{equation} \label{est-ak}
\norm{a_{j+1}}_{H^{1}(\Omega) \times H^{\sfrac{5}{2}}(\T)}^{2} \leqslant  C({T},\sigma) \norm{f_{1}^{{[1]}}}_{L^{2}((T_{j}, T_{j+1})\times\Omega)}^{2}.
\end{equation}
On another hand, we also consider the homogeneous control system 
\begin{equation*}  
\begin{dcases}
\partial_{t} y_{u} - \Delta y_{u} =  u_{j} 1_{\omega}  & \mbox{ in } (T_{j},T_{j+1}) \times \Omega, \\
\partial_{t}  h_{u}(t,x_1) -  \partial_{x_{2}} y_{u}(t,x_1,1) = 0  & \mbox{ on } (T_{j},T_{j+1})  \times \mathbb{T}, \\
y_{u}(t,x_{1}, -1) = 0 &  \mbox{ on } (T_{j},T_{j+1}) \times \mathbb{T}, \\
y_{u}(t, x_{1}, 1) = \sigma \del_{x_1}^2 h_u(t,x_1)  & \mbox{ on } (T_{j},T_{j+1}) \times \mathbb{T}, \\
\left(y_{u}, h_{u}\right)_{|_{t=T_{j}^{+}}} = a_{j} & \mbox{ in } \Omega \times \T.
\end{dcases}
\end{equation*}
where $u_{j} \in L^{2}((T_{j}, T_{j+1})\times\omega)$ is such that 
\begin{equation*}
(y_{u}, h_{u})(T_{j+1}^{-}, \cdot)\equiv0 \hspace{1cm} \text{ in } \Omega \times\T,
\end{equation*}
and
\begin{equation} \label{est-control-k0}
\norm{u_{j}}^{2}_{L^{2}((T_{j}, T_{j+1})\times\omega)} \leqslant M^2\exp\left(\frac{2M}{T_{j+1}-T_j}\right) \norm{a_{j}}^{2}_{L^{2}(\Omega) \times H^{1}(\T)}.
\end{equation}
We readily see that 
\begin{align} \label{est-control-k}
\norm{u_{j+1}}^{2}_{L^{2}((T_{j+1}, T_{j+2})\times\omega)} &\stackrel{\eqref{est-control-k0}}{\leqslant}M^2\exp\left(\frac{2M}{T_{j+2}-T_{j+1}}\right) \norm{a_{j+1}}^{2}_{L^{2}(\Omega) \times H^{1}(\T)} \notag \\
&\stackrel{\eqref{est-ak}}{\leqslant} C(T,\sigma) M^2\exp\left(\frac{2M}{T_{j+2}-T_{j+1}}\right) \norm{f_{1}^{{[1]}}}_{L^{2}((T_{j}, T_{j+1})\times\Omega)}^{2}.
\end{align}
Now using the definition of $f_{1}^{{[1]}}$ in \eqref{def-f-dag}, and the fact that $\rho_{\mathfrak{F}}$ is decreasing, we find 
\begin{align} \label{eq: rhoF.H1}
\norm{f_{1}^{{[1]}}}_{L^{2}((T_{j}, T_{j+1})\times\Omega)}^{2}&\leqslant C_4(T) \norm{\rho_{\mathfrak{F}}}_{L^{\infty}(T_{j}, T_{j+1})}^{2} \left( \norm{\frac{f_{1}}{\rho_{\mathfrak{F}}}}_{L^{2}((T_{j}, T_{j+1})\times\Omega)}^{2} + \norm{y^{{[2]}}}^{2}_{\mathfrak{E}_{y}(T_{j}, T_{j+1})}  \right) \nonumber \\
&\hspace{0.5cm}+ C_4(T)  \norm{\rho'_{\mathfrak{F}}}_{L^{\infty}(T_{j}, T_{j+1})}^{2} \norm{y^{{[2]}}}^{2}_{\mathfrak{E}_{y}(T_{j}, T_{j+1})} \nonumber \\
& \leqslant C_4(T)\norm{\rho_{\mathfrak{F}}}_{W^{1, \infty}(T_{j}, T_{j+1})}^{2} \left( \norm{\frac{f_{1}}{\rho_{\mathfrak{F}}}}_{L^{2}((T_{j}, T_{j+1})\times\Omega)}^{2} + \norm{y^{{[2]}}}^{2}_{\mathfrak{E}_{y}(T_{j}, T_{j+1})}  \right) \nonumber \\
& \leqslant C_4(T)\frac{ \rho_{\mathfrak{F}}(T_{j})^{2} }{(T - T_{j+1})^{6}}\left( \norm{\frac{f_{1}}{\rho_{\mathfrak{F}}}}_{L^{2}((T_{j}, T_{j+1})\times\Omega)}^{2} + \norm{y^{{[2]}}}^{2}_{\mathfrak{E}_{y}(T_{j}, T_{j+1})}  \right) 
\end{align}
for some $C_4(T)>0$ independent of $j$.
Chaining \eqref{est-control-k} and \eqref{eq: rhoF.H1}, using the monotonicity of $\rho_{0}$, and \Cref{lem: weights.facts}
we deduce 
\begin{align*}
\norm{u_{j+1}}^{2}_{L^{2}((T_{j+1}, T_{j+2})\times\omega)} &\leqslant C_{5}(T)\frac{ \rho_{0}(T_{j+2})^{2} }{(T - T_{j+2})^{6}} \Bigg(\norm{\frac{f_{1}}{\rho_{\mathfrak{F}}}}_{L^{2}((T_{j}, T_{j+1})\times\Omega)}^{2} + 
\norm{y^{{[2]}}}^{2}_{\mathfrak{E}_{y}(T_{j}, T_{j+1})} \Bigg).
\end{align*}
for some $C_5(T)>0$ independent of $j$. Using the monotonicity of $\rho_{0}$, we conclude that
\begin{equation} \label{est-control-k2}
\norm{\frac{u_{j+1}}{\sqrt{\rho_{0}}}}^{2}_{L^{2}((T_{j+1}, T_{j+2})\times\omega)} \leqslant C_{5}(T)  \Bigg(\norm{\frac{f_{1}}{\rho_{\mathfrak{F}}}}_{L^{2}((T_{j}, T_{j+1})\times\Omega)}^{2} + 
\norm{y^{{[2]}}}^{2}_{\mathfrak{E}_{y}(T_{j}, T_{j+1})} \Bigg).
\end{equation}

{\bf{Step 3.}} \emph{Gluing the controls.} We now define the control $u$ as
\begin{equation*}
u := \sum_{j=0}^{+\infty} u_{j} 1_{[T_{j}, T_{j+1}]}.
\end{equation*}
As $\frac{1}{\sqrt{\rho_{0}}}\in L^\infty(0,T)$, from \eqref{est-control-k0} we have 
\begin{equation} \label{eq: u0.est}
\norm{\frac{u_{0}}{\sqrt{\rho_{0}}}}^{2}_{L^{2}((T_{0}, T_{1})\times\omega)} \leqslant  C_{6}(T) \norm{a_{0}}^{2}_{L^{2}(\Omega) \times H^{1}(\T)}.
\end{equation} 
Combining \eqref{eq: u0.est} with \eqref{est-control-k2}, \eqref{est-trace-new} and \eqref{est-y-dag-0} we get 
\begin{equation*} \label{est-control-nonh}
\norm{\frac{u}{\sqrt{\rho_{0}}}}^{2}_{L^{2}((0, T)\times\omega)} \leqslant C_{7}(T) \Bigg( \norm{\left(y^{0}, h^{0}\right)}^{2}_{\mathscr{I}}  + \norm{\left(\frac{f_{1}}{\rho_{\mathfrak{F}}}, \frac{f_{2}}{\rho_{\mathfrak{F}}}, \frac{f_{3}}{\rho_{\mathfrak{F}}}\right)}_{\mathfrak{E}_{f}(0,T)}^{2} \Bigg).
\end{equation*}

{\bf Step 4.} \emph{Weighted estimates for the lifted state.}
We now look to estimate the controlled state. Let us set 
\begin{equation*}
(y^{{[1]}}, h) = (y_{f}, h_{f}) + (y_{u}, h_{u}).
\end{equation*} 
Then clearly for every $j\geqslant 0,$ $(y^{{[1]}}, h)_{|(T_j,T_{j+1})}$ satisfies 
\begin{equation*}
\begin{dcases}
\partial_{t} y^{[1]} - \Delta y^{[1]} =  f_{1}^{[1]} + u_{j} 1_{\omega} & \mbox{ in } (T_{j}, T_{j+1}) \times \Omega, \\
\partial_{t}  h(t,x_1) -  \partial_{x_{2}} y^{[1]}(t,x_1,1) = 0 & \mbox{ on } (T_{j}, T_{j+1}) \times \mathbb{T}, \\
y^{[1]}(t,x_{1}, -1) = 0 &  \mbox{ on } (T_{j}, T_{j+1}) \times \mathbb{T}, \\
y^{[1]}(t, x_{1}, 1) = \sigma \del_{x_1}^2 h(t,x_1)  & \mbox{ on } (T_{j}, T_{j+1})  \times \mathbb{T}, \\
\left(y^{[1]}, h\right)_{|_{t=T_{j}^{+}}} = a_{j} & \mbox{ in } \Omega \times \T.
\end{dcases}
\end{equation*}
Moreover,
\begin{align*}
(y^{{[1]}}, h)(T_{j}^{-}) &= (y_{f}, h_{f})(T_{j}^{-}) + (y_{u}, h_{u})(T_{j}^{-})\\
&= a_{j}\\
&= (y_{f}, h_{f})(T_{j}^{+}) + (y_{u}, h_{u})(T_{j}^{+}) = (y^{{[1]}}, h)(T_{j}^{+}),
\end{align*}
so that $(y^{[1]}, h)$ is continuous at each $T_{j}.$  Standard energy estimates yield
\begin{align} \label{est-lin-k}
&\norm{\left(y^{{[1]}}, h\right)}_{L^{2}((T_{j}, T_{j+1});\mathfrak{D}(\mathbf{A})) \cap H^{1}((T_{j}, T_{j+1});\mathscr{H})}^{2} \nonumber \\
&\hspace{1cm}\leqslant C(T,\sigma)\Bigg( \norm{a_{j}}^{2}_{\mathscr{I}} +  \norm{u_{j}}^{2}_{L^{2}((T_{j}, T_{j+1})\times\omega)}
+ \norm{f_{1}^{{[1]}}}_{L^{2}((T_{j}, T_{j+1})\times\Omega)}^{2} \Bigg).
\end{align}
Plugging \eqref{est-control-k0} in \eqref{est-lin-k}, we infer that 
\begin{align} \label{eq: the.above.estimate}
&\norm{\left(y^{{[1]}}, h\right)}_{L^{2}((T_{j}, T_{j+1});\mathfrak{D}(\mathbf{A})) \cap H^{1}(T_{j}, T_{j+1};\mathscr{H})}^{2} \nonumber\\
&\hspace{1cm}\leqslant C(T,\sigma)\left(\norm{a_{j}}^{2}_{\mathscr{I}} + \exp\left(\frac{2M}{T_{j+1}-T_j}\right)\norm{a_{j}}^{2}_{\mathscr{I}} + \norm{f_{1}^{{[1]}}}_{L^{2}((T_{j}, T_{j+1})\times\Omega)}^{2} \right). 
\end{align}
Using \eqref{est-ak}, \eqref{eq: the.above.estimate} can be rewritten as 
\begin{align*}
&\norm{\left(y^{{[1]}}, h\right)}_{L^{2}((T_{j}, T_{j+1});\mathfrak{D}(\mathbf{A})) \cap H^{1}((T_{j}, T_{j+1});\mathscr{H})}^{2}\\
&\hspace{4cm}\leqslant C(T,\sigma)\exp\left(\frac{2M}{T_{j+1}-T_j}\right)  \norm{f_{1}^{{[1]}}}_{L^{2}((T_{j-1}, T_{j+1})\times\Omega)}^{2}.
\end{align*}
Using \eqref{eq: rhoF.H1}, we find 
\begin{align*}
&\norm{\left(y^{{[1]}}, h\right)}_{L^{2}((T_{j}, T_{j+1});\mathfrak{D}(\mathbf{A})) \cap H^{1}((T_{j}, T_{j+1});\mathscr{H})}^{2}\\
&\hspace{0.25cm}\leqslant C_8(T) 
\exp\left(\frac{2M}{T_{j+1}-T_j}\right)\frac{\rho_{\mathfrak{F}}^{2}(T_{j-1})}{(T - T_{j+1})^{6}} \Bigg(\norm{\frac{f_{1}}{\rho_{\mathfrak{F}}}}_{L^{2}((T_{j-1}, T_{j+1})\times\Omega)}^{2} + 
\norm{y^{{[2]}}}^{2}_{\mathfrak{E}_{y}(T_{j-1}, T_{j+1})} \Bigg),
\end{align*}
for some $C_8(T)>0$ independent of $j$. 
Using \Cref{lem: weights.facts} we then gather that 
\begin{multline*}
\norm{\left(y^{{[1]}}, h\right)}_{ L^{2}((T_{j}, T_{j+1});\mathfrak{D}(\mathbf{A})) \cap H^{1}((T_{j}, T_{j+1});\mathscr{H})}^{2}  \\
\leqslant C_8(T) \frac{\rho_{0}^2(T_{j+1})}{(T - T_{j+1})^{6}} 
\Bigg(\norm{\frac{f_{1}}{\rho_{\mathfrak{F}}}}_{L^{2}((T_{j-1}, T_{j+1})\times\Omega)}^{2} + 
\norm{y^{{[2]}}}^{2}_{\mathfrak{E}_{y}(T_{j-1}, T_{j+1})} \Bigg).
\end{multline*}
From the above estimate, and the definitions of $\rho_{0}$ and $\rho_{\alpha},$ we infer that 
\begin{align} \label{eq: the.above.estimate.2}
& \norm{\left(\frac{y^{{[1]}}}{\rho_{\alpha}}, \frac{h}{\rho_{\alpha}}\right)}_{L^{2}((T_{j}, T_{j+1});\mathfrak{D}(\mathbf{A})) \cap H^{1}((T_{j}, T_{j+1});\mathscr{H})}^{2} \nonumber \\
& \leqslant \norm{\frac{1}{\rho_{\alpha}}}_{W^{1,\infty}(T_{j}, T_{j+1})}^{2} \norm{\left(y^{[1]}, h\right)}_{ L^{2}((T_{j}, T_{j+1});\mathfrak{D}(\mathbf{A})) \cap H^{1}((T_{j}, T_{j+1});\mathscr{H})}^{2} \nonumber \\
&\leqslant C_8(T) \norm{\frac{1}{\rho_{\alpha}}}_{W^{1,\infty}(T_{j}, T_{j+1})}^{2}   \frac{\rho_{0}^2(T_{j+1})}{(T - T_{j+1})^{6}} 
\Bigg(\norm{\frac{f_{1}}{\rho_{\mathfrak{F}}}}_{L^{2}((T_{j-1}, T_{j+1})\times\Omega)}^{2} + 
\norm{y^{{[2]}}}^{2}_{\mathfrak{E}_{y}(T_{j-1}, T_{j+1})} \Bigg).
\end{align}
Recall that from \Cref{lem: weights.facts},
\begin{equation} \label{eq: the.above.estimate.3}
 \norm{\frac{1}{\rho_{\alpha}}}_{W^{1,\infty}(T_{j}, T_{j+1})}^{2}\frac{\rho_{0}^2(T_{j+1})}{(T - T_{j+1})^{6}}  \leqslant C_{9}(T),
\end{equation}
for some positive constant $C_{9}(T)$ which does not depend on $j.$ 
Combining \eqref{eq: the.above.estimate.2} and \eqref{eq: the.above.estimate.3} together with \eqref{est-control-k0}, \eqref{est-trace-new} and \eqref{est-lin-k} (for $j=0$), we deduce that 
\begin{align} \label{eq: the.above.estimate.4}
&\norm{\left(\frac{y^{{[1]}}}{\rho_{\alpha}}, \frac{h}{\rho_{\alpha}}\right)}_{ L^{2}((0, T);\mathfrak{D}(\mathbf{A})) \cap H^{1}((0, T);\mathscr{H})}^{2}\nonumber\\
&\hspace{2cm}\leqslant C_{10}(T) \Bigg(  \norm{\left(y^{{[1]}, 0}, h^{0}\right)}^{2}_{\mathscr I}  + \norm{\left(\frac{f_{1}}{\rho_{\mathfrak{F}}}, \frac{f_{2}}{\rho_{\mathfrak{F}}}, \frac{f_{3}}{\rho_{\mathfrak{F}}}\right)}_{\mathfrak{E}_{f}(0,T)}^{2} \Bigg).
\end{align}

{\bf Step 5.} \emph{Controllability and weighted-estimate for true state.}
As $y=y^{[1]}+\rho_{\mathfrak{F}}y^{{[2]}}$, we may conclude the proof of \eqref{est:main-nh-source} by combining \eqref{eq: the.above.estimate.4} and \eqref{est:y-sharp-weight}. Null-controllability for $(y,h)$ then follows from \eqref{est:main-nh-source} since $\rho_{\alpha}(T)=0$.
\end{proof}

\subsection{Change of variables} \label{sec: fixing.domain}
We now describe a  simple change of variables which allows us to pass from \eqref{eq: stefan.surface.tension}, set in the moving domain $\Omega(t)$, to a nonlinear problem in the time-independent reference domain $\Omega$ defined in \eqref{eq: def.omega} (and used throughout).  

Given $(c, d) \subset (-1, 1),$ we select an arbitrary $d_{0} \in (d, 1),$ and we  assume that 
\begin{equation} \label{no-contact-1}
1 + h(t, x_{1}) \geqslant d_{0} \hspace{1cm} \mbox{ for all } \quad (t,x_1)\in [0, T]\times\T.  
\end{equation}
Consider a cut-off function $\chi\in C_{c}^{\infty}(-2,2)$ such that 
\begin{equation*}
\chi(x_{2}) = 1 \quad \mbox{ for } x_{2} \in [d_{0}, 1], \quad \mbox{ and } \quad  \chi(x_{2}) = 0 \quad \mbox{ for } x_{2} \in [-1, d].
\end{equation*}
For $t\geqslant0$, we introduce the map $\Psi(t, \cdot): \overline{\Omega}\to\overline{\Omega(t)}$ by 
\begin{equation*}
\Psi(t,x) :=  \Big(x_{1}, x_{2}\left(1 + \chi(x_{2}) h\left(t,x_{1}\right)\right) \Big), \hspace{1cm} x = (x_{1}, x_{2}) \in  \overline{\Omega}.
\end{equation*}
Observe that, whenever $h$ is sufficiently regular, $\Psi(t,\cdot)$ is a $C^{1}$-diffeomorphism from $\overline{\Omega}$ to $\overline{\Omega(t)}.$ In this case, we denote by 
$X(t, \cdot) = [\Psi(t, \cdot)]^{-1}$ the inverse of $\Psi(t,\cdot)$ for all $t\geqslant0$. We consider the following change of coordinates: 
\begin{equation*}
y(t,x) = \varrho(t, \Psi(t,x)) \hspace{1cm} \text{ for } (t,x) \in (0,T)\times \Omega.
\end{equation*}
In other words, 
\begin{equation*}
\varrho(t, z) = y(t, X(t,z)) \hspace{1cm} \text{ for } (t,z) \in (0,T)\times \Omega(t).
\end{equation*}
We also introduce the standard notation 
\begin{equation*}
\mathfrak{B}_{\Psi} := \mathrm{Cof}(\nabla_x \Psi), \qquad \text{ and } \qquad \mathfrak{A}_{\Psi} := \frac{1}{\mathrm{det}(\nabla_x \Psi)} \mathfrak{B}_{\Psi}^{\top}  \mathfrak{B}_{\Psi},
\end{equation*}
where $\delta_{\Psi} := \mathrm{det}(\nabla_x \Psi)$ denotes the Jacobian determinant of $\nabla_x \Psi$, and $\mathrm{Cof}(M)$ denotes the cofactor matrix of $M$, satisfying $M (\mathrm{Cof}(M))^\top = (\mathrm{Cof}(M))^\top M = \mathrm{det}(M)\text{Id}$. System \eqref{eq: stefan.surface.tension} can then be equivalently rewritten as 
\begin{equation} \label{eq: stefan.transformed}
\begin{cases}
\partial_{t} y - \Delta y = \mathcal{N}_{1}(y,h) & \mbox{ in } (0,T) \times \Omega, \\
\partial_{t} h(t,x_1) =  \left(\nabla_x y(t,x_1,1) \cdot \mathbf{e}_{2} \right)+ \left(\mathcal{N}_{2}\big(y(t,x_1,1)\big)\cdot \mathbf{e}_{2} \right) &\text{ on } (0,T)\times\T, \\
y(t,x_1,-1) = 0 &\text{ on } (0,T) \times\T, \\
y(t,x_1,1) = \sigma \del_{x_1}^2 h(t,x_1) + \mathcal{N}_{3}(h(t,x_1)) &\text{ on } (0,T)\times\T, \\
\left(y, h\right)_{|_{t=0}} = \left(y^{0}, h^0\right)  & \text{ in } \Omega \times \T,
\end{cases}
\end{equation}
where $y^0(\cdot) := \varrho^0(\Psi(0,\cdot))$, $\*e_2:=(0,1)^\top$, with the nonlinear terms having the form
\begin{align*}
\mathcal{N}_{1}(y,h):= - (\mathrm{det}(\nabla_x \Psi) - 1) \partial_{t} y + \mathrm{div}(\mathcal{N}_{2}), \hspace{0.5cm} \mathcal{N}_{2}(y):= \left(\mathfrak{A}_{\Psi} - \text{Id}\right) \nabla_x y.
\end{align*}
and
\begin{equation*}
\mathcal{N}_{3}(h):= \sigma \left(\kappa(h) - \del_{x_1}^2 h \right).
\end{equation*}
Using the above change of variables, our main result, \Cref{thm:main-nonlinear}, can equivalently be rephrased as 

\begin{theorem} \label{thm:main-nonlinear-2}
Suppose $T>0$ and $\sigma>0$ are fixed, and suppose that $\omega=\mathbb{T}\times(c,d)$, where $(c,d)\subset(-1,1)$. There exists some small enough $r>0$ such that for every initial data $\left(y^{0}, h^0\right) \in H^{1}(\Omega)\times H^{\sfrac52}(\mathbb{T})$ satisfying 
\begin{align*}
h_{0}(x_{1}) > -1 &\quad\mbox{ for }  x_1\in\mathbb{T}, \notag\\
y^{0}(x_{1}, -1) = 0  &\quad\mbox{ for } x_1\in\mathbb{T}, 
\\ 
y^{0}(x_{1}, 1) + \sigma \kappa(h^{0}(x_{1})) = 0 &\quad\mbox{ for } x_1\in\mathbb{T}, \notag
\end{align*}
and 
\begin{equation*} 
\norm{y^0}_{H^{1}(\Omega)} + \norm{h^{0}}_{H^{\sfrac52}(\mathbb{T})}\leqslant r, 
\end{equation*}
there exists some control $u \in L^{2}((0, T)\times\omega)$ such that the unique solution\footnote{Recall the definitions of the high-order energy spaces $\mathfrak{E}_y$ and $\mathfrak{E}_h$ in \eqref{eq: Ey} and \eqref{eq: Eh} respectively.} $(y,h)\in\mathfrak{E}_{y}(0,T) \times \mathfrak{E}_{h}(0,T)$ to \eqref{eq: stefan.transformed} satisfies $y(T, \cdot)\equiv0$ in $\Omega$ and $h(T, \cdot)\equiv0$ on $\mathbb{T}$. Moreover, $h(t, x_{1})>d$ for all $(t,x)\in[0, T]\times\mathbb{T}.$
\end{theorem}

Thus, our final goal is to prove \Cref{thm:main-nonlinear-2}. This is done by a Banach fixed point argument in what follows.

}

\subsection{Fixed point argument}

{\color{black}	
For $r>0$ we define the ball $\mathfrak{F}_{r}\subset\mathfrak{F}$ centered at $0$ with radius $r$: 
\begin{equation*}
\mathfrak{F}_{r} = \left\{ (f_{1}, f_{2}, f_{3})\in\mathfrak{F} \,\,\Biggm|\,\ \norm{\left(\frac{f_{1}}{\rho_{\mathfrak{F}}}, \frac{f_{2}}{\rho_{\mathfrak{F}}}, \frac{f_{3}}{\rho_{\mathfrak{F}}}\right)}_{\mathfrak{E}_{f}} \leqslant r \right\}.
\end{equation*} 
Our first goal is to estimate the nonlinear terms in \eqref{eq: stefan.transformed}. 

\begin{proposition}[Quadratic nonlinearities] 
\label{prop-nl-estimate}
There exists some $r\in(0, 1)$ such that  for any $\left(y^{0}, h^{0}\right) \in \mathscr{I}$ with
\begin{equation*} 
r=\norm{\left(y^{0}, h^{0}\right)}_{\mathscr{I}},
\end{equation*}
and for any $(f_{1}, f_{2}, f_{3}) \in \mathfrak{F}_{r}$ satisfying \eqref{cc1}, the controlled trajectory $(y, h)$ for the system \eqref{sys-lin} constructed in \Cref{thm: source.term.method} satisfies 
\begin{equation*}
\norm{\frac{\mathcal{N}_{1}(y,h)}{\rho_{\mathfrak F}}}_{L^{2}((0, T)\times\Omega)}  + \norm{\frac{\mathcal{N}_{2}\big(y(\cdot,\cdot,1)\big)\cdot \mathbf{e}_{2}  }{\rho_{\mathfrak F}}}_{ H^{\frac{1}{2}, \frac{1}{4}}((0,T) \times \T)} +  \norm{\frac{\mathcal{N}_{3}(h)}{\rho_{\mathfrak F}}}_{ H^{\frac{3}{2},\frac{3}{4}}((0,T) \times \T)}   \leqslant C r^{2}
\end{equation*}
for some constant $C=C(T,\sigma)>0$ independent of $r$.
\end{proposition}
}

\begin{proof}
{\color{black}
Note that, from \eqref{est:main-nh-source}, we have 
\begin{equation} \label{est-yh-1}
\norm{\left(\frac{y}{\rho_{\alpha} }, \frac{h}{\rho_{\alpha} }\right)}_{ \mathfrak{E}_{y}(0,T) \times \mathfrak{E}_{h}(0,T)} \lesssim_{T,\sigma} r.
\end{equation}
Using the definition of $\rho_{\alpha}$ in \eqref{eq: rho.alpha.def}, and \eqref{est-yh-1}, we obtain 
\begin{equation} \label{est-h}
\norm{h}_{\mathfrak{E}_{h}(0,T)} \lesssim_{T,\sigma} \norm{\frac{h}{\rho_{\alpha}}}_{\mathfrak{E}_{h}(0,T)} \norm{\rho_{\alpha}}_{W^{2,\infty}(0,T)} \lesssim_{T,\sigma} r.
\end{equation}
Thus, for $r$ small enough, $h$ satisfies \eqref{no-contact-1} for all $t \in [0,T].$
Then, from the definition of $\Psi,$ and setting $\zeta(x_{2}) := \chi(x_{2})  +x_{2}\chi'(x_{2})$ it is readily seen that 
\begin{equation*}
\nabla_x \Psi(t,x) = \begin{bmatrix}
1 & 0 \\ x_{2} \partial_{x_{1}} h(t,x_1)  \chi(x_{2}) & 1+ \zeta(x_{2}) h(t,x_1)
\end{bmatrix}, 
\end{equation*}
as well as
\begin{equation*}
\mathrm{det}(\nabla \Psi - 1) =\zeta(x_{2}) h(t,x_1), 
\end{equation*}
and 
\begin{equation*}
\mathfrak{A}_{\Psi} - \text{Id} =  
\frac{1}{\mathrm{det}(\nabla_x \Psi)}
\begin{bmatrix}
\zeta^{2} h^{2} + \zeta h + x_{2}^{2} \chi^{2} (\partial_{x_{1}} h)^{2} & - x_{2} \chi \partial_{x_{1}} h \\
- x_{2} \chi \partial_{x_{1}} h & - \zeta h 
\end{bmatrix}
  \end{equation*}
and
\begin{equation*} 
\mathfrak{A}_{\Psi}(t, x_{1}, 1) = \frac{1}{1+h(t,x_1)}
\begin{bmatrix}
(1+h(t,x_1))^{2} + (\partial_{x_{1}} h(t,x_1))^{2} & -\partial_{x_{1}} h(t,x_1) \\ -\partial_{x_{1}} h(t,x_1) & 1
\end{bmatrix}.
\end{equation*}
Consequently, from \eqref{est-yh-1} and \eqref{est-h} we gather that
\begin{align}
\norm{\frac{1}{\rho_{\alpha}} \left( \nabla_x \Psi - \text{Id} \right)}_{L^{\infty}((0, T) \times \Omega)}&\lesssim_{T,\sigma}r, \nonumber\\ 
\norm{\frac{1}{\rho_{\alpha}} \left(  \mathrm{det}(\nabla_x \Psi) - 1\right)}_{L^{\infty}((0, T) \times \Omega)}&\lesssim_{T,\sigma} r,  \notag \\
\norm{\frac{1}{\mathrm{det}(\nabla_x \Psi)}}_{L^{\infty}((0,T);H^{\sfrac{3}{2}}(\Omega))} &\lesssim_{T,\sigma} r,\nonumber\\
\norm{\frac{1}{\rho_{\alpha}} \left(  \mathfrak{A}_{\Psi} - \text{Id} \right) }_{L^{\infty}((0,T);H^{\sfrac{3}{2}}(\Omega))}&\lesssim_{T,\sigma}r. \label{est-psi-all}
\end{align}
Note that, in the above estimates we have used the fact that $H^{\sfrac{3}{2}}(\Omega)$ is an algebra in dimension $d=2$. 
}{\color{black}
Next, using well-known trace results (see \cite{lions2012non}) and \eqref{est-yh-1}, we deduce that 
\begin{equation} \label{est-trace-nh}
\norm{\frac{\partial_{x_{j}} y(\cdot, \cdot, 1)}{\rho_{\alpha}}}_{ H^{\sfrac{1}{2}, \sfrac{1}{4}}((0,T) \times \T)}  \lesssim_{T,\sigma} r, \qquad j=1,2.
\end{equation}
We are now in a position to estimate the nonlinear terms. To estimate the product terms we will use \cite[Proposition B.1]{GS91}. 
Combining \Cref{lem: weights.facts}, \eqref{est-yh-1}, \eqref{est-h} and the bounds in \eqref{est-psi-all}, we estimate $\mathcal{N}_{1}(y,h)$ as follows: 
\begin{align*}
&\norm{\frac{\mathcal{N}_{1}(y,h)}{\rho_{\mathfrak F}}}_{L^{2}((0, T)\times\Omega)}\\
&\hspace{1.75cm}\lesssim_{T,\sigma}\Bigg(\norm{\frac{1}{\rho_{\alpha}} \left(  \mathrm{det}(\nabla_x \Psi) - 1\right)}_{L^{\infty}((0, T) \times \Omega)}  \norm{\frac{\partial_{t} y}{\rho_{\alpha}}}_{L^{2}((0, T)\times\Omega)} 
\norm{\frac{\rho_{\alpha}^{2}}{\rho_{\mathfrak F}}}_{L^{\infty}(0, T)} \\
&\hspace{2.65cm}+\norm{\frac{1}{\rho_{\alpha}} \left(  \mathfrak{A}_{\Psi} - \text{Id} \right) }_{L^{\infty}((0,T);H^{\sfrac{3}{2}}(\Omega))}  \norm{\frac{\nabla_x y}{\rho_{\alpha}}}_{L^{2}((0, T); H^{1}(\Omega))}
 \norm{\frac{\rho_{\alpha}^{2}}{\rho_{\mathfrak F}}}_{L^{\infty}(0, T)}\Bigg) \\
 &\hspace{1.75cm}\lesssim_{T,\sigma} r^{2}.
\end{align*}
Using the definition of $\Psi$, we note that  
\begin{align}
\norm{\frac{\mathcal{N}_{2}\big(y(\cdot,\cdot,1)\big)\cdot \mathbf{e}_{2}  }{\rho_{\mathfrak F}}}_{ H^{\sfrac{1}{2}, \sfrac{1}{4}}((0,T) \times \T)}&\lesssim_{T,\sigma}\Bigg(\norm{\left(\frac{\partial_{x_{1}} h}{1+ h}\right)\left(\frac{\partial_{x_{1}} y(t, \cdot, 1)}{\rho_{\mathfrak F}}\right)}_{H^{\sfrac{1}{2}, \sfrac{1}{4}}((0,T) \times \T)} \nonumber \\
&\hspace{0.9cm}+ \norm{\left(\frac{ h}{1+ h}\right)\left(\frac{\partial_{x_{2}} y(t, \cdot, 1)}{\rho_{\mathfrak F}}\right)}_{H^{\sfrac{1}{2}, \sfrac{1}{4}}((0,T) \times \T)}\Bigg). \label{n2-trace-1}
\end{align}
Combing \Cref{lem: weights.facts}, \eqref{est-h} and \eqref{est-trace-nh} then yields 
\begin{align*}
&\norm{\left(\frac{\partial_{x_{1}} h}{1+ h}\right)\left(\frac{\partial_{x_{1}} y(t, \cdot, 1)}{\rho_{\mathfrak F}}\right)}_{L^{2}(0, T; H^{\sfrac{1}{2}}(\mathbb{T}))}\\
&\lesssim_{T,\sigma} \norm{\frac{1}{1+h}}_{L^{\infty}(0, T; H^{1}(\T))} \norm{\frac{\partial_{x_{1}} h}{\rho_{\alpha}}}_{L^{\infty}(0, T; H^{1}(\T))}\norm{\frac{\partial_{x_{1}} y(t, \cdot, 1)}{\rho_{\mathfrak F}}}_{L^{2}(0, T; H^{\sfrac{1}{2}}(\mathbb{T}))} \norm{\frac{\rho_{\alpha}^{2}}{\rho_{\mathfrak F}}}_{L^{\infty}(0, T)}\\
&\lesssim_{T,\sigma} r^{2}, 
\end{align*}
and 
\begin{align*}
&\norm{\left(\frac{\partial_{x_{1}} h}{1+ h}\right)\left(\frac{\partial_{x_{1}} y(t, \cdot, 1)}{\rho_{\mathfrak F}}\right)}_{H^{\sfrac{1}{4}}(0,T; L^{2}(\mathbb{T}))}\\
&\lesssim_{T,\sigma} \norm{\frac{1}{1+h}}_{H^{1}((0, T)\times\T)} \norm{\frac{\partial_{x_{1}} h}{\rho_{\alpha}}}_{H^{\sfrac{3}{4}}(0, T; H^{1}(\T))} \norm{\frac{\partial_{x_{1}} y(t, \cdot, 1)}{\rho_{\mathfrak F}}}_{H^{\sfrac{1}{4}}(0, T; L^{2}(\mathbb{T}))} \norm{\frac{\rho_{\alpha}^{2}}{\rho_{\mathfrak F}}}_{W^{1, \infty}(0, T)}\\
&\lesssim_{T,\sigma} r^{2}.
\end{align*}
Arguing in a similar manner, we can show that 
\begin{equation*}
\norm{\left(\frac{ h}{1+ h}\right)\left(\frac{\partial_{x_{2}} y(t, \cdot, 1)}{\rho_{\mathfrak F}}\right)}_{H^{\sfrac{1}{2}, \sfrac{1}{4}}((0,T) \times \T)}  \lesssim_{T,\sigma} r^{2}.
\end{equation*}
Combining the preceding three estimates together with \eqref{n2-trace-1}, we obtain 
\begin{equation*}
\norm{\frac{\mathcal{N}_{2}\big(y(\cdot,\cdot,1)\big)\cdot \mathbf{e}_{2}  }{\rho_{\mathfrak F}}}_{ H^{\sfrac{1}{2}, \sfrac{1}{4}}((0,T) \times \T)} \lesssim_{T,\sigma} r^{2}.
\end{equation*}
To estimate $\mathcal{N}_{3}(h)$, we first note that 
\begin{equation*}
\left(1+|\partial_{x_{1}} h|^{2}\right)^{-\frac{3}{2}}-1 = - \frac{3}{2} \left|\partial_{x_{1}} h \right|^{2}  \left(1 + s |\partial_{x_{1}} h|^{2}\right)^{-\frac{5}{2}}
\end{equation*}
for some $s\in(0,1)$.
Using Sobolev embeddings,  \eqref{est-yh-1} and \eqref{est-h}, we find
\begin{equation*}
\norm{ \frac{\partial_{x_{1}} h}{\rho_{\alpha}} }_{L^{\infty}((0,T);H^{\sfrac{3}{2}}(\T)) \cap H^{\sfrac{3}{4}}((0, T); L^{\infty}(\T)) } + \norm{\partial_{x_{1}} h}_{L^{\infty}((0,T);H^{\sfrac{3}{2}}(\T)) \cap H^{\sfrac{3}{4}}((0, T); L^{\infty}(\T)) } \lesssim_{T,\sigma}  r.
\end{equation*}
Since $L^{\infty}((0,T);H^{\sfrac{3}{2}}(\T))$ and $H^{\sfrac{3}{4}}((0, T); L^{\infty}(\T))$ are both algebras, we also know that
\begin{equation*}
\norm{ \left(1 + s |\partial_{x_{1}} h|^{2}\right)^{-\sfrac{5}{2}}}_{{L^{\infty}((0,T);H^{\sfrac{3}{2}}(\T))} \cap H^{\sfrac{3}{4}}((0, T);L^{\infty}(\T))} \leqslant C,
\end{equation*}
for some $C=C(T,s)>0$ independent of $r.$
Using the above estimates, together with \Cref{lem: weights.facts}, \eqref{est-yh-1}, \eqref{est-h}, we estimate $\mathcal{N}_{3}$ as follows  
\begin{align*}
\norm{\frac{\mathcal{N}_{3}(h)}{\rho_{\mathfrak F}}}_{L^{2}((0,T);H^{\sfrac{3}{2}}(\T)} & \lesssim_{T,\sigma} \norm{ \frac{|\partial_{x_{1}} h|^{2}}{\rho_{\alpha}} }_{L^{\infty}((0,T);H^{\sfrac{3}{2}}(\T))} \norm{\frac{\partial_{x_{1}}^{2} h}{\rho_{\alpha}}}_{L^{2}((0,T);H^{\sfrac{3}{2}}(\T))} \\
&\quad\cdot\norm{ \left(1 + s |\partial_{x_{1}} h|^{2}\right)^{-\sfrac{5}{2}}}_{{L^{\infty}((0,T);H^{\sfrac{3}{2}}(\T))} } \norm{\frac{\rho_{\alpha}^{2}}{\rho_{\mathfrak F}}}_{L^{\infty}(0, T)}\\
&\lesssim_{T,\sigma} r^{2}, 
\end{align*}
and 
\begin{align*}
\norm{\frac{\mathcal{N}_{3}(h)}{\rho_{\mathfrak F}}}_{H^{\sfrac{3}{4}}((0,T); L^{2}(\T))} & \lesssim_{T,\sigma} \norm{ \frac{|\partial_{x_{1}} h|^{2}}{\rho_{\alpha}} }_{H^{\sfrac{3}{4}}((0,T); L^{\infty}(\T))} \norm{\frac{\partial_{x_{1}}^{2} h}{\rho_{\alpha}}}_{H^{\sfrac{3}{4}}((0,T);L^{2}(\T))} \\
&\quad\cdot\norm{ \left(1 + s |\partial_{x_{1}} h|^{2}\right)^{-\sfrac{5}{2}}}_{{H^{\sfrac {3}{4}}((0,T);L^{\infty}(\T))} } \norm{\frac{\rho_{\alpha}^{2}}{\rho_{\mathfrak F}}}_{W^{1,\infty}(0, T)}\\
&\lesssim_{T,\sigma} r^{2}.
\end{align*}
This concludes the proof. 
}
\end{proof}

{\color{black}
In an exactly analogous manner, we can also show the following Lipschitz estimates of the nonlinear terms.  We omit the proof. 

\begin{proposition}[Lipschitz bounds for nonlinearities] 
\label{prop-lip-estimate}
There exists some $r\in(0, 1)$ such that  for any $\left(y^{0}, h^{0}\right) \in \mathscr{I}$ with
\begin{equation*} 
r=\norm{\left(y^{0}, h^{0}\right)}_{\mathscr{I}},
\end{equation*}
and for any $(f_{1}^{i}, f_{2}^{i}, f_{3}^{i}) \in \mathfrak{F}_{r},$ with $i =1,2,$ satisfying \eqref{cc1}, the controlled trajectory $(y_{i}, h_{i})$ for the system \eqref{sys-lin} constructed in \Cref{thm: source.term.method} satisfies 
\begin{multline*}
\norm{\frac{\mathcal{N}_{1}(y_{1},h_{1})}{\rho_{\mathfrak F}} - \frac{\mathcal{N}_{1}(y_{2},h_{2})}{\rho_{\mathfrak F}}}_{L^{2}((0, T)\times\Omega)} 
 \norm{\frac{\mathcal{N}_{3}(h_{1})}{\rho_{\mathfrak F}} - \frac{\mathcal{N}_{3}(h_{2})}{\rho_{\mathfrak F}}}_{ H^{\sfrac{3}{2},\sfrac{3}{4}}((0,T) \times \T)}  \\
 + \norm{\frac{\mathcal{N}_{2}\big(y_{1}(\cdot,\cdot,1)\big)\cdot \mathbf{e}_{2}  }{\rho_{\mathfrak F}} - \frac{\mathcal{N}_{2}\big(y_{2}(\cdot,\cdot,1)\big)\cdot \mathbf{e}_{2}  }{\rho_{\mathfrak F}}}_{ H^{\sfrac{1}{2}, \sfrac{1}{4}}((0,T) \times \T)}\\
 \leqslant C r 
  \norm{\left(\frac{f_{1}^{1}}{\rho_{\mathfrak{F}}}, \frac{f_{2}^{1}}{\rho_{\mathfrak{F}}}, \frac{f_{3}^{1}}{\rho_{\mathfrak{F}}}\right) - \left(\frac{f_{1}^{2}}{\rho_{\mathfrak{F}}}, \frac{f_{2}^{2}}{\rho_{\mathfrak{F}}}, \frac{f_{3}^{2}}{\rho_{\mathfrak{F}}}\right) }_{\mathfrak{E}_{f}(0,T)}
\end{multline*}
for some constant $C=C(T,\sigma)>0$ independent of $r$.
\end{proposition}

We are now in position to prove \Cref{thm:main-nonlinear-2}.

\begin{proof}[Proof of \Cref{thm:main-nonlinear-2}]
We use a Banach fixed point argument.
Let us take 
\begin{equation*} 
r=\norm{\left(y^{0}, h^{0}\right)}_{\mathscr{I}}.
\end{equation*}
Consider the map 
\begin{equation*}
\mathfrak{N} : \mathfrak{F}_{r} \mapsto \mathfrak{F}_{r}, \qquad  \left( f_{1}, f_{2}, f_{3} \right) \mapsto  \left(\mathcal{N}_{1}(y,h), \mathcal{N}_{2}\big(y(t,x_1,1)\big)\cdot \mathbf{e}_{2}, \mathcal{N}_{3}(h) \right),
\end{equation*}
where $(y, h)$ is the controlled trajectory constructed in \Cref{thm: source.term.method}. According to \Cref{prop-nl-estimate} and \Cref{prop-lip-estimate}, there exists some $r > 0$ small enough such that 
 $\mathfrak{N}$ is a strict contraction on $\mathfrak{F}_{r}$. Finally, by \Cref{thm: source.term.method} we obtain  
 $$ \left(\frac{y}{\rho_{\alpha} }, \frac{h}{\rho_{\alpha} }\right) \in \mathfrak{E}_{y}(0,T) \times \mathfrak{E}_{h}(0,T),$$
and the null-controllability for $(y,h)$ then follows since $\rho_{\alpha}(T)=0$. This concludes the proof. 
\end{proof}

}

		\section{Epilogue} \label{sec: conclusion} We have shown that the  Stefan problem with surface tension (Gibbs-Thomson correction) is locally null-controllable, in the sense that both the temperature and the height function are controllable to zero under smallness assumptions on the initial data.
		
		There are, however, several questions and problems which we believe merit further attention and clarity, even in the linearized regime. In addition to addressing the case $\sigma=0$, discussed in greater depth in \Cref{sec: discussion}, other problems may include
		  
		\begin{enumerate}
		
		\item{\bf{Control on the free boundary.}} While not exactly perfectly clear to interpret in physical terms, one can consider the problem of controlling \emph{through} the free boundary, which would mean putting a control $u1_\omega$ in the second equation in \eqref{sys-lin-con.intro} (namely the evolution equation for $h$). This would be in the spirit of works in control of water waves (\cite{Al18}), and also the simplified piston problem (\cite{tucsnak_burgers_jmpa}). We expect this problem to be significantly more challenging to address than the one we had considered here.
		
		\smallskip
		\item{\bf{Spectral optimization.}} The governing operator $\*A$ of \eqref{sys-lin-con.intro} appears somewhat opaque, and a variety of alternative (control) problems can be envisaged and studied regarding \eqref{sys-lin-con.intro}. We believe that further analysis is warranted in the analysis of the spectral properties of $\*A$, in particular with regard to extension of various control results to general geometries (beyond strips). 
		
		As we have noted, $-\*A$ is self-adjoint when considered on the space $\mathscr{H}$ with functions of zero mean on $\mathbb{T}$. In particular, one could define the first eigenvalue $\lambda_1>0$ of $-\*A$ through the min-max theorem as the Rayleigh quotient 
		\begin{equation*}
		\lambda_1:=\inf_{f\in\mathfrak{D}(\*A)\setminus\{0\}}\frac{\langle -\*Af,f\rangle_{\mathscr{H}}}{\|f\|_{\mathscr{H}}^2}.
		\end{equation*}
		But as is typical for the Laplacian, one seeks to use the symmetry and obtain a more tractable representation. In \cite[Lemma 4.5]{hadvzic2010stability} (see also \cite[Chapter 3, Section 4, Lemma 4.5]{hadzic2010stability}), it also is stated that
		\begin{equation} \label{eq: rayleigh}
		\lambda_1=\inf_{f=(f_1,f_2)\in\mathcal{S}\setminus\{0\}}\frac{\displaystyle\int_\Omega |\nabla f_1|^2\diff x}{\displaystyle\int_\Omega |f_1|^2\diff x+\sigma\int_{\T}|\del_{x_1}f_2|^2\diff x_1},
		\end{equation} 
		where the space $\mathcal{S}$ is defined as
		\begin{equation*}
		\mathcal{S}:=\left\{f=(f_1,f_2)\in H^1(\Omega)\times H^{\sfrac{5}{2}}(\T)\,\Biggm|\,\int_{\Omega} f_1(x_1,\cdot)\diff x_1+\int_{\T} f_2(x_1)\diff x_1=0\right\}.
		\end{equation*}
		Taking stock of \eqref{eq: rayleigh}, there are a variety of different spectral optimization problems one could then envisage for \eqref{sys-lin-con.intro}, such as characterizing optimal actuator and observer domains $\omega$ in the spirit of \cite{privat2015optimal, privat2016optimal, privat2017actuator, geshkovski2021optimal}, and in particular, comparing how these designs differ from that of the classical heat equation, or the limit of these designs as $\sigma\searrow0$ (should controllability hold for the latter).
		
		\smallskip
		\item{\bf{The obstacle problem.}} It is by now well-known that the classical Stefan problem ($\sigma=0$), without source terms, is related to the \emph{parabolic obstacle problem} through the so-called \emph{Duvaut transform} (see \cite{figalli2019regularity, ros2018obstacle} and the references therein). For control purposes, one could envisage transferring results from the Stefan problem to the parabolic obstacle problem (which is actually a problem to be studied in its own right, \cite{ros2018obstacle, geshkovskiobstacle}). But this is highly nontrivial due to the fact that the Duvaut transform applies to non-negative solutions of the Stefan problem, and it is not clear if existing techniques on controllability under positivity constraints (\cite{loheac2017minimal}, or the so-called staircase method \cite{pighin2018controllability, ruiz2020control, mazari2020constrained}) would be applicable here. The bottom line is that the controllability properties of the parabolic obstacle problem remain widely open. (See \cite[Section 1.5.1]{geshkovski2021control}.)
		\end{enumerate}
		
\subsection*{Acknowledgments}
This research began with discussions between both authors during the “VIII Partial Differential Equations, Optimal Design and Numerics” workshop that took place in Benasque, in August 2019. We thank the organizing committee, and the center "Pedro Pascual" for their hospitality.
We thank the reviewers for their scrutiny, which has greatly improved the results of this work.

\subsection*{Funding}
B.G. has received funding from the European Union's Horizon 2020 research and innovation programme under the Marie Sklodowska-Curie grant agreement No.765579-ConFlex. Debayan Maity was partially supported by INSPIRE faculty fellowship (IFA18-MA128) and by Department of  Atomic Energy, Government of India, under project no. 12-R \& D-TFR-5.01-0520. 

\begin{appendix}

\section{Toolkit} \label{sec: toolkit}

{\color{black}
		
		\begin{lemma} \label{lem: convergence.series}
		The identity  
		\begin{equation*}
		\sum_{j=1}^{+\infty} \partial_{x_2}\phi_j(1)\lambda_j^{-1}\langle 1, \phi_j\rangle_{L^2(c,d)} = \frac14(c^2-d^2).
		\end{equation*}
		holds true. In particular, the series (appearing in the denominator in \eqref{eq: v0}) is absolutely convergent, and non-zero when $c\neq\pm d$.
		\end{lemma}
		
		\begin{proof}[Proof of \Cref{lem: convergence.series}]
		The series is clearly absolutely convergent, and it is not difficult to see that
		\begin{equation*}
		\sum_{j=1}^{+\infty} \frac{\partial_{x_2}\phi_j(1)}{\lambda_j}\langle 1, \phi_j\rangle_{L^2(c,d)}=\sum_{j=1}^{+\infty} \frac{\cos(j\pi (c+1))-\cos(j\pi (d+1))}{\pi^2j^2} =: f(c+1)-f(d+1),
		\end{equation*}
		where $f(x) = \sum_{j=1}^{+\infty} \frac{\cos(j\pi x)}{\pi^2 j^2}$. 
		In fact, 
		\begin{equation} \label{eq: f.Li}
		f(x) = \frac{\mathrm{Li}_2(e^{i\pi x})-\mathrm{Li}_2(e^{-i\pi x})}{2\pi^2},
		\end{equation}
		where $\mathrm{Li}_n(x) = \sum_{j=1}^{+\infty} \frac{x^j}{j^n}$ is the polylogarithm function. 
		Using well-known identities for this special function, for $x\in(0,2)$ we find
		\begin{align*}
		f'(x) = i\frac{e^{i\pi x}\mathrm{Li}_2'(e^{i\pi x})-e^{-i\pi x}\mathrm{Li}_2'(e^{-i\pi x})}{2\pi}&= \frac{i}{2\pi}(\mathrm{Li}_1(e^{i\pi x})-\mathrm{Li}_1(e^{-i\pi x})) \\
		&=-\frac{i}{2\pi}\left(\mathrm{Log}(1-e^{i\pi x})-\mathrm{Log}(1-e^{-i\pi x})\right),
		\end{align*}
		where $\mathrm{Log}:\mathbb{C}^*\to\mathbb{C}$ is the principal value of the complex logarithm. Then formal computations yield
		\begin{align*}
		\mathrm{Log}(1-e^{i\pi x})-\mathrm{Log}(1-e^{-i\pi x})&=\mathrm{Log}\left(\frac{e^{i\frac{\pi}{2}\left(e^{-i\frac{\pi}{2}x}-e^{i\frac{\pi}{2}x}\right)}}{e^{-i\frac{\pi}{2} }\left(e^{i\frac{\pi}{2}x}-e^{-i\frac{\pi}{2}x}\right)}\right)\\
		&=\mathrm{Log}(-e^{i\pi x})\\
		&=\mathrm{Log}(-1)+\mathrm{Log}(e^{i\pi x}) =\pi i(x-1),
		\end{align*}
		for $x\in(0,2)$.
		(In the last step, we have chosen $\mathrm{Log}(-1)=-i\pi$.)
		Hence, formally,
		\begin{align} \label{eq: no.change.branch}
		f'(x) = -\frac{i}{2\pi}\left(\mathrm{Log}(1-e^{i\pi x})-\mathrm{Log}(1-e^{-i\pi x})\right)= \frac{x-1}{2}.
		\end{align}
		These identities are all only formal due to the possibility that the complex logarithm changes a branch. Such a change would only amount to adding an integer multiple (independent of $x$) of $i\pi$ to the final result. To show that this multiple is $0$ (and thus no change of branch happens), we simply observe that the identity in \eqref{eq: no.change.branch} exactly holds for $x=1$, hence any such constant multiple would have to be $0$. 
		Integrating \eqref{eq: no.change.branch}, we find
		\begin{equation*}
		f(x) = \frac{(x-1)^2}{4}+\kappa
		\end{equation*}
		for some $\kappa\in\mathbb{R}$. 
		(Using the form of $f$ in \eqref{eq: f.Li} and known results from classical analysis, it can be shown that $\kappa=-\frac{1}{12}$.) 
		Whence 
		\begin{equation*}
		\sum_{j=1}^{+\infty} \frac{\partial_{x_2}\phi_j(1)}{\lambda_j}\langle 1, \phi_j\rangle_{L^2(c,d)}= \frac{c^2}{4} - \frac{d^2}{4}.
		\end{equation*}
		This yields the desired conclusion.
		\end{proof}
		}

\section{Numerics} \label{sec: numerics}

\subsection{Discretizing \eqref{sys-lin-con.intro}}
As we did not find precisely the same scheme in the literature, for completeness and future reproducibility purposes, let us briefly discuss the numerical discretization we used for computing, and obtaining the simulations presented in \Cref{fig: hum.1} and \Cref{fig: hum.2}. 

\begin{enumerate}
\item {\bf Setup.}
We shall focus on $\Omega=(0,2)\times(-1,1)$ for simplicity. The control domain $\omega$ is a thin neighborhood of a line ranging from $x_1=\frac12$ to $x_1=\frac32$, tilted at an angle of $45^\circ$ (i.e. with slope $1$) with respect to the horizontal axis (chosen for the relative simplicity of numerical implementation, and the sparsity of the resulting matrix).
We choose the same mesh-size $\bigtriangleup x>0$ for the horizontal and vertical variables, and we set $n_x:=\frac{2}{\bigtriangleup{x}}-1$. Since we are working with periodic boundary conditions in the $x_1$-variable, and Dirichlet boundary conditions in the $x_2$-variables, $y(t,x_1,x_2)$ will be unknown at $(n_x+1)n_x$ points, and $h(t,x_1)$ at $(n_x+1)$ points.
\smallskip 
\item {\bf Finite difference semi-discretization.} We define an equi-distributed grid $\left\{x_1^i, x_2^j\right\}_{\substack{i\in\{0, \ldots, n_x+1\}\\j\in\{0,\ldots,n_x+1\}}}$ of $\Omega$ through 
\begin{equation*}
x_1^i=i{\bigtriangleup{x}} \hspace{1cm} \text{ and } \hspace{1cm} x_2^j=-1+j{\bigtriangleup{x}}.
\end{equation*}
We discretize the two-dimensional Laplacian $\Delta_{x_1, x_2}$ with the classical $5$-point finite-difference stencil, and the Neumann trace $\partial_{x_2}y(t,x_1,1)$ with a centered difference scheme.
Henceforth denoting 
\begin{equation*}
y_{i,j}(t):=y\left(t,x_1^i, x_2^j\right),
\end{equation*} 
with analog definitions for $u_{i,j}$ and $h_i$, the finite-difference semi-discretization of \eqref{sys-lin-con.intro} reads as 
\begin{equation} \label{eq: finite.diff.stefan}
\begin{dcases}
\dot{y}_{i,j}-\frac{y_{i+1,j}+y_{i-1,j}+y_{i,j+1}+y_{i,j-1}-4y_{i,j}}{\bigtriangleup{x}^2}=u_{i,j}1_{\omega_{n_x}} &\{1,\ldots,n_x+1\}\times\{1,\ldots, n_x\},\\
\dot{h}_i=\frac{y_{i,n_x+1}-y_{i,n_x-1}}{2{\bigtriangleup{x}}} &\{1,\ldots,n_x+1\},\\
y_{0,j}=y_{n_x+1,j} &\{1,\ldots,n_x\},\\
y_{i,0}=0 &\{1,\ldots,n_x+1\},\\
y_{i,n_x+1}=\sigma\frac{h_{i+1}+h_{i-1}-2h_i}{\bigtriangleup{x}^2} &\{1,\ldots,n_x+1\},
\end{dcases}
\end{equation}
for $t\in(0,T)$. With analog definitions for $u_{[j]}$ and $\bf{h}$, setting
\begin{equation*}
y_{[j]}:=
\begin{bmatrix}
y_{1,j}\\
y_{2,j}\\
\vdots\\
y_{n_x+1,j}
\end{bmatrix}
\end{equation*}
for $j\in\{1,\ldots,n_x\}$, as well as ${\bf{z}}:=(y_{[1]},\ldots,y_{[n_x]},{\bf{h}})$ and then, similarly, setting ${\bf{u}}:=(u_{[1]},\ldots,u_{[n_x]},0_{\mathbb{R}^{n_x+1}})$, we may rewrite \eqref{eq: finite.diff.stefan} as a canonical finite-dimensional linear system $\dot{\bf{z}}={\bf{A}}_{n_x}{\bf{z}}+{\bf{B}}_{n_x}{\bf{u}}$, where
\begin{equation*}
{\bf{A}}_{n_x}:=
\sbox0{$\begin{matrix}A_0&A_1&&\\A_1&A_0&A_1&\\&\ddots&\ddots&\ddots&\\ &&A_1&A_0&A_1\\&&&A_1&A_0\end{matrix}$}
\sbox1{$\begin{matrix}  \\  \\  \\  \\[0.5em] A_2\end{matrix}$}
\sbox2{$\begin{matrix}  &  &  & \hspace{2.5em}A_3 &\end{matrix}$}
\left[
\begin{array}{c|c}
\usebox{0}& \usebox{1}\\
\hline
 \usebox{2}& A_4
\end{array}
\right]
\end{equation*}
where $A_1=\frac{1}{{\bigtriangleup{x}^2}}\text{Id}_{n_x+1}$ and $A_3=-\frac{1}{2{\bigtriangleup{x}}}\text{Id}_{n_x+1}$, whereas
\begin{equation*}
A_0=
\frac{1}{{\bigtriangleup{x}^2}}\begin{bmatrix}
-4&1& & & 1\\
1&-4&1& & \\
&\ddots&\ddots&\ddots& \\
 & &1&-4 &1 \\
1    &    &     & 1&-4    
\end{bmatrix}
\quad
\text{ and }
\quad
A_\ell=
c_\ell\begin{bmatrix}
-2&1& & & 1\\
1&-2&1& & \\
&\ddots&\ddots&\ddots& \\
 & &1&-2 &1 \\
1    &    &     & 1&-2    
\end{bmatrix},
\end{equation*}
for $\ell\in\{2,4\}$, with $c_2=\frac{\sigma}{{\bigtriangleup{x}}^4}$ and $c_4=\frac{2\sigma}{{\bigtriangleup{x}}^3}$.
\begin{figure}[h!]
\includegraphics[scale=0.5]{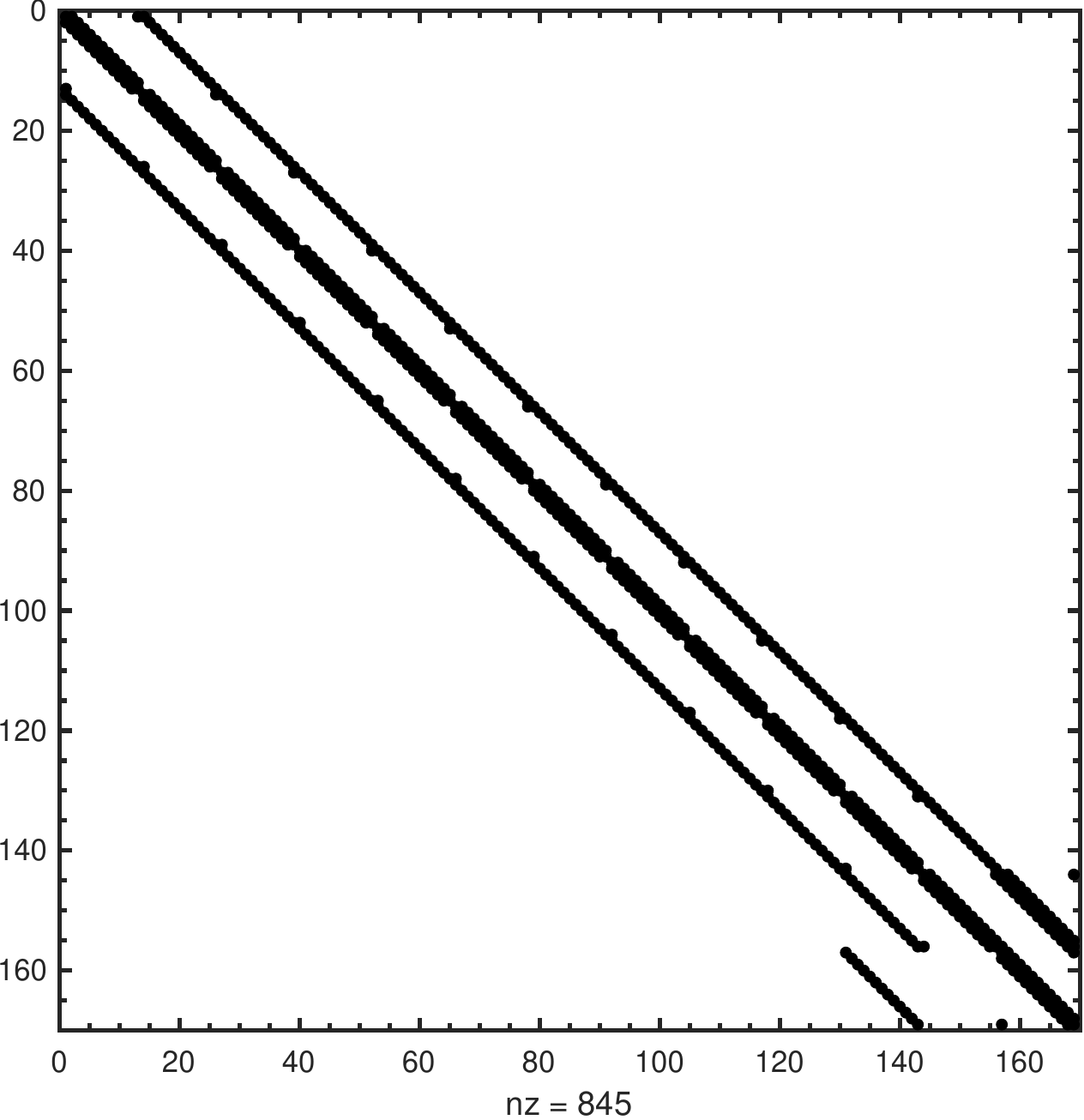}
\hspace{0.5cm}
\includegraphics[scale=0.5]{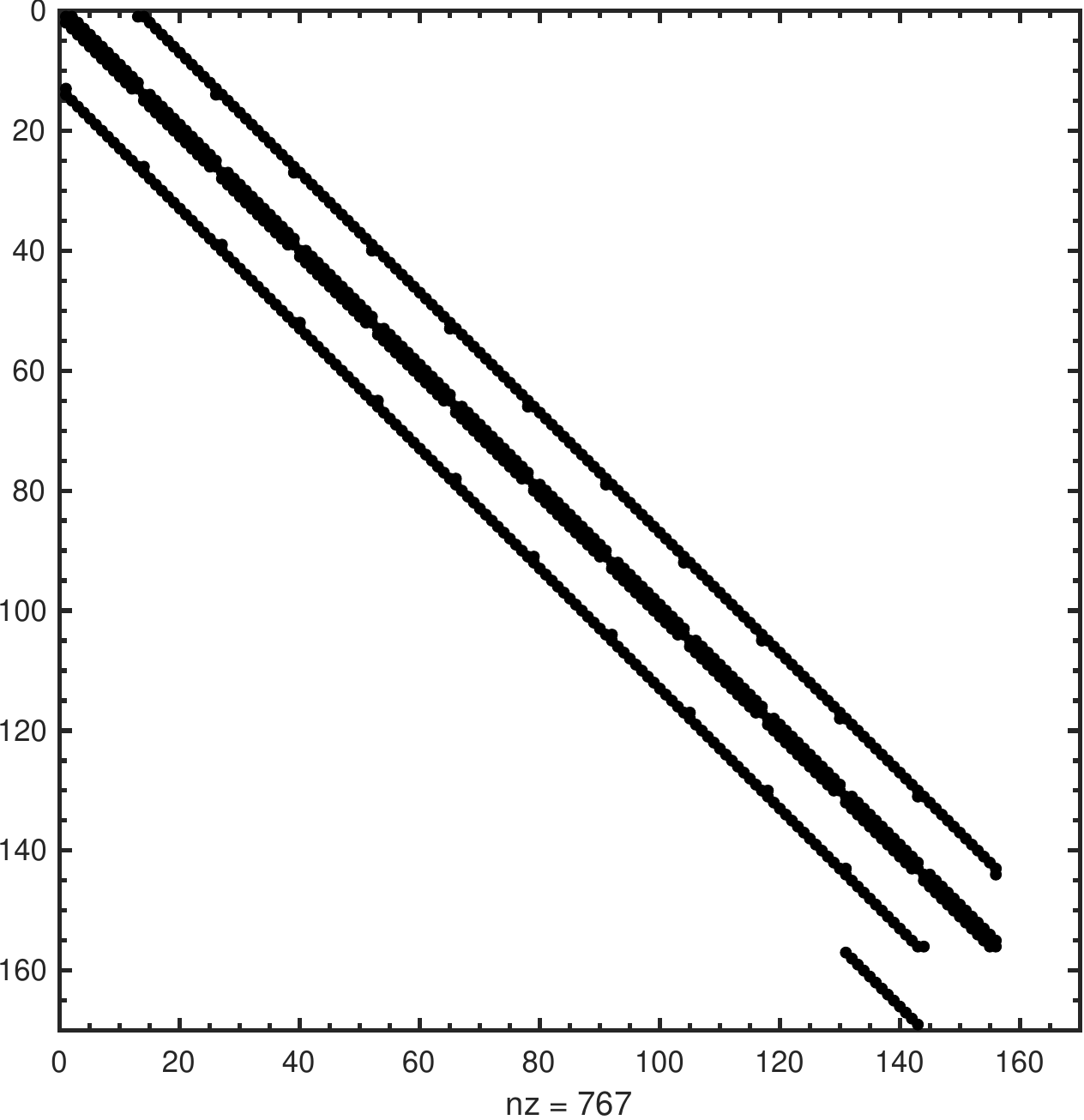}
\caption{The sparsity pattern of the matrix ${\bf A}_{n_x}$ when $\sigma=10$ (\emph{left}) and $\sigma=0$ (\emph{right}).}
\end{figure}
\smallskip
\item {\bf Time-stepping.} Given nodes $\{t_k\}_{k\in\{0,\ldots,n_t\}}$ with $t_k=k{\bigtriangleup{t}}$ for some ${\bigtriangleup t}=\frac{T}{n_t}$, we set ${\bf z}^k={\bf z}(t_k)$. Time-stepping for \eqref{sys-lin-con.intro} is done with a Crank-Nicolson method
\begin{equation*}
\frac{{\bf z}^{k+1}-{\bf z}^k}{{\bigtriangleup t}} = \frac12{\bf A}_{n_x}\left({\bf z}^k+{\bf z}^{k+1}\right) + \frac12 {\bf B}_{n_x}\left({\bf{u}}^k+{\bf{u}}^{k+1}\right)
\end{equation*}
for $k\in\{0,\ldots,n_t\}$, which is unconditionally stable.
\end{enumerate} 

\subsection{Computing} We solve
\begin{equation*}
\min_{\substack{u\in L^2((0,T)\times\omega)\\ (y,h) \text{ solve } \eqref{sys-lin-con.intro} \\ (y(T),h(T))\equiv0}} \|u\|_{L^2((0,T)\times\omega)}^2
\end{equation*}
by making use of the above discretization for parametrizing the PDE constraints.
This is a convex program, which can be solved using interior point methods. We use the IpOpt solver embedded into \texttt{Casadi} for Matlab (\cite{andersson2019casadi}). 
For the experiments in \Cref{fig: hum.1} and \Cref{fig: hum.2}, we took $\bigtriangleup x = \frac{2}{13}$, and thus $n_x=12$, with $n_t=200$. 
All codes are openly made available at \href{https://github.com/borjanG/2022-stefan-control}{{\color{dukeblue}\texttt{https://github.com/borjanG/2022-stefan-control}}}.

\end{appendix}

\bibliographystyle{acm} 
\bibliography{Stefan}    

\bigskip

\begin{minipage}[t]{.5\textwidth}
  {\bf Borjan Geshkovski}\par
  Department of Mathematics\par
  Massachusetts Institute of Technology\par
  Simons Building, Room 246C\par
  77 Massachusetts Avenue\par
  Cambridge\par
  MA\par
  02139-4307 USA\par
  e-mail: \href{mailto:borjan@mit.edu}{\textcolor{dukeblue}{\texttt{borjan@mit.edu}}}
\end{minipage}
\begin{minipage}[t]{.5\textwidth}
  {\bf Debayan Maity}\par
  TIFR Centre for Applicable Mathematics\par
  560065 Bangalore\par
  Karnataka, India\par 
  e-mail:  \href{mailto:debayan.maity@cnrs.fr}{\textcolor{dukeblue}{\texttt{debayan@tifrbng.res.in}}}
\end{minipage} 

\end{document}